%% file: tightness-generic.tex
\documentclass[11pt,letterpaper]{amsart}
\usepackage[foot]{amsaddr}

\usepackage{mathtools}
\usepackage{amsmath}
\usepackage[T1]{fontenc}
\usepackage[latin9]{inputenc}
\usepackage{geometry}
\usepackage{wrapfig}
\usepackage{multicol}
\usepackage{graphicx}
\usepackage{soul}
\usepackage{xcolor}
\usepackage{amssymb}
\usepackage{placeins}
\usepackage{bbm}
\usepackage{cite}
\usepackage{caption}
\usepackage{enumerate}
\usepackage{afterpage}
\usepackage{enumitem}
\usepackage{bmpsize}
\usepackage[hidelinks]{hyperref}
\usepackage{tabu}
\usepackage{enumitem}
\numberwithin{equation}{section}
\usepackage{stmaryrd}
\usepackage{tikz}
\usetikzlibrary{matrix,graphs,arrows,positioning,calc,decorations.markings,shapes.symbols}

\makeatletter
\renewcommand{\email}[2][]{%
  \ifx\emails\@empty\relax\else{\g@addto@macro\emails{,\space}}\fi%
  \@ifnotempty{#1}{\g@addto@macro\emails{\textrm{(#1)}\space}}%
  \g@addto@macro\emails{#2}%
}
\makeatother

\makeatletter
\def\@tocline#1#2#3#4#5#6#7{\relax
	\ifnum #1>\c@tocdepth 
	\else
	\par \addpenalty\@secpenalty\addvspace{#2}%
	\begingroup \hyphenpenalty\@M
	\@ifempty{#4}{%
		\@tempdima\csname r@tocindent\number#1\endcsname\relax
	}{%
		\@tempdima#4\relax
	}%
	\parindent\z@ \leftskip#3\relax \advance\leftskip\@tempdima\relax
	\rightskip\@pnumwidth plus4em \parfillskip-\@pnumwidth
	#5\leavevmode\hskip-\@tempdima
	\ifcase #1
	\or\or \hskip 1em \or \hskip 2em \else \hskip 3em \fi%
	#6\nobreak\relax
	\hfill\hbox to\@pnumwidth{\@tocpagenum{#7}}\par
	\nobreak
	\endgroup
	\fi}
\makeatother

\newtheorem{theorem}{Theorem}[section]
\newtheorem{lemma}[theorem]{Lemma}
\newtheorem{proposition}[theorem]{Proposition}

\newtheorem{corollary}[theorem]{Corollary}
\newtheorem{manualtheoreminner}{Assumption}
\newenvironment{assumption}[1]{%
	\renewcommand\themanualtheoreminner{#1}%
	\manualtheoreminner
}{\endmanualtheoreminner}
\newtheorem{manualtheoreminnerr}{Example}
\newenvironment{example}[1]{%
	\renewcommand\themanualtheoreminnerr{#1}%
	\manualtheoreminnerr
}{\endmanualtheoreminnerr}
{ \theoremstyle{definition}
\newtheorem{definition}[theorem]{Definition}}
{ \theoremstyle{remark}
\newtheorem{remark}[theorem]{Remark}}

\newcommand{\ex}{\mathbb{E}}

\newcommand{\pr}{\mathbb{P}}

\newcommand{\weyl}{W^\circ}

\title{Tightness of discrete Gibbsian line ensembles}
\author{Christian Serio}

\begin{document}

\maketitle

\begin{abstract}
	A discrete Gibbsian line ensemble $\mathfrak{L} = (L_1,\dots,L_N)$ consists of $N$ independent random walks on the integers conditioned not to cross one another, i.e., $L_1 \geq \cdots \geq L_N$. In this paper we provide sufficient conditions for convergence of a sequence of suitably scaled discrete Gibbsian line ensembles $f^N = (f_1^N,\dots,f_N^N)$ as the number of curves $N$ tends to infinity. Assuming log-concavity and a KMT-type coupling for the random walk jump distribution, we prove that under mild control of the one-point marginals of the top curves with a global parabolic shift, the full sequence $(f^N)$ is tight in the topology of uniform convergence over compact sets, and moreover any weak subsequential limit possesses the Brownian Gibbs property. If in addition the top curves converge in finite-dimensional distributions to the parabolic $\mathrm{Airy}_2$ process, we show that $(f^N)$ converges to the parabolically shifted Airy line ensemble. This generalizes the results of \cite{Ber} for Bernoulli line ensembles to a broad class of discrete jump distributions, including geometric as well as any distribution whose support forms a compact integer interval.
\end{abstract}

\tableofcontents

\input{section1.tex}
\input{section2.tex}

\input{section3.tex}
\input{section4.tex}
\input{section5.tex}
\input{section6.tex}

\input{appendix.tex}

\bibliographystyle{alpha}
\bibliography{bib}

\end{document}

%% file: section1.tex
%
\section{Introduction}\label{Section1}

\subsection{Gibbsian line ensembles}\label{Section1.1}

	In recent years, there has been a growing interest in certain models of avoiding random walks known as \textit{line ensembles} satisfying a \textit{Gibbs property}. Put simply, a line ensemble is a collection of random real-valued continuous curves indexed by a set of integers and defined on a common interval on the real line. If $k$ is a positive integer and $\mathfrak{L} = (L_1,\dots,L_k)$ is such a line ensemble defined on an interval $\Lambda\subset\mathbb{R}$, we say that $\mathfrak{L}$ satisfies the \textit{Brownian Gibbs property} if the curves of $\mathfrak{L}$ are non-intersecting, i.e., $L_1 > \cdots > L_k$, and if the following resampling invariance condition holds. We choose a subinterval $\{k_1,k_1+1,\dots,k_2\}$ of the index set $\{1,\dots,k\}$ and a subinterval $[a,b]\subset\Lambda$. Then the law of the curves $(L_{k_1},\dots,L_{k_2})$ restricted to $[a,b]$, conditioned on their values at $a$ and $b$ and the values of the bounding curves $L_{k_1-1}$ and $L_{k_2+1}$ on $[a,b]$, is simply the law of $k_2-k_1+1$ Brownian bridges with the appropriate values at $a$ and $b$ conditioned to avoid one another and the two bounding curves on $[a,b]$.
	
	An example of particular interest is the \textit{Airy line ensemble}, a central object in the \textit{Kardar-Parisi-Zhang}\! (\textit{KPZ})\! \textit{universality class}. This is a stationary process with curves indexed by $\mathbb{N}$, whose top curve has finite-dimensional distributions given by the $\mathrm{Airy}_2$ \textit{process}, with one-point marginals given by the \textit{Tracy-Widom distribution}. The Brownian Gibbs property was first introduced in \cite{CH14} and used to prove the existence of the Airy line ensemble. In particular, the authors of that paper showed that the \textit{parabolic Airy line ensemble}, obtained from the Airy line ensemble by subtracting a parabola, satisfies the Brownian Gibbs property. This paper initiated the use of the Gibbs property as a powerful tool to prove regularity properties of line ensembles, including tightness under appropriate scalings, and through this technique the authors exhibited the parabolic Airy line ensemble as an edge scaling limit of Dyson Brownian motion.
	
	More recently, there has been significant interest in scaling limits of discrete Gibbsian line ensembles, in which the curves are obtained from discrete random walks with various jump distributions. This scaling limit is in many cases expected to be the parabolic Airy line ensemble, reminiscent of Donsker's classical theorem establishing Brownian motion as a scaling limit of random walks with generic jump distributions. This conjectural limit has recently been proven in a few particular cases. In the preprint \cite{DNV}, uniform convergence over compact sets to the parabolic Airy line ensemble was proven for sequences of (centered and scaled) non-intersecting Bernoulli, geometric, exponential, and Poisson random walks started from the origin, as the number of curves tends to infinity\textemdash under assumptions of exact formulas for the finite-dimensional distributions of all random walks, as well as a Gibbs property. Another preprint, \cite{Wu19}, studied distinct but related objects known as $(H,H^{\mathrm{RW}})$-\textit{Gibbsian line ensembles}, which consist of discrete random walks with a continuous jump distribution specified by a \textit{random walk Hamiltonian} $H^{\mathrm{RW}}$. The random walks are not required to avoid one another, but rather the \textit{interaction Hamiltonian} $H$ serves as an exponential energy penalty for crossings. These models correspond to a positive temperature regime, as opposed to non-crossing random walks which correspond to zero temperature. Assuming that the top curve of a sequence of such line ensembles converges weakly after a global parabolic shift to a stationary process, in addition to various technical conditions on the Hamiltonians $H$ and $H^{\mathrm{RW}}$, \cite{Wu19} proves tightness of the sequence as well as an $H_1$-Brownian Gibbs property for all subsequential limits. 
	
	Subsequently, an analogous convergence result was proven in \cite{Ber} for sequences of Bernoulli Gibbsian line ensembles with no prescribed initial or terminal data. One of the central goals of that work was to weaken the assumptions used in \cite{DNV}, with regard to both the initial data of the random walks and the control over the distributions of the random walks. Assuming only that the \textit{top} curves of a sequence of Bernoulli Gibbsian line ensembles satisfy a certain uniform one-point tightness condition at integer times with a global parabolic shape, the authors proved tightness of the sequence with respect to uniform convergence over compacts. (We remark that the idea of using control on the top curve alone to prove tightness for the full line ensemble originates in \cite{CH14}.) Moreover, it was shown that any weak subsequential limit of such a sequence satisfies the Brownian Gibbs property. In the prior work \cite{DimMat}, it was shown that the parabolic Airy line ensemble is uniquely characterized by the Brownian Gibbs property and the finite-dimensional distributions of its top curve. Thus, the results of \cite{Ber} show that finite-dimensional convergence of the top curve to the parabolic $\mathrm{Airy}_2$ process implies weak convergence to the parabolic Airy line ensemble.
	
	The primary goal of the present paper is to generalize the results of \cite{DNV} and \cite{Ber} to Gibbsian line ensembles arising from discrete random walks with more general jump distributions. Of particular interest are geometric Gibbsian line ensembles, which arise in connection with models of geometric last passage percolation (LPP) via the Robinson-Schensted-Knuth (RSK) correspondence\textemdash see \cite{J03} and \cite[Section 6]{CLW} for details of this construction. One such model was treated in \cite{DNV}, and convergence to the Airy line ensemble was proven. In the language of Schur processes, the model considered in that paper corresponds to the case of pure-$\alpha$ specializations which are all identical. It would be of interest to study versions of this model in which a finite number of the $\alpha$ parameters are changed, and in this case one may hope to prove Baik-Ben Arous-Peche (BBP) asymptotics rather than Tracy-Widom. With applications of this sort in mind, our goal is to provide a broad result which allows one to prove convergence of a sequence of Gibbsian line ensembles with fairly general discrete jump distributions. Thus, we treat the broader class of $H^{\mathrm{RW}}$-\textit{Gibbsian line ensembles}, consisting of random walks whose jumps are prescribed by a general random walk Hamiltonian function $H^{\mathrm{RW}}$ on the extended integers. We define these objects more precisely in the following section, but we note that $H^{\mathrm{RW}}$ is defined on the integers rather than the real line as in \cite{Wu19}, as we work here with discrete jumps rather than continuous. We assume that the Hamiltonian $H^{\mathrm{RW}}$ is a convex function, which allows us to prove a monotone coupling result, and that the $H^{\mathrm{RW}}$ random walk bridges satisfy a certain strong coupling result studied in \cite{DW19}, similar to the classical Koml\'os-Major-Tusn\'ady (KMT) embedding. The comparison between random walk bridges and Brownian bridges afforded by this strong coupling is a fundamental tool in many of our arguments. We note that these two assumptions are satisfied by Hamiltonians corresponding to many natural distributions, including geometric as well as any log-concave distribution whose support is a compact integer interval; see Section \ref{Section2.3} for more details.
	
	The main accomplishment of this paper is to extend the results proven in \cite{Ber} to line ensembles satisfying the above assumptions on their jump distributions, thus developing a general theory which allows one to prove convergence of a sequence of discrete Gibbsian line ensembles with minimal input from only the top curves. In principle, proving weak convergence requires proving both tightness and finite-dimensional convergence for the entire line ensemble, but our results show that in fact a condition weaker than finite-dimensional convergence for the top curve alone is sufficient for tightness. Namely, if $\mathfrak{L}^N = (L_1^N, \dots, L_N^N)$ is a sequence of $H^{\mathrm{RW}}$-Gibbsian line ensembles with $N$ curves, we show that if the top curves $L_1^N$ (centered and scaled) have tight one-point marginals near integer times that roughly display a global parabolic shape, then the whole sequence $(\mathfrak{L}^N)$ under the same centering and scaling is tight with respect to uniform convergence over compact sets. Moreover, we prove that any subsequential limit satisfies the Brownian Gibbs property. As a consequence, we prove that if $L_1^N$ (centered and scaled) in particular converges in finite-dimensional distributions to the parabolic $\mathrm{Airy}_2$ process, then the full sequence $(\mathfrak{L}^N)$ converges weakly to the parabolic Airy line ensemble. We also note that we assume a weaker one-point tightness condition than that in \cite{Ber} by allowing certain sub-parabolic fluctuations, and consequently our result may be amenable to proving convergence to objects other than the Airy line ensemble, such as the \textit{Airy processes with wanderers} studied in \cite{AFM}. Our results in addition generalize those of \cite{DNV} for integer jump distributions by requiring less control on the top curves and no control on the lower curves, and we believe that the ideas of our arguments can be applied to non-integer jump distributions treated there as well.
	
	Before describing our results in greater detail in the next section, we remark on two more recent works dealing with tightness of Gibbsian line ensembles, \cite{BCD21} and \cite{DW21}. Here the authors studied $(H,H^{\mathrm{RW}})$-Gibbsian line ensembles similar to those in \cite{Wu19}, but with general fixed interaction Hamiltonians $H$, whereas \cite{Wu19} required $H$ to vary along a sequence and converge to the exponential function. The authors prove results on tightness and the Brownian Gibbs property for subsequential limits similar to those of \cite{Ber}, and they provide an application to line ensembles arising from the log-gamma polymer. In particular, \cite{BCD21} deals with tightness of the top curve, while \cite{DW21} proves tightness of the full line ensemble. Together these two works generalize the results of \cite{Ber} to the case of line ensembles with continuous jump distributions which arise in positive temperature models. In comparison, our paper generalizes the results of \cite{Ber} to line ensembles with more general discrete jump distributions, with an outlook towards applications to certain zero temperature models, specifically geometric last passage percolation and its analogues.

\subsection{Main results}\label{Section1.2}

	We begin by giving some definitions necessary for our main result; we state these in more detail in Section \ref{Section2}. For $a,b\in\mathbb{Z}\cup\{\pm\infty\}$ with $a\leq b$ we denote the integer interval $[a,b]\cap\mathbb{Z}$ by $\llbracket a,b\rrbracket$. Given $T_0,T_1\in\mathbb{Z}$ with $T_0\leq T_1$ and $N\in\mathbb{N}$, we define a $\llbracket 1,N\rrbracket$\textit{-indexed} \textit{discrete line ensemble} $\mathfrak{L} = (L_1,\dots,L_N)$ with \textit{entry data} $\vec{x}\in\mathbb{Z}^N$ and \textit{exit data} $\vec{y}\in\mathbb{Z}^N$ to be a collection of $N$ random paths on $\llbracket T_0, T_1\rrbracket$ taking values in $\mathbb{Z}$, such that $L_i(T_0) = x_i$ and $L_i(T_1) = y_i$ for $i\in\llbracket 1,N\rrbracket$. 
	
	Given $\alpha<\beta$, possibly infinite, an $\llbracket\alpha,\beta\rrbracket$-\textit{supported random walk Hamiltonian} is a function $H^{\mathrm{RW}} : \mathbb{Z}\cup\{\pm\infty\} \to [0,\infty]$ which is finite precisely on the set $\llbracket \alpha,\beta\rrbracket$ and satisfies the normalization condition $\sum_{k\in\llbracket\alpha,\beta\rrbracket} \exp(-H^{\mathrm{RW}}(k)) = 1$. We say that $\mathfrak{L}$ above satisfies the $H^{\mathrm{RW}}$\textit{-Gibbs property} if
	\begin{enumerate}[label=(\arabic*)]
		
		\item The curves of $\mathfrak{L}$ are almost surely non-crossing, i.e., $L_1(s) \geq \cdots \geq L_N(s)$ for all $s\in\llbracket T_0,T_1\rrbracket$.
		
		\item For any $\llbracket k_1,k_2\rrbracket \subseteq \llbracket 1,N-1\rrbracket$ and $\llbracket a,b\rrbracket \subseteq\llbracket T_0,T_1\rrbracket$, the law of $(L_{k_1},\dots,L_{k_2})$ on $\llbracket a,b\rrbracket$ conditioned on the values of $\vec{x} := (L_{k_1}(a),\dots,L_{k_2}(a))$, $\vec{y} := (L_{k_1}(b),\dots,L_{k_2}(b))$, and $L_{k_1-1}$ and $L_{k_2+1}$ restricted to $\llbracket a,b\rrbracket $, is that of $k_2-k_1+1$ random walk bridges on $\llbracket a,b\rrbracket$ from $\vec{x}$ to $\vec{y}$, sampled independently with probabilities proportional to 
		\[
		\exp\left[-\sum_{i=a}^{b-1} H^{\mathrm{RW}}(L_k(i+1)-L_k(i))\right], \quad k_1\leq k \leq k_2,
		\]
		and conditioned to not cross one another or $L_{k_1-1}$ or $L_{k_2+1}$ on $\llbracket a,b\rrbracket$.

	\end{enumerate}

	This is a natural extension of the Schur Gibbs property for Bernoulli line ensembles introduced in \cite{Ber}. In particular, the case when $\alpha = 0$, $\beta = 1$, and $H^{\mathrm{RW}}(0) = H^{\mathrm{RW}}(1) = \log 2$ describes the Bernoulli Gibbsian line ensembles studied there. Moreover, this essentially corresponds to the $(H,H^{\mathrm{RW}})$-Gibbs property of \cite{Wu19}, \cite{BCD21}, and \cite{DW21} if the interaction Hamiltonian is taken to be $H = \infty\cdot\mathbf{1}_{(0,\infty)}$. In essence, the $H^{\mathrm{RW}}$-Gibbs property says that a line ensemble is non-crossing and has local distribution given by non-crossing random walk bridges with jumps prescribed by the Hamiltonian $H^{\mathrm{RW}}$. We note that the bottom curve $L_N$ is distinguished in the Gibbs property, in that we do not assume anything about its conditional law. Rather, this curve acts essentially as a bottom boundary for the line ensemble which the other curves are conditioned not to cross. We refer to discrete line ensembles with this Gibbs property as $H^{\mathrm{RW}}$\textit{-Gibbsian line ensembles}.
	
	In the remainder of this section, we fix $\alpha$ and $\beta$ and a random walk Hamiltonian $H^{\mathrm{RW}}$ satisfying the two technical assumptions described in the previous section. Namely, we require $H^{\mathrm{RW}}$ to be a convex function in order to obtain a monotone coupling result, and we assume the following strong coupling condition. For any $p\in(\alpha,\beta)$, there exist constants $C,A,a\in(0,\infty)$ such that for all $n\in\mathbb{N}$ and $z\in\llbracket n\alpha,n\beta\rrbracket$, we can couple a Brownian bridge $B^\sigma$ on $[0,1]$ of some variance $\sigma^2 = \sigma_p^2$ with random paths $\ell^{(n,z)}$ on $\llbracket 0,n\rrbracket$ from 0 to $z$ with jumps distributed according to $\exp(-H^{\mathrm{RW}}(\cdot))$, so that
	\[
	\ex\left[\exp\left(a\sup_{0\leq t\leq n}\left|\sqrt{n}\,B^\sigma_{t/n} +\frac{t}{n}z - \ell^{(n,z)}(t)\right|\right)\right] \leq Ce^{A\log n}e^{|z-pn|^2/n}.
	\]
	To see the utility of this assumption, note that if $|z-pn| = O(\sqrt{n})$, then the right hand side is $O(n^A)$. Chebyshev's inequality then implies that for any $\epsilon > 0$, the supremum on the left-hand side is bounded by $\epsilon\sqrt{n}$ with high probability for sufficiently large $n$. In view of the $1/\sqrt{n}$ scaling of random walks that we use, this provides close comparisons between the random walk bridges $\ell^{(n,z)}$ and a Brownian bridge, for which we have many useful formulas. We state these two assumptions in more detail in Section \ref{Section2.3}, and we give examples of particular Hamiltonians satisfying them.
	
	We fix a sequence of $\llbracket 1,N\rrbracket$-indexed $H^{\mathrm{RW}}$-Gibbsian line ensembles $\mathfrak{L}^N = (L_1^N,\dots,L_N^N)$ on $\llbracket a_N,b_N\rrbracket$, where $a_N \leq 0$ and $b_N\geq 0$ are integers. The results we prove concern scaling limits of the sequence $\mathfrak{L}^N$ as the number of curves $N$ tends to infinity. We first state two assumptions on the $\mathfrak{L}^N$ in terms of parameters $\gamma > 0$, $p\in(\alpha,\beta)$, $\lambda > 0$, and $\theta\in[0,2)$. The first parameter $\gamma$ is related to the fluctuation exponent of the line ensemble; in our setup the line ensemble will fluctuate on order $N^{\gamma/2}$. (In \cite{Ber}, this parameter was denoted by $\alpha$.) The second parameter $p$ can be thought of as the global slope of the random walks, which must lie between $\alpha$ and $\beta$ since the allowed jumps lie within this range. Thirdly, $\lambda$ is related to the global curvature of the line ensemble. Lastly, the parameter $\theta$ is new to this work, and it allows for sub-parabolic fluctuations in the line ensemble.
	
	In sum, our assumptions imply that on large scales the top curve $L_1^N$ roughly approximates the inverted parabola $-\lambda x^2$ with an affine shift of slope $p$. The precise hypotheses are as follows.
	
	\begin{assumption}{1}\label{a1}
		
		There exists a function $\psi : \mathbb{N} \to (0,\infty)$ such that $\psi(N)\to\infty$ as $N\to\infty$ and $a_N < -\psi(N)N^\gamma$, $b_N > \psi(N)N^\gamma$.
		
	\end{assumption}
	This assumption means essentially that in the limit as $N\to\infty$, the domain of definition $[a_N,b_N]$ of $\mathfrak{L}^N$ covers the whole real line on scale $N^\gamma$.
	
	\begin{assumption}{2}\label{a2}
		
		There is a function $\varphi: (0, \infty) \rightarrow (0,\infty)$ such that for any $\epsilon > 0$ we have 
		\[
		\sup_{n\in\mathbb{Z}}\limsup_{N\to\infty}\mathbb{P} \left( \left|N^{-\gamma/2}\left(L_1^N(\lfloor nN^\gamma\rfloor) - p\lfloor nN^\gamma\rfloor \right) + \lambda n^2 \right| \geq \varphi(\epsilon) + |n|^\theta \right) \leq \epsilon.
		\]
	
	\end{assumption}
	Here we use the convention $0^0 = 1$ in the case that $\theta = 0$. Observe that when $n = 0$, Assumption 2 implies that $N^{-\gamma/2}L_1^N(0)$ is a tight sequence, so the fluctuation exponent is $\gamma/2$. The transversal exponent is $\gamma$, twice $\gamma/2$ as expected with Brownian scaling. The assumption implies that after applying an affine shift of slope $p$ and scaling vertically by $N^{\gamma/2}$ and horizontally by $N^\gamma$, the top curve $L_1^N$ globally approximates the inverted parabola $-\lambda x^2$ near integer times, in a uniform way. We include the term $|n|^\theta$ to allow for weaker control on the one-point distributions than was required in \cite{Ber} and other works; we are able to do this since $\theta < 2$, so that for large $n$, $|n|^\theta$ is negligible compared to the parabolic shift $\lambda n^2$. Our reason for allowing these sub-parabolic fluctuations is to make our results potentially applicable to convergence not only to the Airy line ensemble, but also to the Airy processes with wanderers appearing in \cite{AFM}, whose one-point distributions correspond to certain spiked matrix models. Although to our knowledge there is so far no explicit construction in the literature of a Gibbsian line ensemble with these finite-dimensional distributions, it is generally believed that the Airy line ensemble is not the only possible limiting object for sequences of Gibbsian line ensembles. Moreover, it is expected that Airy processes with wanderers satisfy a condition on the one-point marginals analogous to Assumption 2, but this has not yet been proven in the literature. We have formulated Assumption 2 with the outlook of potentially proving convergence of geometric last passage percolation models either to the Airy line ensemble or to Airy processes with wanderers, but we leave this outside the scope of the present paper. We note that \cite{DNV} handles the case of the Airy line ensemble limit, but their argument relies heavily on properties of the Airy kernel and cannot be immediately generalized.

	Before finally stating our main results, we define the precise centering and scaling we will use. We will embed each of our line ensembles into the space $C(\mathbb{N}\times\mathbb{R})$ with the topology of uniform convergence over compact sets and the corresponding Borel $\sigma$-algebra, where $\mathbb{N}\times\mathbb{R}$ has the product topology. For $N\in\mathbb{N}$ we define
	\[
	f_i^N(x) = N^{-\gamma/2}\left(L_i^N(xN^\gamma) - pxN^\gamma\right), \mbox{ for } x\in[-\psi(N),\psi(N)] \mbox{ and } i\in\llbracket 1, N\rrbracket,
	\]
	and we extend $f_i^N$ to all of $\mathbb{R}$ by defining $f_i^N(x) = f_i^N(-\psi(N))$ for $x\leq -\psi(N)$ and $f_i^N(x) = f_i^N(\psi(N))$ for $x\geq\psi(N)$. For $i\geq N+1$, we simply define $f_i^N \equiv 0$. Then the map $(i,x)\mapsto f_i^N(x)$ defines a $C(\mathbb{N}\times\mathbb{R})$-valued random variable, whose law we denote by $\mathbb{P}_N$. Note that the way we choose to extend $f_i^N$ outside of $[-\psi(N),\psi(N)]$ is immaterial since $\psi(N)\to\infty$ as $N\to\infty$ and we are dealing with uniform convergence on compact sets.
	
	We now state the primary result of this paper. We let $\sigma_p^2$ denote the variance of the Brownian bridge afforded by our strong coupling assumption for the particular choice of $p$ in Assumptions \ref{a1} and \ref{a2}. We give the proof of this theorem in Section \ref{Section2.4}.
	
	\begin{theorem}\label{mainthm}
		
		Under Assumptions \ref{a1} and \ref{a2}, $(\mathbb{P}_N)_{N\in\mathbb{N}}$ is a tight sequence of probability measures. Moreover, if $f^\infty$ is any random variable whose law is a weak subsequential limit of $(\mathbb{P}_N)$, then $\mathcal{L}^\infty := \sigma_p^{-1} f^\infty$ satisfies the Brownian Gibbs property $\mathrm{(}$see Section \ref{Section1.1} and Definition \ref{DefBGP}$\mathrm{)}$.
		
	\end{theorem}

	We next aim to strengthen this result to obtain weak convergence of $(\mathbb{P}_N)$ to the Airy line ensemble. We let $(\mathcal{A}_i)_{i\in\mathbb{N}}$ denote the $\mathbb{N}$-indexed Airy line ensemble, as constructed in \cite[Theorem 3.1]{CH14}, and we let $(\mathcal{L}_i^{\mathrm{Ai}})_{i\in\mathbb{N}}$ denote the parabolic Airy line ensemble defined by $\mathcal{L}_i^{\mathrm{Ai}}(x) = 2^{-1/2}\left(\mathcal{A}_i(x) - x^2\right)$. In particular, these are both $C(\mathbb{N}\times\mathbb{R})$-valued random variables, with $\mathcal{A}_1$ the $\mathrm{Airy}_2$ process and $\mathcal{L}_1^{\mathrm{Ai}}$ the parabolic $\mathrm{Airy}_2$ process. It is shown in \cite[Theorem 3.1]{CH14} that $\mathcal{L}^{\mathrm{Ai}}$ possesses the Brownian Gibbs property; the factor of $2^{-1/2}$ is included to ensure that the Brownian bridges related to this property have diffusion parameter 1. In order to obtain convergence of $\sigma_p^{-1} f^N$ to $\mathcal{L}^{\mathrm{Ai}}$, we use the following strengthening of Assumption \ref{a2}.
	
	\begin{assumption}{2'}\label{a2'}
		Let $c = (2\lambda^2/\sigma_p^2)^{1/3}$. Then for any $k\in\mathbb{N}$ and $t_1,\dots,t_k,x_1,\dots,x_k\in\mathbb{R}$, we have
		\[
		\lim_{N\to\infty} \mathbb{P}\left(\,\bigcap_{i=1}^k\left\{\frac{L_1^N(\lfloor t_i N^\gamma\rfloor) - p\lfloor t_i N^\gamma\rfloor}{\sigma_p N^{\gamma/2}} \leq x_i \right\}\right) = \mathbb{P}\left(\,\bigcap_{i=1}^k\left\{c^{-1/2}\mathcal{L}_1^{\mathrm{Ai}}(ct_i)\leq x_i\right\}\right).
		\]
	\end{assumption}

	We observe that Assumption 2' implies Assumption \ref{a2} with $\theta = 0$, therefore for any $\theta\in[0,2)$. Indeed, Assumption \ref{a2'} implies that the sequence $N^{-\gamma/2}(L_1^N(\lfloor xN^\gamma\rfloor ) - p\lfloor xN^\gamma\rfloor) + \lambda x^2$ converges in finite-dimensional distributions to $\sigma(2c)^{-1/2}(\mathcal{A}_1(cx) - c^2x^2) + \lambda x^2$, where $\sigma := \sigma_p$. Since $\sigma(2c)^{-1/2}c^2 = \lambda$, this last random variable is equal to $\sigma(2c)^{-1/2}\mathcal{A}_1(cx)$. Thus, using the fact that $\mathcal{A}_1(x)$ is a stationary process with one-point marginals given by the GUE Tracy-Widom distribution $F_{\mathrm{GUE}}$ (see \cite{TW}), we have for each $n\in\mathbb{Z}$ and $a\geq 0$ that
	\begin{align*}
		&\lim_{N\to\infty}\mathbb{P}\left(\left|N^{-\gamma/2}\left(L_1^N(\lfloor nN^\gamma\rfloor) - p\lfloor nN^\gamma\rfloor\right) + \lambda n^2\right| \geq a\right) = \mathbb{P}\left(|\mathcal{A}_i(cx)| \geq a\sqrt{2c}/\sigma\right)\\
		= \; & 1 - F_{\mathrm{GUE}}\left(a\sqrt{2c}/\sigma\right) + F_{\mathrm{GUE}}\left(-a\sqrt{2c}/\sigma\right).
	\end{align*}
	Given $\epsilon > 0$, we can choose $a$ large enough so that the last line is less than $\epsilon$, and taking this choice of $a$ for $\varphi(\epsilon)$ gives the content of Assumption \ref{a2} with $\theta = 0$.
	
	With the additional strength of Assumption 2', we obtain the following corollary of Theorem \ref{mainthm}, which follows quickly from the uniqueness result of \cite[Theorem 1.1]{DimMat}. The proof is given in Section \ref{Section2.4}.
	
	\begin{theorem}\label{airythm}
		Under Assumptions \ref{a1} and \ref{a2'}, the sequence $\mathcal{L}^N := \sigma_p^{-1} f^N$ converges weakly in the topology of uniform convergence over compact sets to the line ensemble $\mathcal{L}^\infty$ defined by
		\[
		\mathcal{L}_i^\infty(x) = c^{-1/2}\mathcal{L}_i^{\mathrm{Ai}}(cx), \mbox{ for } i\in\mathbb{N},\, x\in\mathbb{R}, \mbox{ where } c = (2\lambda^2/\sigma_p^2)^{1/3}.
		\]
	\end{theorem}

	The general structure of the proofs of Theorems \ref{mainthm} and \ref{airythm} follows that of \cite{Ber}, and most of our auxiliary results are analogues of those appearing in that paper. However, the broader class of random walk Hamiltonians we treat here introduces further complexities. Of greatest importance is the fact that a key component of the arguments in \cite{Ber} consisted of exact computations involving non-crossing Bernoulli random walks. The authors of this paper established a connection between non-crossing Bernoulli random walks and Schur symmetric functions, and then employed asymptotic formulas for these symmetric functions to compute exact limit formulas for the fixed time distribution of a non-crossing Bernoulli line ensemble. These computations cannot be extended to random walks with a generic Hamiltonian $H^{\mathrm{RW}}$, and thus it is a priori not clear how to extend the arguments beyond the case of Bernoulli random walks. An innovation of this paper is to make further use of the KMT coupling assumption\textemdash which is also satisfied by Bernoulli random walks\textemdash to translate the results of \cite[Section 8]{Ber} to the more general setting we deal with here. The central idea is that since Bernoulli random walks may be coupled with Brownian bridges of a certain variance, any other discrete random walk which also satisfies a KMT coupling assumption can effectively be coupled with Bernoulli random walks with a different parameter, and this allows us to conclude convergence of the fixed time distribution of general non-crossing $H^{\mathrm{RW}}$-line ensembles. The detailed argument may be found in Section \ref{Section3.3}. This observation allows us to avoid the lengthy technical computations performed in \cite{Ber}, and we are able to use the outputs of these efforts to make relatively simple arguments that do not directly rely on integrability.

%% file: section2.tex
%
\section{Line ensembles} \label{Section2}

In this section we state definitions and basic properties of line ensembles we use throughout the paper, as well as our main results.

%
\subsection{Line ensembles and the Brownian Gibbs property}\label{Section2.1}

In this section, we give formal definitions of \textit{line ensembles} and the \textit{Brownian Gibbs property}, a resampling invariance property of which we make great use. The reader may find most of the information in this section in \cite[Section 2.1]{Ber}, but we include it here for the sake of completeness.

Let us first establish some notation we will use throughout the paper. For integers $a\leq b$, we denote by $\llbracket a,b\rrbracket$ the integer interval $[a,b]\cap\mathbb{Z}$. Given an interval $\Lambda\subset\mathbb{R}$ we let $(C(\Lambda),\mathcal{C})$ denote the space of continuous real valued functions on $\Lambda$ with the topology of uniform convergence over compact sets and the Borel $\sigma$-algebra $\mathcal{C}$. For an index set $\Sigma\subset\mathbb{Z}$ equipped with the discrete topology, we endow the product $\Sigma\times\Lambda$ with the product topology. We give $C(\Sigma\times\Lambda)$ the topology of uniform convergence over compacts and the Borel $\sigma$-algebra as above, which we denote by $\mathcal{C}^\Sigma$. We will usually take $\Sigma = \llbracket 1,N\rrbracket$, where $1\leq N \leq \infty$.

\begin{definition}\label{LEdef}
	Let $\Sigma\subset\mathbb{Z}$ and $\Lambda\subset\mathbb{R}$ an interval. A $\Sigma$\textit{-indexed line ensemble} $\mathcal{L}$ is a random variable defined on a probability space $(\Omega,\mathcal{F},\mathbb{P})$ taking values in $(C(\Sigma\times\Lambda),\mathcal{C}^\Sigma)$. We say that $\mathcal{L}$ is \textit{non-intersecting} if $\mathbb{P}$-a.s. we have $\mathcal{L}(i,x) > \mathcal{L}(j,x)$ for all $i,j\in\Sigma$ with $i<j$ and all $x\in\Lambda$.
\end{definition}

\begin{remark}
	Intuitively, $\mathcal{L}$ in the previous definition is a random collection of continuous curves on $\Lambda$ indexed by $\Sigma$. To say that $\mathcal{L}$ is non-intersecting means that the curves remain in a specified order vertically, with lower indexed curves at the top. We often view $\mathcal{L}$ as a function $\Sigma\times\Lambda \to \mathbb{R}$, although it is in fact $\mathcal{L}(\omega)$ for each $\omega\in\Omega$ which is such a function. We write $\mathcal{L}_i = \mathcal{L}(i,\cdot)$ for $i\in\Sigma$ to denote the $i$th curve of $\mathcal{L}$; note that each $\mathcal{L}_i$ is then a $(C(\Lambda),\mathcal{C})$-valued random variable on $(\Omega,\mathcal{F},\mathbb{P})$.
\end{remark}

We now state a result given as Lemma 2.2 in \cite{Ber}, which says that the space $C(\Sigma\times\Lambda)$ with the topology of uniform convergence over compacts is a Polish space. We will use this result frequently in order to apply the Skorohod representation theorem. The proof is elementary and can be found in \cite[Section 7.1]{Ber}.

\begin{lemma}\label{Polish} Let $\Sigma \subset \mathbb{Z}$ and $\Lambda \subset \mathbb{R}$ be an interval. Suppose that $(a_n)_{n\geq 1}, (b_n)_{n\geq 1}$ are sequences of real numbers such that $a_n < b_n$, $[a_n, b_n] \subset \Lambda$, $a_{n+1} \leq a_n$, $b_{n+1} \geq b_n$ and $\bigcup_{n = 1}^\infty [a_n, b_n] = \Lambda$. For $n \in \mathbb{N}$ we let $K_n := \Sigma_n \times [a_n, b_n]$ where $\Sigma_n := \Sigma \cap \llbracket -n, n \rrbracket$. Then there exists a metric $d$ on $C (\Sigma \times \Lambda) $ which induces the topology of uniform convergence over compact sets, and the metric space $(C (\Sigma \times \Lambda), d)$ is complete and separable.
\end{lemma}

Our main results regarding line ensembles will be concerned with tightness. The following gives sufficient conditions for tightness of a sequence of line ensembles, which appear as \cite[Lemma 2.4]{Ber}; we refer to \cite[Section 7.2]{Ber} for a proof.

\begin{lemma}\label{2Tight}
	Let $\Sigma \subset \mathbb{Z}$ and $\Lambda\subset\mathbb{R}$ an interval. Suppose that $(a_n)_{n\geq 1}, (b_n)_{n\geq 1}$ are sequences of real numbers such that $a_n < b_n$, $[a_n, b_n] \subset \Lambda$, $a_{n+1} \leq a_n$, $b_{n+1} \geq b_n$ and $\bigcup_{n = 1}^\infty [a_n, b_n] = \Lambda$. Let $(\mathcal{L}^n)_{n\geq 1}$ be a sequence of line ensembles on $[a_n,b_n]$. Then $(\mathcal{L}^n)$ is tight if and only if for every $i\in\Sigma$ we have
	\begin{enumerate}[label=(\roman*)]
		\item $\lim_{a\to\infty} \limsup_{n\to\infty}\, \pr(|\mathcal{L}^n_i(a_0)|\geq a) = 0$;
		\item For all $\epsilon>0$ and $k \in \mathbb{N}$,  $\lim_{\delta\to 0} \limsup_{n\to\infty}\, \pr\bigg(\sup_{\substack{x,y\in [a_k,b_k], \\ |x-y|\leq\delta}} |\mathcal{L}^n_i(x) - \mathcal{L}^n_i(y)| \geq \epsilon\bigg) = 0.$
		
	\end{enumerate}
\end{lemma}

We now define \textit{Brownian line ensembles} and formulate the Brownian Gibbs property. We first state some useful information about Brownian bridges. If $(W_t)_{0\leq t\leq 1}$ is a standard Wiener process on $[0,1]$, then we refer to the process
\[
\tilde{B}(t) =  W_t - t W_1, \hspace{5mm} 0 \leq t \leq 1,
\]
as a {\em Brownian bridge} (from $\tilde{B}(0) = 0$ to $\tilde{B}(1) = 0 $) with \textit{diffusion parameter} 1 (or more frequently, \textit{variance} 1). For brevity we refer to $(\tilde{B}_t)_{0\leq t\leq 1}$ as a {\em standard Brownian bridge} on $[0,1]$.

Given $a, b,x ,y \in \mathbb{R}$ with $a < b$, we define a random variable on $(C([a,b]), \mathcal{C})$ via
\begin{equation}\label{BBDef}
	B(t) = (b-a)^{1/2} \cdot \tilde{B} \left( \frac{t - a}{b-a} \right) + \left(\frac{b-t}{b-a} \right) \cdot x + \left( \frac{t- a}{b-a}\right) \cdot y, 
\end{equation}
and refer to this random variable as a Brownian bridge on $[a,b]$ (from $B(a) = x$ to $B(b) = y$) with variance $1$. Given $k \in \mathbb{N}$ and $\vec{x}, \vec{y} \in \mathbb{R}^k$ we let $\mathbb{P}^{a,b, \vec{x},\vec{y}}_{\mathrm{free}}$ denote the law of $k$ independent Brownian bridges $B_i$ on $[a,b]$ from $B_i(a) = x_i$ to $B_i(b) = y_i$ all with variance $1$.

We next state two results about Brownian bridges from \cite{CH14} for future use. These are given as Corollary 2.9 and Corollary 2.10 respectively in \cite{CH14}. 

\begin{lemma}\label{NoTouch} Fix a continuous function $f: [0,1] \rightarrow \mathbb{R}$ such that $f(0) > 0$ and $f(1) > 0$. Let $B$ be a standard Brownian bridge and let $C = \{ B(t) > f(t) \mbox{ for some $t \in [0,1]$}\}$ $\mathrm{(}$crossing$\mathrm{)}$ and $T = \{ B(t) = f(t) \mbox{ for some } t\in [0,1]\}$ $\mathrm{(}$touching$\mathrm{)}$. Then $\mathbb{P}(T \cap C^c) = 0.$
\end{lemma}
\begin{lemma}\label{Spread} Let $U$ be an open subset of $C([0,1])$, which contains a function $f$ such that $f(0) = f(1) = 0$. If $B$ is a standard Brownian bridge on $[0,1]$ then $\mathbb{P}(B[0,1] \subset U) > 0$.
\end{lemma}

In simple terms, an $(f,g)$-avoiding Brownian line ensemble, with $f,g$ two continuous functions, is a collection of $k$ independent Brownian bridges, conditioned on not intersecting one another and staying above the graph of $g$ and below the graph of $f$.

\begin{definition}\label{DefAvoidingLaw}
	Let $k \in \mathbb{N}$ and $\weyl_k$ denote the open Weyl chamber in $\mathbb{R}^{k}$, i.e.
	$$\weyl_k = \{ \vec{x} = (x_1, \dots, x_k) \in \mathbb{R}^k: x_1 > x_2 > \cdots > x_k \}.$$
	Let $\vec{x}, \vec{y} \in \weyl_k$, $a,b \in \mathbb{R}$ with $a < b$, and $f: [a,b] \rightarrow (-\infty, \infty]$ and $g: [a,b] \rightarrow [-\infty, \infty)$ be two continuous functions. (Either $f: [a,b] \rightarrow \mathbb{R}$ is continuous or $f = \infty$ everywhere, and similarly for $g$.) Assume that $f(t) > g(t)$ for all $t \in[a,b]$, $f(a) > x_1, f(b) > y_1$ and $g(a) < x_k, g(b) < y_k.$
	
	We define the {\em $(f,g)$-avoiding Brownian line ensemble on $[a,b]$ with entry data $\vec{x}$ and exit data $\vec{y}$} to be the $\llbracket 1,k\rrbracket$-indexed line ensemble $\mathcal{Q}$ on $[a,b]$, with law $\mathbb{P}^{a,b, \vec{x},\vec{y}}_{\mathrm{free}}$ (see below \eqref{BBDef}) conditioned on the avoidance event 
	\[
	E  = \left\{ f(r) > B_1(r) > B_2(r) > \cdots > B_k(r) > g(r) \mbox{ for all $r \in[a,b]$} \right\}.
	\] 
	We denote the probability distribution of $\mathcal{Q}$ as $\mathbb{P}_{\mathrm{avoid}}^{a,b, \vec{x}, \vec{y}, f, g}$ and write $\mathbb{E}_{\mathrm{avoid}}^{a,b, \vec{x}, \vec{y}, f, g}$ for the expectation with respect to this measure.
\end{definition}

\begin{remark}
	We point out that the event $E$ above is an open set of positive measure, and thus the conditional law above is well-defined. Indeed, it is easy to see that we can find $u_1, \dots, u_k \in C([0,1])$ and $\epsilon > 0$ such that 
	\[
	\mathbb{P}(E) \geq \mathbb{P}\left(\max_{1 \leq i \leq k} \sup_{r \in [0,1]}|B_i(r) - u_i(r)| < \epsilon \right) = \prod_{i = 1}^k \mathbb{P} \left( \sup_{r \in [0,1]}|B_i(r) - u_i(r)| < \epsilon  \right),
	\]
	and by Lemma \ref{Spread} the product on the right is positive.
\end{remark}

The following definition gives the notion of the Brownian Gibbs property from \cite{CH14}.

\begin{definition}\label{DefBGP}
	Fix a set $\Sigma = \llbracket 1, N \rrbracket$ and an interval $\Lambda \subset \mathbb{R}$, and let $K = \llbracket k_1,k_2\rrbracket \subset \Sigma$ be finite and $a,b \in \Lambda$ with $a < b$. Set $f = \mathcal{L}_{k_1 - 1}$ and $g = \mathcal{L}_{k_2 + 1}$ (where $f = \infty$ if $k_1 - 1 \not \in \Sigma$ and $g = -\infty$ if $k_2 +1 \not \in \Sigma$). Write $D_{K,a,b} = K \times (a,b)$ and $D_{K,a,b}^c = (\Sigma \times \Lambda) \setminus D_{K,a,b}$. A $\Sigma$-indexed line ensemble $\mathcal{L}$ on $\Lambda$ is said to have the {\em Brownian Gibbs property} if it is non-intersecting and 
	\[
	\mbox{ Law}\left( \mathcal{L}|_{K \times [a,b]} \mbox{ conditional on } \mathcal{L}|_{D^c_{K,a,b}} \right)= \mbox{Law} \left( \mathcal{Q} \right),
	\]
	where $\mathcal{Q}_i = \tilde{\mathcal{Q}}_{i - k_1 + 1}$, with $\tilde{\mathcal{Q}}$ denoting the $(f,g)$-avoiding Brownian line ensemble on $[a,b]$ with entry data $(\mathcal{L}_{k_1}(a), \dots, \mathcal{L}_{k_2}(a))$ and exit data $(\mathcal{L}_{k_1}(b), \dots, \mathcal{L}_{k_2}(b))$ from Definition \ref{DefAvoidingLaw}. 
	
	Equivalently, $\mathcal{L}$ on $\Lambda$ satisfies the Brownian Gibbs property if and only if it is non-intersecting and for any finite $K = \llbracket k_1,k_2\rrbracket \subset \Sigma$ and $[a,b] \subset \Lambda$ and any bounded Borel-measurable function $F: C(K \times [a,b]) \rightarrow \mathbb{R}$, we have $\mathbb{P}$-a.s.
	\begin{equation}\label{BGPTower}
		\mathbb{E} \left[ F\left(\mathcal{L}|_{K \times [a,b]} \right) \,\big|\, \mathcal{F}_{\mathrm{ext}} (K \times (a,b))  \right] =\mathbb{E}_{\mathrm{avoid}}^{a,b, \vec{x}, \vec{y}, f, g} \bigl[ F(\tilde{\mathcal{Q}}) \bigr].
	\end{equation}
	Here,
	\[
	\mathcal{F}_{\mathrm{ext}} (K \times (a,b)) := \sigma \left ( \mathcal{L}_i(s): (i,s) \in D_{K,a,b}^c \right),
	\]
	$\vec{x} = (\mathcal{L}_{k_1}(a), \dots, \mathcal{L}_{k_2}(a))$, $\vec{y} = (\mathcal{L}_{k_1}(b), \dots, \mathcal{L}_{k_2}(b))$, $f = \mathcal{L}_{k_1 - 1}|_{[a,b]}$, and $g = \mathcal{L}_{k_2 + 1}|_{[a,b]}$ ($f = \infty$ if $k_1 - 1 \not \in \Sigma$ and $g =-\infty$ if $k_2 +1 \not \in \Sigma$.)
\end{definition}

\begin{remark}\label{RemMeas} We note briefly that equation \eqref{BGPTower} does in fact make sense, as the right-hand side is $\mathcal{F}_{\mathrm{ext}}(K\times(a,b))$-measurable. This follows from \cite[Lemma 3.4]{DimMat}, which shows that the right side is measurable with respect to the $\sigma$-algebra 
	$$ \sigma \left( \mathcal{L}_i(s) : \mbox{  $i \in K$ and $s \in \{a,b\}$, or $i \in \{k_1 - 1, k_2 +1 \}$ and $s \in [a,b]$} \right).$$
\end{remark}

In this paper we use a modification of the above definition, known as the partial Brownian Gibbs property. This notion was introduced in \cite{DimMat} and used in \cite{Ber}. We explain the difference between the two definitions in Remark \ref{RPBGP}.

\begin{definition}\label{DefPBGP}
	Fix a set $\Sigma = \llbracket 1 , N \rrbracket$ and an interval $\Lambda \subset \mathbb{R}$. A $\Sigma$-indexed line ensemble $\mathcal{L}$ on $\Lambda$ is said to satisfy the {\em partial Brownian Gibbs property} if and only if it is non-intersecting and for any finite $K = \llbracket k_1,k_2\rrbracket \subset \Sigma$ with $k_2 \leq N - 1$, $[a,b] \subset \Lambda$, and any bounded Borel-measurable function $F: C(K \times [a,b]) \rightarrow \mathbb{R}$, we have $\mathbb{P}$-a.s.
	\begin{equation}\label{PBGPTower}
		\mathbb{E} \left[ F(\mathcal{L}|_{K \times [a,b]}) \,\big|\, \mathcal{F}_{\mathrm{ext}} (K \times (a,b))  \right] =\mathbb{E}_{\mathrm{avoid}}^{a,b, \vec{x}, \vec{y}, f, g} \bigl[ F(\tilde{\mathcal{Q}}) \bigr],
	\end{equation}
	with notation as in Definition \ref{DefBGP}. 
\end{definition}

\begin{remark}\label{RPBGP}
	Definition \ref{DefPBGP} is slightly different from the Brownian Gibbs property of Definition~\ref{DefBGP}. If $\Sigma = \mathbb{N}$, the two definitions are equivalent. However, if $\Sigma = \llbracket 1,N\rrbracket$ with $1 \leq N < \infty$, then the partial Brownian Gibbs property is weaker than the Brownian Gibbs property. Specifically, the Brownian Gibbs property allows for the possibility that $k_2 = N$ in Definition \ref{DefPBGP}, in which case the convention is that $g = -\infty$. We prefer to work with the more general partial Brownian Gibbs property, and most of the results later in this paper are stated in its terms.
\end{remark}

%
\subsection{$H^{\mathrm{RW}}$-Gibbsian line ensembles}\label{Section2.2}

In this section we introduce discrete analogues of the line ensembles and Gibbs properties in Section \ref{Section2.1} above. Our exposition mirrors that in \cite[Section 2.2]{Ber}, in which Bernoulli Gibbsian line ensembles were defined. Here we discuss a more general class of discrete Gibbsian line ensembles, which we refer to as $H^{\mathrm{RW}}$\textit{-Gibbsian line ensembles}.

\begin{definition}\label{discLEdef}
	Let $\Sigma\subseteq\mathbb{Z}$ and $T_0,T_1\in\mathbb{Z}$ with $T_0 < T_1$. We call a function $\llbracket T_0, T_1\rrbracket\to\mathbb{Z}$ a \textit{discrete path}. Let $(Y,\mathcal{D})$ denote the space of functions $f : \Sigma \times\llbracket T_0,T_1\rrbracket \to \mathbb{Z}$ (collections of discrete paths), endowed with the discrete topology. A $\Sigma$-\textit{indexed discrete line ensemble} $\mathfrak{L}$ is a $(Y,\mathcal{D})$-valued random variable defined on a probability space $(\Omega,\mathcal{B},\mathbb{P})$.
\end{definition}

\begin{remark}\label{discLErmk}
	A discrete line ensemble $\mathfrak{L}$ is intuitively a collection of random walks on the integer interval $\llbracket T_0,T_1\rrbracket$, indexed by $\Sigma$. We often view $\mathfrak{L}$ as a function $\Sigma\times\llbracket T_0,T_1\rrbracket\to \mathbb{Z}$, although it is in fact $\mathfrak{L}(\omega)$ for each $\omega\in\Omega$ which is such a function. By linearly interpolating between the points $\mathfrak{L}(i,j)$ and $\mathfrak{L}(i,j+1)$ for each $i\in\Sigma$ and $j\in\llbracket T_0, T_1-1\rrbracket$, we view $\mathfrak{L}$ as a collection of random continuous curves on the interval $[T_0,T_1]$, indexed by $\Sigma$. Thus, we view discrete line ensembles as particular examples of the line ensembles of Definition \ref{LEdef}. If $T_0 \leq a < b \leq T_1$, we denote by $\mathfrak{L}\llbracket a, b\rrbracket$ the restricted random variable given by $\mathfrak{L}\llbracket a, b\rrbracket(\omega) = \mathfrak{L}(\omega)|_{\llbracket a, b\rrbracket}$.
\end{remark}

\begin{definition}\label{hamdef}
	Given integers $\alpha,\beta\in\mathbb{Z}\cup\{\pm\infty\}$, we define an $\llbracket \alpha,\beta \rrbracket$-\textit{supported random walk Hamiltonian} to be a function $H^{\mathrm{RW}} : \mathbb{Z}\cup\{\pm\infty\}\to [0,\infty]$ such that $H^{\mathrm{RW}}(k) < \infty$ if and only if $k\in\llbracket \alpha,\beta \rrbracket$, and
	\[
	\sum_{k\in \llbracket \alpha,\beta \rrbracket} \exp\left(-H^{\mathrm{RW}}(k)\right) = 1.
	\]
	For $t_1,t_2,z_1,z_2\in\mathbb{Z}$ with $t_1<t_2$ and $\alpha(t_2-t_1) \leq z_2-z_1 \leq \beta(t_2-t_1)$, we let $\Omega(t_1,t_2,z_1,z_2)$ denote the set of discrete paths $\ell$ starting at $(t_1,z_1)$ and ending at $(t_2,z_2)$, such that $\ell(i+1)-\ell(i)\in\llbracket\alpha,\beta\rrbracket$ for each $i$. We denote by $\mathbb{P}_{H^{\mathrm{RW}}}^{t_1,t_2,z_1,z_2}$ the measure on $\Omega(t_1,t_2,z_1,z_2)$ given by
	\begin{equation}\label{HRWmeasdef}
	\mathbb{P}_{H^{\mathrm{RW}}}^{t_1,t_2,z_1,z_2}(\ell) = Z^{-1}\exp\left[-\sum_{i=t_1}^{t_2-1} H^{\mathrm{RW}}(\ell(i+1)-\ell(i))\right],
	\end{equation}
	where $Z$ is the normalization constant chosen so that this defines a probability measure. Here we use the convention that $\exp(-\infty) = 0$. Note that our assumptions on $z_1,z_2$ and $H^{\mathrm{RW}}$ ensure that $\Omega(t_1,t_2,z_1,z_2)$ is nonempty and this measure is well-defined.
	
	Lastly, given $k\in\mathbb{N}$, $T_0,T_1\in\mathbb{Z}$ with $T_0 < T_1$, and $\vec{x},\vec{y}\in\mathbb{Z}^k$, we let $\mathbb{P}_{H^{\mathrm{RW}}}^{T_0,T_1,\vec{x},\vec{y}}$ denote the law of $k$ independent discrete paths in $\Omega(T_0,T_1,x_i,y_i)$ with laws $\mathbb{P}_{H^{\mathrm{RW}}}^{T_0,T_1,x_i,y_i}$.
\end{definition}

\begin{remark}\label{consec}
	We emphasize that we require $H^{\mathrm{RW}}$ to be finite at each point of its support interval $\llbracket\alpha,\beta\rrbracket$, which is to say that there is a positive probability that the random walk will jump by each amount in the interval. This is to avoid issues of conditioning on events with zero probability when fixing the values $z_1$ and $z_2$ of the random walk at the endpoints. We refer to \cite[Section 2.3]{DW19} for more details on this assumption.
\end{remark}

\begin{definition}\label{HRWavoiddef}
	Let $k \in \mathbb{N}$ and let $\mathfrak{W}_k$ denote the set of signatures of length $k$, i.e.,
	$$\mathfrak{W}_k = \{ \vec{x} = (x_1, \dots, x_k) \in \mathbb{Z}^k: x_1 \geq  \cdots \geq  x_k \}.$$ Note this is distinct from the closed Weyl chamber $\overline{W}_k$ of Definition \ref{DefAvoidingLaw}. Let $\vec{x}, \vec{y} \in \mathfrak{W}_k$, $T_0, T_1 \in \mathbb{Z}$ with $T_0 < T_1$, $S\subseteq\llbracket T_0,T_1\rrbracket$, and let $f: \llbracket T_0, T_1 \rrbracket \rightarrow (-\infty, \infty]$, $g: \llbracket T_0, T_1 \rrbracket \rightarrow [-\infty, \infty)$. Let $H^{\mathrm{RW}}$ be an $\llbracket \alpha,\beta \rrbracket$-supported random walk Hamiltonian.
	
	We define the {\em $(f,g;S)$-non-crossing $H^{\mathrm{RW}}$ line ensemble} on the interval $\llbracket T_0, T_1 \rrbracket$ with {\em entry data $\vec{x}$} and {\em exit data $\vec{y}$} to be the $\llbracket 1,k\rrbracket$-indexed $H^{\mathrm{RW}}$ line ensemble $\mathfrak{Q} = (Q_1,\dots,Q_k)$ on $\llbracket T_0, T_1 \rrbracket$, with law given by $\mathbb{P}^{T_0,T_1, \vec{x},\vec{y}}_{H^{\mathrm{RW}}}$ conditioned on the non-crossing event 
	\[
	A  = \left\{ f(r) \geq Q_1(r) \geq \cdots \geq Q_k(r) \geq g(r) \mbox{ for all $r \in S$} \right\}.
	\]
	This definition is well-posed if there exist $Q_i \in \Omega(T_0,T_1,x_i,y_i)$ for $i \in \llbracket 1, k \rrbracket$ that satisfy the conditions in $A$. The \textit{acceptance probability} is defined to be
	\[
	Z(T_0,T_1,\vec{x},\vec{y},f,g;S) = \mathbb{P}^{T_0,T_1,\vec{x},\vec{y}}_{H^{\mathrm{RW}}}(A).
	\] 
	We will denote by $\Omega_{\mathrm{avoid}}(T_0, T_1, \vec{x}, \vec{y}, f,g; S)$ the set of $\mathfrak{Q}$ satisfying the conditions in $A$. We denote the conditional probability distribution of $\mathfrak{Q}$ by $\mathbb{P}_{\mathrm{avoid}, H^{\mathrm{RW}};S}^{T_0,T_1, \vec{x}, \vec{y}, f, g}$ and write $\mathbb{E}_{\mathrm{avoid}, H^{\mathrm{RW}}; S}^{T_0, T_1, \vec{x}, \vec{y}, f, g}$ for the expectation with respect to this measure. If $S = \llbracket T_0,T_1\rrbracket$, we write $\Omega_{\mathrm{avoid}}(T_0,T_1,\vec{x},\vec{y},f,g)$, $\mathbb{P}_{\mathrm{avoid}, H^{\mathrm{RW}}}^{T_0,T_1, \vec{x}, \vec{y}, f, g}$, $\mathbb{E}_{\mathrm{avoid}, H^{\mathrm{RW}}}^{T_0, T_1, \vec{x}, \vec{y}, f, g}$, and $Z(T_0,T_1,\vec{x},\vec{y},f,g)$. Likewise if $f=+\infty$ and $g=-\infty$, we omit $f$ and $g$ from the notation.
\end{definition}

\begin{definition}\label{DefHRWGP}
	Fix a set $\Sigma = \llbracket 1, N \rrbracket$ with $N \in \mathbb{N}$ or $N = \infty$ and $T_0, T_1\in \mathbb{Z}$ with $T_0 < T_1$, as well as an $\llbracket \alpha,\beta \rrbracket$-supported random walk Hamiltonian $H^{\mathrm{RW}}$. A $\Sigma$-indexed discrete line ensemble $\mathfrak{Q}=(Q_1,\dots,Q_N)$ on $\llbracket T_0, T_1 \rrbracket$ is said to satisfy the {\em $H^{\mathrm{RW}}$-Gibbs property} if it is non-crossing, meaning that 
	$$ Q_j(i) \geq Q_{j+1}(i) \mbox{ for all $j \in \llbracket 1, N-1\rrbracket$ and $i \in \llbracket T_0, T_1 \rrbracket$},$$
	and for any finite $K = \llbracket k_1,k_2\rrbracket \subseteq \llbracket 1, N - 1 \rrbracket$ and $a,b \in \llbracket T_0, T_1 \rrbracket$ with $a < b$, the following holds.  Suppose that $f, g$ are two discrete paths on $\llbracket a,b\rrbracket$ and $\vec{x}, \vec{y} \in \mathfrak{W}_{k_2 - k_1 + 1}$, and define the event 
	\[ 
	E=\{ \vec{x} = (Q_{k_1}(a), \dots, Q_{k_2}(a)), \vec{y} = (Q_{k_1}(b), \dots, Q_{k_2}(b)), Q_{k_1-1} \llbracket a,b \rrbracket = f, Q_{k_2+1} \llbracket a,b \rrbracket = g \},
	\]
	where if $k_1 = 1$ we use the convention $f = \infty = Q_0$. If $L_i \in \Omega(a, b, x_i , y_i)$ for $i\in\llbracket 1, k\rrbracket$ with $k = k_2-k_1+1$, then 
	\begin{equation}\label{GibbsEq}
	\mathbf{1}_E \cdot \mathbb{P}\left( Q_{i + k_1-1}\llbracket a,b \rrbracket = L_{i} \mbox{ for $i \in \llbracket 1, k\rrbracket$} \, \big|\,  \mathcal{F}_{\mathrm{ext}}^{H^{\mathrm{RW}}} \right) = \mathbf{1}_E\cdot\mathbb{P}_{\mathrm{avoid}, H^{\mathrm{RW}}}^{a,b, \vec{x}, \vec{y}, f, g} (L_1,\dots,L_k),
	\end{equation}
	where $\mathcal{F}_{\mathrm{ext}}^{H^{\mathrm{RW}}}$ is the $\sigma$-algebra generated by $L_i(j)$ for $i\in\llbracket k_1,k_2\rrbracket$ and $j\in\llbracket T_0,T_1\rrbracket\setminus\llbracket a+1,b-1\rrbracket$.
\end{definition}

%
\subsection{Main technical results}\label{Section2.3}

In this section, we state the main technical results of our paper, which we use in the following section to prove Theorem \ref{mainthm}. We establish two assumptions we will require our $\llbracket\alpha,\beta\rrbracket$-supported random walk Hamiltonians $H^{\mathrm{RW}}$ to satisfy throughout the rest of this paper, unless stated otherwise. We recall here that implicit in the definition of an $\llbracket\alpha,\beta\rrbracket$-supported random walk Hamiltonian is that $\exp(-H^{\mathrm{RW}}(k)) > 0$ for \textit{every} $k\in \llbracket\alpha,\beta\rrbracket$, meaning that each of these jump increments for the random walk has positive probability. (See Definition \ref{hamdef} and Remark \ref{consec}.)

The first assumption is needed for a monotone coupling lemma, \ref{MCL} below.

\begin{assumption}{A1}\label{assume1}
	The function $H^{\mathrm{RW}} : \mathbb{Z}\cup\{\pm\infty\} \to [0,\infty]$ is convex. That is, the extension of $H^{\mathrm{RW}}$ to $\mathbb{R}\cup\{\pm\infty\}$ by linear interpolation is a convex function.
\end{assumption}

The second assumption is a strong coupling condition, analogous to the Koml\'{o}s-Major-Tusn\'{a}dy (KMT) embedding. This is the key tool we use to study discrete random walks, by comparing them with Brownian bridges. These embeddings were recently studied extensively in \cite{DW19}.

\begin{assumption}{A2}\label{KMT}
	Let $p \in (\alpha,\beta)$. Then there exist constants $0 < C,A,a < \infty$ depending on $p$ and $H^{\mathrm{RW}}$ such that for every $n\in\mathbb{N}$, there is a probability space supporting a Brownian bridge $B^\sigma$ with variance $\sigma^2 = \sigma_p^2$ depending on $p$ and $H^{\mathrm{RW}}$, as well as a family of random paths $\ell^{(n,z)} \in \Omega(0,n, 0, z)$ for $z \in \llbracket n\alpha,n\beta\rrbracket$, such that $\ell^{(n,z)}$ has law $\mathbb{P}^{0,n,0,z}_{H^{\mathrm{RW}}}$ and
	\begin{equation}\label{KMTeq}
		\mathbb{E}\left[ e^{a \Delta(n,z)} \right] \leq C e^{A \log n}e^{|z- p n|^2/n}, \mbox{ where $\Delta(n,z)=  \sup_{0 \leq t \leq n} \left| \sqrt{n} B^\sigma_{t/n} + \frac{t}{n}z - \ell^{(n,z)}(t) \right|.$}
	\end{equation}
\end{assumption}

We refer to \cite[Section 2.2]{DW19} for a list of five technical conditions which ensure that Assumption \ref{KMT} holds. In Section 8.2 of that paper, the authors give several specific examples in which these conditions are satisfied. We summarize two of these here with which we are particularly concerned.

\begin{example}{E1}\label{ex1}
	If $H^{\mathrm{RW}}$ is an $\llbracket\alpha,\beta\rrbracket$-supported random walk Hamiltonian where $-\infty < \alpha < \beta < \infty$, then Assumption \ref{KMT} holds. This is to say that the allowable jumps of the random walk form a compact integer interval, and the jump probabilities may be arbitrary but nonzero. In particular, if we take $\alpha = 0, \beta = 1$, $q\in(0,1)$, and $H^{\mathrm{RW}}(0) = -\log q$, $H^{\mathrm{RW}}(1) = -\log (1-q)$, then Assumptions \ref{assume1} and \ref{KMT} are both satisfied, with the variance $\sigma^2$ of the Brownian bridge given by that of a Bernoulli random variable with parameter $q$, i.e., $\sigma^2 = q(1-q)$. (An earlier treatment of this special case appears in \cite[Theorem 8.1]{CD}.) This choice corresponds to the Bernoulli line ensembles studied in \cite{Ber}. 
\end{example}

\begin{example}{E2}\label{ex2}
	Suppose $H^{\mathrm{RW}}$ corresponds to a geometric distribution with parameter $q\in (0,1)$, i.e., $\alpha = 0$, $\beta = \infty$, and $H^{\mathrm{RW}}(k) = -\log[q(1-q)^k] = -k\log(1-q) - \log q$ for integers $k\geq 0$. Then Assumption \ref{KMT} is satisfied, and the extension of $H^{\mathrm{RW}}$ to $\mathbb{R}\cup\{\pm\infty\}$ is linear and thus convex, so Assumption \ref{assume1} also holds.
\end{example}

We now give a definition which contains the assumptions we require for tightness of a sequence of $H^{\mathrm{RW}}$-Gibbsian line ensembles. This is essentially a restatement of Assumptions \ref{a1} and \ref{a2}. The most important condition is a mild uniform control on the one-point marginals of the top curve near integer times, with a global parabolic shift.

\begin{definition}\label{gooddef} Fix $k \in \mathbb{N}$, $\gamma, \lambda > 0$, and $\theta\in[0,2)$. Let $H^{\mathrm{RW}}$ be an $\llbracket \alpha,\beta \rrbracket$-supported random walk Hamiltonian, and fix $p \in (\alpha,\beta)$. Suppose we have a sequence $( T_N )_{N\in\mathbb{N}}$ in $\mathbb{N}$ and that $(\mathfrak{L}^N)_{N \in\mathbb{N}}$, $\mathfrak{L}^N = (L^N_1, \dots, L^N_k)$, is a sequence of $\llbracket 1, k \rrbracket$-indexed discrete line ensembles on $ \llbracket -T_N, T_N \rrbracket$. We call the sequence $(\gamma,p,\lambda,\theta)$-{\em good} if 
	\begin{enumerate}[label=(\arabic*)]
		\item For each $N \in \mathbb{N}$, $\mathfrak{L}^N$ satisfies the $H^{\mathrm{RW}}$-Gibbs property of Definition \ref{DefHRWGP};  
		\item There is a function $\psi: \mathbb{N} \rightarrow (0, \infty)$ such that $\lim_{N \rightarrow \infty} \psi(N) = \infty$ and for each $N \in \mathbb{N}$ we have $ T_N > \psi(N)N^{\gamma}$;
		\item There are functions $\Phi:\mathbb{Z}\times(0,\infty)\to(0,\infty)$ and $\phi: (0, \infty) \rightarrow (0,\infty)$ such that for any $\epsilon > 0$, $n\in\mathbb{Z}$, and $N\geq\Phi(n,\epsilon)$, we have 
		\begin{equation}\label{globalParabola}
		\mathbb{P} \left( \left|N^{-\gamma/2}\left(L_1^N(\lfloor nN^\gamma\rfloor) - p\lfloor nN^\gamma\rfloor\right) + \lambda n^2 \right| \geq \phi(\epsilon) + |n|^\theta \right) \leq \epsilon.
		\end{equation}
	\end{enumerate}
\end{definition}

We now state our two main technical results, which we use below to prove Theorem \ref{mainthm}. Note that the line ensembles in the statements of these results have a fixed number of curves. The first result concerns tightness, and its proof occupies Section \ref{Section4}.

\begin{theorem}\label{GoodTight}
	Fix $k \in \mathbb{N}$ with $k \geq 2$, $\gamma, \lambda > 0$, $\theta\in[0,2)$, $H^{\mathrm{RW}}$ an $\llbracket \alpha,\beta \rrbracket$-supported random walk Hamiltonian satisfying Assumptions \ref{assume1} and \ref{KMT}, and $p \in (\alpha,\beta)$. Let $\mathfrak{L}^N = (L^N_1, \dots, L^N_k)$ be a $(\gamma, p, \lambda, \theta)$-good sequence of $\llbracket 1, k \rrbracket$-indexed discrete line ensembles. Define
	\[
	f^N_i(x) =  N^{-\gamma/2}\left(L^N_i(xN^{\gamma}) - p x N^{\gamma}\right), \mbox{ for $x\in [-\psi(N) ,\psi(N)]$ and $i \in \llbracket 1, k-1 \rrbracket$,}
	\]
	and extend $f^N_i$ to $\mathbb{R}$ by setting, for $i \in \llbracket 1, k - 1 \rrbracket$,
	\[
	f^N_i(x) = f^N_i(-\psi(N)) \mbox{ for $x \leq -\psi(N)$ and } f^N_i(x) = f_i^N(\psi(N)) \mbox{ for $x \geq \psi(N)$}.
	\]
	Let $\mathbb{P}_N$ denote the law of $f^N = (f_1^N,\dots,f^N_{k-1})$ as a $\llbracket 1, k-1 \rrbracket$-indexed line ensemble, i.e., as a random variable in $(C( \llbracket 1, k -1 \rrbracket \times \mathbb{R}), \mathcal{C}^{\llbracket 1,k-1\rrbracket})$. Then the sequence $(\mathbb{P}_N)_{N\in\mathbb{N}}$ is tight. 
\end{theorem}

The second result shows that the $H^{\mathrm{RW}}$-Gibbs property transfers to the partial Brownian Gibbs property in the limit. We give the proof in Section \ref{Section5}.

\begin{theorem}\label{GoodBGP}
	Fix the same notation as in Theorem \ref{GoodTight}. If $f^\infty = (f_1^{\infty},\dots,f_{k-1}^\infty)$ denotes any random variable whose law is a weak subsequential limit of $\mathbb{P}_N$, then there is a constant $\sigma > 0$ depending on $p$ and $H^{\mathrm{RW}}$ so that $\mathcal{L}^\infty := \sigma^{-1}f^{\infty}$ satisfies the partial Brownian Gibbs property.
\end{theorem}

\subsection{Proof of Theorems \ref{mainthm} and \ref{airythm}}\label{Section2.4}
	
	With Theorems \ref{GoodTight} and \ref{GoodBGP} stated above, we are ready to give the proof of Theorems \ref{mainthm} and \ref{airythm}. The primary difference between the statement of Theorem \ref{mainthm} and the combined statements of the previous two theorems is that the former deals with a sequence of line ensembles with an increasing number of curves, while in the latter the number of curves is fixed. As it suffices to show that each curve individually forms a tight sequence, the extension to the case of an increasing number of curves is fairly immediate. We now give the argument.
	
	\begin{proof}[Proof of Theorem \ref{mainthm}]
		
		We fix the same notation as in Section \ref{Section1.2}. We prove the two statements of the theorems in separate steps.\\
		
		\noindent\textbf{Step 1.} We first prove that the sequence $(f^N)_{N\geq 1}$, or equivalently $(\mathbb{P}_N)_{N\geq 1}$, is tight. By Lemma \ref{2Tight}, it suffices to prove for each $k\in\mathbb{N}$ that
		\begin{enumerate}[label=(\roman*)]
			\item $\lim_{a\to\infty} \limsup_{N\to\infty}\, \pr(|f^N_k(0)|\geq a) = 0$;
			\item For all $\epsilon>0$ and $m\in\mathbb{N}$, $\lim_{\delta\to 0} \limsup_{N\to\infty}\, \pr\bigg(\sup_{\substack{x,y\in [-m,m], \\ |x-y|\leq\delta}} |f^N_k(x) - f^N_k(y)| \geq \epsilon\bigg) = 0.$
		\end{enumerate}
		Set $T_N = \min(-a_N,b_N)$, and for $N \geq k+1$ let $\tilde{\mathfrak{L}}^N = (\tilde{L}_1^N, \dots, \tilde{L}_{k+1}^N)$ denote the $\llbracket 1, k+1\rrbracket$-indexed line ensemble obtained by restricting $\mathfrak{L}^N$ to the set $\llbracket 1,k+1\rrbracket\times\llbracket -T_N,T_N\rrbracket$. Since $\mathfrak{L}^N$ satisfies the $H^{\mathrm{RW}}$-Gibbs property, so does $\tilde{\mathfrak{L}}^N$, and Assumptions 1 and 2 in Section \ref{Section1.2} imply that $(\tilde{\mathfrak{L}}^N)_{N\geq k+1}$ is a $(\gamma,p,\lambda, \theta)$-good sequence in the sense of Definition \ref{gooddef}. In particular, for the function $\phi$ appearing in this definition, we take $\phi(\epsilon) = \varphi(\epsilon/2)$, where the function $\varphi$ is afforded by Assumption 2. Thus 
		\[
		\sup_{n\in\mathbb{Z}}\limsup_{N\to\infty}\mathbb{P} \left( \left|N^{-\gamma/2}\left(L_1^N(\lfloor nN^\gamma\rfloor) - p\lfloor nN^\gamma\rfloor\right) + \lambda n^2 \right| \geq \phi(\epsilon) + |n|^\theta \right) \leq \epsilon/2.
		\] 
		This allows us to choose, for each $n\in\mathbb{Z}$ and $\epsilon>0$, a $\Phi(n,\epsilon)\geq k+1$ such that 
		\[
		\mathbb{P} \left( \left|N^{-\gamma/2}\left(L_1^N(\lfloor nN^\gamma\rfloor) - p\lfloor nN^\gamma\rfloor\right) + \lambda n^2 \right| \geq \phi(\epsilon) + |n|^\theta \right) \leq \epsilon
		\] 
		for all $N\geq\Phi(n,\epsilon)$; this provides the function $\Phi$ appearing in Definition \ref{gooddef}. Now Theorem \ref{GoodTight} implies that the random variables $(\tilde{f}_i^N)_{i=1}^k$, defined as in the theorem statement for the line ensembles $\tilde{\mathfrak{L}}^N$, form a tight sequence in $(C(\llbracket 1,k\rrbracket\times\mathbb{R}), \mathcal{C}^{\llbracket 1,k\rrbracket})$. Note that for each $N$, $\tilde{f}_k^N$ has the same law as $f_k^N$ defined in Section \ref{Section1.2}. Since projection maps are continuous, the sequences $\tilde{f}_k^N(0)$ and $\tilde{f}_k^N|_{[-m,m]}$ are tight in $\mathbb{R}$ and $C[-m,m]$ respectively, implying conditions (i) and (ii) above. We conclude that the full sequence $(f^N)_{N\geq 1}$ is tight.\\
		
		\noindent\textbf{Step 2.} We now fix a weak subsequential limit $f^\infty$ of $f^N$ and show that $\mathcal{L}^\infty := \sigma^{-1}f^\infty$ possesses the Brownian Gibbs property. Without loss of generality, we can assume that $f^N \implies f^\infty$. It suffices to show that $\mathcal{L}^\infty$ is almost surely non-intersecting, and that if $a,b\in\mathbb{R}$ with $a<b$ and $K = \llbracket k_1,k_2\rrbracket\subset\mathbb{N}$ are fixed, then for any bounded Borel function $F : C(K\times[a,b])\to\mathbb{R}$, we have a.s.
		\begin{equation}\label{BGPmain}
		\ex\left[ F(\mathcal{L}^\infty |_{K\times(a,b)}) \, | \, \mathcal{F}_{\mathrm{ext}}(K\times(a,b))\right] = \ex^{a,b,\vec{x},\vec{y},f,g}_{\mathrm{avoid}}\big[F(\tilde{\mathcal{Q}})\big].
		\end{equation}
		Here our notation is as in Definition \ref{DefBGP}. Fix $k\geq k_2 + 1$. Then since the projection map $f^N \mapsto f^N|_{\llbracket 1,k\rrbracket\times\mathbb{R}}$ is continuous, $ f^N|_{\llbracket 1,k\rrbracket\times\mathbb{R}}$ converges weakly to $f^\infty|_{\llbracket 1,k\rrbracket\times\mathbb{R}}$ in $C(\llbracket 1,k\rrbracket\times\mathbb{R})$. Let $(\tilde{f}_i^N)_{i=1}^k$ be as in Step 1. By construction, the restriction of $(\tilde{f}_i^N)_{i=1}^k$ to the interval $[-\psi(N),\psi(N)]$ has the same law as the restriction of $f^N|_{\llbracket 1,k\rrbracket\times\mathbb{R}}$ to this interval. Since $\psi(N)\to\infty$ and $f^N|_{\llbracket 1,k\rrbracket\times\mathbb{R}} \implies f^\infty|_{\llbracket 1,k\rrbracket\times\mathbb{R}}$ with respect to uniform convergence over compact sets, it follows that $(\tilde{f}_i^N)_{i=1}^k$ converges weakly to $f^\infty|_{\llbracket 1,k\rrbracket\times\mathbb{R}}$ as $N\to\infty$. Theorem \ref{GoodBGP} now implies that $\mathcal{L}^\infty|_{\llbracket 1,k\rrbracket\times\mathbb{R}} = \sigma^{-1} f^\infty|_{\llbracket 1,k\rrbracket\times\mathbb{R}}$ possesses the partial Brownian Gibbs property. Then in particular $\mathcal{L}^\infty|_{\llbracket 1,k\rrbracket\times\mathbb{R}}$ is a.s. non-intersecting, and since $k\geq k_2+1$ was arbitrary it follows that $\mathcal{L}^\infty$ is a.s. non-intersecting. Moreover, we have almost surely
		\begin{equation}\label{BGPk}
			\ex\big[ F(\mathcal{L}^\infty|_{K\times(a,b)})\, |\, \tilde{\mathcal{F}}_{\mathrm{ext}}(K\times(a,b))\big] = \ex_{\mathrm{avoid}}^{a,b,\vec{x},\vec{y},f,g}\big[F(\tilde{\mathcal{Q}})\big],
		\end{equation}
		where
		\[
		\tilde{\mathcal{F}}_{\mathrm{ext}}(K\times(a,b)) = \sigma\big(f_i^\infty(s) : (i,s) \in (\llbracket 1,k\rrbracket \times \mathbb{R})\setminus(K\times(a,b))\big).
		\]
		Now define the $\pi$-system $\mathcal{A}$ of sets $A$ of the form
		\[
		A = \{f^\infty_{i_r}(x_r) \in B_r \mbox{ for } r\in\llbracket 1,m\rrbracket\},
		\]
		where $m\in\mathbb{N}$, $B_r\in\mathcal{B}(\mathbb{R})$, and $(i_r,x_r)\notin K\times(a,b)$. Since $k\geq k_2+1$ was arbitrary in \eqref{BGPk}, for any fixed $A\in\mathcal{A}$ we can take $k\geq m$, so that $A\in\tilde{\mathcal{F}}_{\mathrm{ext}}(K\times(a,b))$. It then follows that
		\[
		\ex\left[F(\mathcal{L}^\infty|_{K\times(a,b)})\,\mathbf{1}_A\right] = \ex\left[\ex_{\mathrm{avoid}}^{a,b,\vec{x},\vec{y},f,g}\big[F(\tilde{\mathcal{Q}})\big]\,\mathbf{1}_A\right].
		\]
		Let $\mathcal{B}$ denote the collection of sets $A$ satisfying the above equation. Then $\mathcal{B}$ is a $\lambda$-system. Indeed, \eqref{BGPk} implies that the equation holds with $\mathbf{1}_A = 1$, and since $\mathbf{1}_{A^c} = 1-\mathbf{1}_A$ we see that $\mathcal{B}$ is closed under complementation. Lastly if $A_1\subseteq A_2\subseteq\cdots$ are elements of $\mathcal{B}$ and $A = \bigcup_{i=1}^\infty A_i$, then $\mathbf{1}_{A_i}\nearrow\mathbf{1}_A$, so the monotone convergence theorem implies that $A\in\mathcal{B}$. Since $\mathcal{B}$ contains the $\pi$-system $\mathcal{A}$, it follows from the $\pi$-$\lambda$ theorem \cite[Theorem 2.1.6]{Durrett} that $\mathcal{B}$ contains $\sigma(\mathcal{A}) = \mathcal{F}_{\mathrm{ext}}(K\times(a,b))$. It is proven in \cite[Lemma 3.4]{DimMat} that the right hand side of \eqref{BGPmain} is $\mathcal{F}_{\mathrm{ext}}(K\times(a,b))$-measurable, and so \eqref{BGPmain} follows from the definition of conditional expectation. This completes the proof.
		
	\end{proof} 

We now turn to the proof of Theorem \ref{airythm}. We recall from Section \ref{Section1.2} that Assumption \ref{a2'} implies Assumption \ref{a2}, so that Theorem \ref{mainthm} may be applied in the proof. 

	\begin{proof}[Proof of Theorem \ref{airythm}]
		
		In view of Theorem \ref{mainthm}, we know that $(\mathcal{L}^N)_{N\in\mathbb{N}}$ is a tight sequence. Let $\mathcal{K}^\infty$ denote any weak subsequential limit of this sequence. We will show that $\mathcal{K}^\infty$ has the same law as the line ensemble $\mathcal{L}^\infty$ defined in the theorem statement. Thus $\mathcal{L}^N$ has at most one possible subsequential limit, $\mathcal{L}^\infty$, and in combination with tightness this implies that $\mathcal{L}^N \implies \mathcal{L}^\infty$.
		
		Let $(\mathcal{L}^{N_\ell})_{\ell\in\mathbb{N}}$ be a subsequence converging weakly to $\mathcal{K}^\infty$. By the Skorohod representation theorem (using Lemma 2.3), by moving to another probability space we may assume that $\mathbb{P}$-almost surely, $\mathcal{L}^{N_\ell}\to\mathcal{K}^\infty$ uniformly on compact sets as $\ell\to\infty$. Note that we may rewrite Assumption 2' in terms of the random variables $\mathcal{L}^N = \sigma_p^{-1}f^N$ as the statement that for any $k\in\mathbb{N}$ and $t_1,\dots,t_k,x_1,\dots,x_k\in\mathbb{R}$, we have
		\begin{equation}\label{lim1}
		\lim_{N\to\infty}\mathbb{P}\left(\mathcal{L}_1^N(\lfloor t_i N^\gamma\rfloor N^{-\gamma}) \leq x_i \mbox{ for } i\in\llbracket 1,k\rrbracket\right) = \mathbb{P}\left(\mathcal{L}^\infty_1(t_i)\leq x_i \mbox{ for } i\in\llbracket 1,k\rrbracket\right).
		\end{equation}
		Now $\lfloor t_i N^\gamma\rfloor N^{-\gamma} \to t_i$ as $N\to\infty$ for each $i$, and because the sequence of continuous functions $\mathcal{L}^{N_\ell}$ converges uniformly to $\mathcal{K}^\infty$ a.s. it follows that $\mathcal{L}_1^{N_\ell}(\lfloor t_i N^\gamma\rfloor N^{-\gamma}) \to \mathcal{K}^\infty(t_i)$, $\mathbb{P}$-a.s. as $\ell\to\infty$. This implies that
		\begin{equation}\label{lim2}
		\lim_{\ell\to\infty}\mathbb{P}\left(\mathcal{L}_1^{N_\ell}(\lfloor t_i N^\gamma\rfloor N^{-\gamma})\leq x_i \mbox{ for } i\in\llbracket 1,k\rrbracket\right) = \mathbb{P}\left(\mathcal{K}^\infty_1(t_i)\leq x_i \mbox{ for } i\in\llbracket 1,k\rrbracket\right).
		\end{equation}
		Taking the limit in \eqref{lim1} along the subsequence $(N_\ell)$ and combining with \eqref{lim2}, we see that the top curves $\mathcal{K}^\infty_1$ and  $\mathcal{L}^\infty_1$ have the same finite-dimensional distributions.
		
		Moreover, by Theorem \ref{mainthm}, $\mathcal{K}^\infty$ satisfies the Brownian Gibbs property. By \cite[Theorem 3.1]{CH14}, $\mathcal{L}^{\mathrm{Ai}}$ satisfies the Brownian Gibbs property, and since $\mathcal{L}_i^\infty(x) = c^{-1/2}\mathcal{L}_i^{\mathrm{Ai}}(cx)$, we see that $\mathcal{L}^\infty$ satisfies the property as well. (This follows from the fact that if $W_t$ is a standard Brownian motion then so is $c^{-1/2}W_{ct}$; we refer to \cite[Lemma 3.5]{DimMat} for a proof of a similar result.) In summary, $\mathcal{L}^\infty$ and $\mathcal{K}^\infty$ both satisfy the Brownian Gibbs property, and their top curves have the same finite-dimensional distributions. It follows from \cite[Theorem 1.1]{DimMat} that $\mathcal{L}^\infty$ and $\mathcal{K}^\infty$ have the same law. This proves the theorem.
		
	\end{proof}

%% file: section3.tex
%
\section{Properties of discrete line ensembles}\label{Section3}

In this section, we prove several results for discrete line ensembles which will be used in the proof of Theorems \ref{GoodTight} and \ref{GoodBGP} in Sections \ref{Section4} and \ref{Section5}. Throughout this section, we fix $\alpha,\beta\in\mathbb{Z}\cup\{\pm\infty\}$ with $\alpha < \beta$ and an $\llbracket \alpha,\beta \rrbracket$-supported random walk Hamiltonian $H^{\mathrm{RW}}$ satisfying Assumptions \ref{assume1} and \ref{KMT}.

\subsection{Monotone coupling lemma}\label{Section3.1}

The following lemma provides a coupling of the non-crossing $H^{\mathrm{RW}}$ line ensemble measures of Definition \ref{HRWavoiddef} which is monotone in the entry and exit data and the values of the lower bounding curve. Roughly, it states that if we raise the entry and exit data or the lower bounding curve of a non-crossing line ensemble, then the curves of the ensemble will be stochastically shifted upwards on their entire interval of definition. We note that the proof requires only Assumption \ref{assume1} on the Hamiltonian $H^{\mathrm{RW}}$. The proof of the lemma is given in the \hyperref[Appendix]{Appendix}.

\begin{lemma}\label{MCL}
	Fix $k \in \mathbb{N}$, $T_0, T_1 \in \mathbb{Z}$ with $T_0 < T_1$, $S\subseteq\llbracket T_0, T_1\rrbracket$, and two functions $g^{\mathrm{b}}, g^{\mathrm{t}}: \llbracket T_0, T_1 \rrbracket  \rightarrow [-\infty, \infty)$ with $g^{\mathrm{b}}\leq g^{\mathrm{t}}$ on $S$. Also fix $\vec{x}, \vec{y}, \vec{x}\,', \vec{y}\,' \in \mathfrak{W}_k$ such that $x_i\leq x_i'$, $y_i\leq y_i'$ for $i\in\llbracket 1,k\rrbracket$. Assume that $\Omega_{\mathrm{avoid}}(T_0, T_1, \vec{x}, \vec{y}, \infty,g^{\mathrm{b}}; S)$ and $\Omega_{\mathrm{avoid}}(T_0, T_1, \vec{x}\,', \vec{y}\,', \infty,g^{\mathrm{t}}; S)$ are both non-empty. Then there exists a probability space $(\Omega, \mathcal{F}, \mathbb{P})$, which supports two $\llbracket 1, k \rrbracket$-indexed $H^{\mathrm{RW}}$ line ensembles $\mathfrak{L}^{\mathrm{b}}$ and $\mathfrak{L}^{\mathrm{t}}$ on $\llbracket T_0, T_1 \rrbracket$ such that the law of $\mathfrak{L}^{\mathrm{b}}$ {\big (}resp. $\mathfrak{L}^{\mathrm{t}}${\big )} under $\mathbb{P}$ is given by $\mathbb{P}_{\mathrm{avoid}, H^{\mathrm{RW}}; S}^{T_0, T_1, \vec{x}, \vec{y}, \infty, g^{\mathrm{b}}}$ {\big (}resp. $\mathbb{P}_{\mathrm{avoid}, H^{\mathrm{RW}}; S}^{T_0, T_1, \vec{x}\,', \vec{y}\,', \infty, g^{\mathrm{t}}}${\big )} and such that $\mathbb{P}$-almost surely we have $\mathfrak{L}_i^{\mathrm{b}}(r) \leq \mathfrak{L}^{\mathrm{t}}_i(r)$ for all $i \in \llbracket 1, k\rrbracket$ and $r \in \llbracket T_0, T_1 \rrbracket$.
\end{lemma}

%
\subsection{Properties of random walk bridges}\label{Section3.2}

We next prove several results of random walk bridges with laws $\mathbb{P}^{T_0,T_1,x,y}_{H^{\mathrm{RW}}}$ (see Definition \ref{hamdef}), as well as Brownian bridges. The first result we state is an immediate consequence of Assumption \ref{KMT} that allows us to directly compare random walk bridges to Brownian bridges. 

\begin{corollary}\label{Cheb}
	Fix $p\in(\alpha,\beta)$ and $\gamma, M > 0$. Suppose $|z-pn| \leq K\sqrt{n}$ for some constant $K > 0$. Then for any $\epsilon > 0$, there exists $N\in\mathbb{N}$ depending on $M,K,p,\epsilon,\gamma$ so that for $n\geq N$ we have
	\[
	\mathbb{P}\left(\Delta(n,z) \geq Mn^\gamma\right) < \epsilon,
	\]
	where $\Delta(n,z)$ is defined in \eqref{KMTeq}.
\end{corollary}

\begin{proof}
	Using Chebyshev's inequality with \eqref{KMTeq}, we find 
	\begin{align*}
	\mathbb{P}\left(\Delta(n,z) \geq Mn^\gamma\right) &\leq e^{-aMn^\gamma}\ex\left[e^{a\Delta(n,z)}\right] \leq C\exp\left[-aMn^\gamma + A\log n + \frac{|z-pn|^2}{n}\right]\\
	&\leq C\exp\left[-aMn^\gamma + A\log n + K^2\right],
	\end{align*}
	where the constants $0 < C,A,a < \infty$ depend on $p$. The expression in the last line tends to 0 as $n\to\infty$, and the conclusion follows.
\end{proof}

The next result gives exact formulas for the distribution of the maximum of a Brownian bridge and the maximum absolute value of a Brownian bridge with prescribed variance. The proof is an elementary computation following from the reflection principle; see \cite[Propositions 12.3.3 and 12.3.4]{Dudley} and \cite[Lemma 3.6]{Ber}.

\begin{lemma}\label{BBmax}
	Let $B^\sigma$ be a Brownian bridge of variance $\sigma^2$ on $[0,1]$. Then for any $C,T > 0$, we have
	\begin{align}
	\mathbb{P}\left(\max_{s\in[0,T]} B^\sigma_{s/T} \geq C\right) &= \exp\left(-2C^2/\sigma^2\right),\label{maxDist}\\
	\mathbb{P}\left(\max_{s\in[0,T]} |B^\sigma_{s/T}| \geq C\right) &= 2\sum_{n=1}^\infty (-1)^{n-1}\exp\left(-2n^2C^2/\sigma^2\right).\label{maxAbsDist}
	\end{align}
	In particular,
	\begin{equation}\label{maxAbsBd}
	\mathbb{P}\left(\max_{s\in[0,T]} |B^\sigma_{s/T}| \geq C\right) \leq 2\exp\left(-2C^2/\sigma^2\right).
	\end{equation}
\end{lemma}

We state one more lemma about Brownian bridges, which says essentially that a bridge on $[0,1]$ can be decomposed into two independent bridges with appropriate Gaussian affine shifts meeting at a given point in $(0,1)$. This lemma can be proven by computing the covariance of the centered Gaussian process $\tilde{B}^\sigma$ and using the fact that a Brownian bridge is characterized by its covariance. We refer to \cite[Lemma 3.7]{Ber} for the details of this argument.

\begin{lemma}\label{2bridges}
	Fix $T>0$, $t\in(0,T)$, and let $\xi$ be a Gaussian random variable with mean 0 and variance
	\[
	\ex[\xi^2] = \sigma^2\frac{t}{T}\left(1-\frac{t}{T}\right).
	\]
	Let $B^1,B^2$ be two independent Brownian bridges on $[0,1]$ with variances $\sigma^2 t/T$ and $\sigma^2(T-t)/T$ respectively. Define the process
	\[
	\tilde{B}^\sigma_{s/T} = \begin{dcases}
		\frac{s}{t}\,\xi + B^1\Big(\frac{s}{t}\Big), & s\leq t,\\
		\frac{T-s}{T-t}\,\xi + B^2\Big(\frac{s-t}{T-t}\Big), & s\geq t,
	\end{dcases}
	\]
	for $s\in [0,T]$. Then $\tilde{B}^\sigma$ is a Brownian bridge with variance $\sigma$.
\end{lemma} 

We now give a few estimates for random walk bridges, analogous to those appearing in \cite[Section 3.2]{Ber}. Each of these lemmas are proven using the monotone coupling Lemma \ref{MCL} and strong coupling Lemma \ref{Cheb}. If not explicitly stated, in the following $\ell$ will denote a random variable whose law is whichever measure it appears alongside; the measures may change from line to line, but we use the same letter $\ell$ for simplicity. 

The first lemma gives a lower bound on the probability that a random walk bridge does not fall far below the straight line connecting its endpoints at a fixed interior time.

\begin{lemma}\label{BridgeInterp}
	Fix $p\in(\alpha,\beta)$ and $M_1,M_2\in\mathbb{R}$. Then we can find $T_1\in\mathbb{N}$ depending on $M_2-M_1$ and $p$ so that the following holds for all $T\geq T_1$. If $x,y\in\mathbb{Z}$ satisfy $y-x\in\llbracket\alpha T,\beta T\rrbracket$, $x\geq M_1\sqrt{T}$, and $y \geq pT + M_2\sqrt{T}$, then for all $s\in[0,T]$ we have
	\begin{equation}\label{BridgeInterpEq}
	\mathbb{P}^{0,T,x,y}_{H^{\mathrm{RW}}}\left(\ell(s) \geq \left(1-\frac{s}{T}\right) M_1\sqrt{T} + \frac{s}{T}\left(pT+M_2\sqrt{T}\right) - T^{1/3}\right) \geq \frac{1}{3}.
	\end{equation}
	
\end{lemma}

\begin{proof}
	Let us define $A = \lfloor M_1\sqrt{T}\rfloor$, $B = \lfloor pT + M_2\sqrt{T}\rfloor$, so that $x\geq A$, $y\geq B$. By Lemma \ref{MCL}, there is a probability space supporting random variables $\ell_1,\ell_2$ with laws $\mathbb{P}^{0,T,x,y}_{H^{\mathrm{RW}}},\mathbb{P}^{0,T,A,B}_{H^{\mathrm{RW}}}$ respectively, such that $\ell_1 \geq \ell_2$ a.s. It follows that the probability on the left hand side of \eqref{BridgeInterpEq} is bounded below by
	\begin{equation}\label{BridgeIntMC}
	\begin{split}
	&\mathbb{P}^{0,T,A,B}_{H^{\mathrm{RW}}}\left(\ell(s) \geq \left(1-\frac{s}{T}\right)M_1\sqrt{T} + \frac{s}{T}\left(pT+M_2\sqrt{T}\right) - T^{1/3}\right)\\
	= \; & \mathbb{P}^{0,T,0,B-A}_{H^{\mathrm{RW}}}\left(\ell(s) + A \geq \left(1-\frac{s}{T}\right)M_1\sqrt{T} + \frac{s}{T}\left(pT+M_2\sqrt{T}\right) - T^{1/3}\right).
	\end{split}
	\end{equation}
	The equality in the second law follows from the simple observation that the distribution on discrete paths is unchanged if we shift vertically by $A$.
	
	We now utilize the strong coupling given by Assumption \ref{KMT}. We have a probability space with a measure $\mathbb{P}_0$, supporting a random variable $\ell^{(T,B-A)}$ with law $\mathbb{P}^{0,T,0,B-A}_{H^{\mathrm{RW}}}$, as well as a Brownian bridge $B^\sigma$. The second line of \eqref{BridgeIntMC} is equal to
	\begin{equation}\label{BridgeIntKMT}
	\begin{split}
	&\mathbb{P}_0\bigg(\left[\ell^{(T,B-A)}(s) - \sqrt{T}\,B^\sigma_{s/T} - \frac{s}{T}(B-A)\right] + \sqrt{T}\,B^\sigma_{s/T} \\
	&\qquad \leq -A - \frac{s}{T}(B-A) + \left(1-\frac{s}{T}\right) M_1\sqrt{T} + \frac{s}{T}\left(pT + M_2\sqrt{T}\right) - T^{1/3}\bigg)\\
	\geq \; & \mathbb{P}_0\left(\left[\ell^{(T,B-A)}(s) - \sqrt{T}\,B^\sigma_{s/T} - \frac{s}{T}(B-A)\right] + \sqrt{T}\,B^\sigma_{s/T} \leq -T^{1/3} + 1\right).
	\end{split}
	\end{equation}
	The inequality in the third line follows by rewriting the expression appearing in the second line as
	\begin{align*}
	&\frac{T-s}{T}\left(M_1\sqrt{T}-A\right) +\frac{s}{T}\left(pT+M_2\sqrt{T}-B\right) - T^{1/3}\\
	\leq \; & \frac{s}{T} + \frac{T-s}{T} - T^{1/3} = -T^{1/3} + 1.
	\end{align*}
	Observe that the expression in brackets in the third line of \eqref{BridgeIntKMT} is bounded in absolute value by $\Delta(T,B-A)$, as defined in Assumption \ref{KMT}. Thus we can bound the last line of \eqref{BridgeIntKMT} from below by
	\begin{equation}\label{BridgeIntSplit}
	\begin{split}
	&\mathbb{P}_0\left(\Delta(T,B-A) \leq T^{1/3} - 1 \quad \mathrm{and} \quad \sqrt{T}\,B^\sigma_{s/T} \geq 0\right)\\
	\geq \; & \mathbb{P}_0\left(B^\sigma_{s/T} \geq 0\right) - \mathbb{P}_0\left(\Delta(T,B-A) > T^{1/3} - 1\right)\\
	= \; & \frac{1}{2} - \mathbb{P}_0\left(\Delta(T,B-A) > T^{1/3} - 1\right).
	\end{split}
	\end{equation}
	Note that $|B-A-pT| \leq (|M_2-M_1|+1)\sqrt{T}$, so Lemma \ref{Cheb} allows us to choose $T_1$ depending on $M_2-M_1$ and $p$ so that $\mathbb{P}_0(\Delta(T,B-A) > T^{1/3}/2) \leq 1/6$ for all $T\geq T_1$. As long as $T_1^{1/3} \geq 2$, it then follows that the last line of \eqref{BridgeIntSplit} is bounded below by $1/2-1/6 = 1/3$ for $T\geq T_1$. This proves \eqref{BridgeInterpEq}.
	
\end{proof}

The next lemma shows roughly that if the endpoints of a bridge on $[0,T]$ are not too low (resp. high) on scale $\sqrt{T}$, then with high probability the bridge will not become very low (resp. high) on scale $\sqrt{T}$ anywhere on $[0,T]$.

\begin{lemma}\label{BridgeInfSup}
	Fix $p\in(\alpha,\beta)$, $M,\epsilon > 0$. Then we can find $T_2\in\mathbb{N}$ and $A > 0$ both depending on $M,p,\epsilon$ so that the following holds for all $T \geq T_2$. If $y,z\in\llbracket\alpha T,\beta T\rrbracket$ satisfy $y \geq pT - M\sqrt{T}$ and $z\leq pT + M\sqrt{T}$, then
	\begin{align}
	&\mathbb{P}^{0,T,0,y}_{H^{\mathrm{RW}}}\left(\inf_{s\in[0,T]} \left( \ell^y(s) - ps \right) \leq -A\sqrt{T}\right) < \epsilon,\label{BridgeInf}\\
	&\mathbb{P}^{0,T,0,z}_{H^{\mathrm{RW}}}\bigg(\sup_{s\in[0,T]} \left( \ell^z(s) - ps \right) \geq A\sqrt{T}\bigg) < \epsilon. \label{BridgeSup}
	\end{align}
	Here, $\ell^y$ and $\ell^z$ denote random variables with laws $\mathbb{P}^{0,T,0,y}_{H^{\mathrm{RW}}}$ and $\mathbb{P}^{0,T,0,z}_{H^{\mathrm{RW}}}$ respectively. 
\end{lemma}

\begin{proof}
	We first prove \eqref{BridgeInf}. Let $B = \lfloor pT - M\sqrt{T}\rfloor$, so that $y\geq B$. It follows from Lemma \ref{MCL} that the probability on the left of \eqref{BridgeInf} is bounded below by
	\begin{equation}\label{BridgeInfMC}
	\mathbb{P}^{0,T,0,y}_{H^{\mathrm{RW}}}\left(\inf_{s\in[0,T]} \left( \ell^y(s) - ps \right) \leq -A\sqrt{T}\right) \leq \mathbb{P}^{0,T,0,B}_{H^{\mathrm{RW}}}\left(\inf_{s\in[0,T]}\left(\ell(s)-ps\right) \leq -A\sqrt{T}\right).
	\end{equation}
	By Assumption \ref{KMT}, there is a probability space with measure $\mathbb{P}_0$ supporting a random variable $\ell^{(T,B)}$ with law $\mathbb{P}^{0,T,0,B}_{H^{\mathrm{RW}}}$, as well as a Brownian bridge $B^\sigma$ with variance $\sigma^2$ depending only on $p$. The expression on the right of \eqref{BridgeInfMC} is bounded above by
	\begin{equation}\label{BridgeInfKMT}
	\begin{split}
	&\mathbb{P}_0\left(\inf_{s\in[0,T]} \sqrt{T}\,B^\sigma_{s/T} \leq -A\sqrt{T}/2\right) + \mathbb{P}_0\bigg(\sup_{s\in[0,T]}\left|\ell^{(T,B)}(s) - ps - \sqrt{T}\,B^\sigma_{s/T}\right| \geq A\sqrt{T}/2\bigg)\\
	\leq \; & \mathbb{P}_0\left(\max_{s\in[0,T]} B^\sigma_{s/T} \geq A/2\right) + \mathbb{P}_0\left(\Delta(T,B) \geq (A/2-M)\sqrt{T}-1\right).
	\end{split}
	\end{equation}
	Here, $\Delta(T,B)$ is as defined in Assumption \ref{KMT}. For the first term in the second line, we used the fact that $B^\sigma$ and $-B^\sigma$ have the same law. For the second term, we used the fact that 
	\[
	\sup_{s\in[0,T]}|ps - (s/T)B| \leq \sup_{s\in[0,T]}\left|ps - \frac{s}{T}\big(pT-M\sqrt{T})\right| + 1 = M\sqrt{T} + 1.
	\]
	Now by Lemma \ref{BBmax}, the first term in the second line of \eqref{BridgeInfKMT} is equal to $e^{-A^2/2\sigma^2}$. This is $<\epsilon/2$ as long as $A > \sigma\sqrt{2\log(2/\epsilon)}$. Since $|B-pT| \leq (M+1)\sqrt{T}$, as long as $A > 2M$, Lemma \ref{Cheb} allows us to choose $T_2$ depending on $M,p,\epsilon$ so that the second term in the last line of \eqref{BridgeInfKMT} is also $<\epsilon/2$ for $T\geq T_2$. This implies \eqref{BridgeInf}.
	
	The proof of \eqref{BridgeSup} is nearly identical. After replacing $B$ with $\lceil pT + M\sqrt{T}\rceil$ and swapping signs and inequalities where appropriate, the same argument applies.
\end{proof}

The third lemma concerns the modulus of continuity of a rescaled bridge. For a function $f\in C([a,b])$ and any $\delta > 0$, we write
\[
w(f,\delta) := \sup_{x,y\in[a,b],\, |x-y|<\delta} |f(x)-f(y)|
\]
for the modulus of continuity of $f$.

\begin{lemma}\label{MOC}
	Fix $p\in(\alpha,\beta)$ and $M,\epsilon,\eta > 0$. Then we can find $T_3\in\mathbb{N}$ depending on $M,p,\epsilon,\eta$ and $\delta > 0$ depending on $M,\epsilon,\eta$ so that the following holds for all $T\geq T_3$. If $y\in\llbracket \alpha T,\beta T\rrbracket$ satisfies $|y-pT| \leq M\sqrt{T}$, then
	\begin{equation}\label{MOCeq}
	\mathbb{P}^{0,T,0,y}_{H^{\mathrm{RW}}}\left(w(f^\ell,\delta) \geq \epsilon\right) < \eta,
	\end{equation}
	where $f^\ell(s) = T^{-1/2}(\ell(sT) - psT)$ for $s\in [0,1]$. 
\end{lemma}

\begin{proof}
	By Assumption \ref{KMT}, there is a probability space with measure $\mathbb{P}_0$ supporting a random variable $\ell^{(T,y)}$ with law $\mathbb{P}^{0,T,0,y}_{H^{\mathrm{RW}}}$, as well as a Brownian bridge $B^\sigma$. We have
	\begin{align*}
	w\big(f^{\ell^{(T,y)}},\delta\big) &= T^{-1/2}\sup_{s,t\in[0,1],\,|s-t| < \delta} \left|\ell^{(T,y)}(sT) - psT - \ell^{(T,y)}(tT) + ptT\right|\\
	&\leq T^{-1/2}\sup_{s,t\in[0,1],\,|s-t| < \delta} \bigg(\left|\sqrt{T}\,B^\sigma_s + sy - psT - \sqrt{T}\,B^\sigma_t - ty + ptT\right|\\
	&\qquad\qquad\qquad + \left|\sqrt{T}\,B^\sigma_s + sy - \ell^{(T,y)}(sT)\right| + \left|\sqrt{T}\,B^\sigma_t + ty - \ell^{(T,y)}(tT)\right|\bigg)\\
	&\leq \sup_{s,t\in[0,1],\,|s-t|<\delta} \left|B_s^\sigma - B_t^\sigma + T^{-1/2}(y-pT)(s-t)\right| + 2T^{-1/2}\Delta(T,y)\\
	&\leq w(B^\sigma,\delta) + M\delta + 2T^{-1/2}\Delta(T,y).
	\end{align*}
	Here, $\Delta(T,y)$ is as defined in Assumption \ref{KMT}. For the second term in the last line, we used the assumption that $|y-pT| \leq M\sqrt{T}$. It follows that
	\begin{align*}
	\mathbb{P}^{0,T,0,y}_{H^{\mathrm{RW}}}\left(w(f^\ell,\delta) \geq \epsilon\right) &\leq \mathbb{P}_0\left(w(B^\sigma,\delta) + M\delta + 2T^{-1/2}\Delta(T,y) \geq \epsilon\right)\\
	&\leq \mathbb{P}_0\left(w(B^\sigma,\delta) + M\delta \geq \epsilon/2\right) + \mathbb{P}_0\left(\Delta(T,y) \geq \epsilon\sqrt{T}/4\right).
	\end{align*}
	Since the paths of $B^\sigma$ are a.s. continuous and $[0,1]$ is compact, we have $w(B^\sigma,\delta)\to 0$ a.s. as $\delta \to 0$. It follows that $w(B^\sigma,\delta)\to 0$ in probability, so we can choose $\delta < \epsilon/4M$ depending on $\epsilon,\eta$ so that $w(B^\sigma,\delta) < \epsilon/4$ with probability at least $1-\eta/2$. Then $\mathbb{P}_0(w(B^\sigma,\delta) + M\delta \geq \epsilon/2) < \eta/2$. Since $|y-pT| \leq M\sqrt{T}$, Lemma \ref{Cheb} allows us to choose $T_3$ depending on $M,p,\epsilon,\eta$ so that also $\mathbb{P}_0(\Delta(T,y) \geq \epsilon\sqrt{T}/4) < \eta/2$ for $T\geq T_3$. This implies \eqref{MOCeq}.
\end{proof}

The next result concerns independent random walk bridges. It says roughly that if the entry and exit data of such bridges are ordered and well-separated, then with some nonzero probability the curves will not cross one another anywhere on their domain of definition. This will be useful for analyzing random walk bridges conditioned to avoid one another, as it allows us to reduce the analysis to that of independent bridges, for which the previous three lemmas can be used. We employ this type of argument extensively in the proof of Theorem \ref{GoodTight} in Section \ref{Section4}.

\begin{lemma}\label{CurveSep}
	Fix $p\in(\alpha,\beta)$, $k\in\mathbb{N}$ with $k\geq 2$, and $C,K > 0$. Let $\sigma = \sigma_p$ be as in Assumption \ref{KMT} for this value of $p$. Then we can find $T_4\in\mathbb{N}$ depending on $C,K,p$ so that the following holds for all $T\geq T_4$. Let $a,b\in\mathbb{Z}$ be such that $\Omega(0,T,a,b)$ is nonempty, and let $\ell_{\mathrm{b}} \in \Omega(0,T,a,b)$. Suppose $\vec{x},\vec{y}\in\mathfrak{W}_{k-1}$ are such that $y_i-x_i\in\llbracket\alpha T,\beta T\rrbracket$ for $\llbracket 1,k-1\rrbracket$, write $\vec{z}=\vec{y}-\vec{x}$, and suppose that
	\begin{enumerate}[label=(\arabic*)]
		
		\item $x_{k-1} + (z_{k-1}/T)s - \ell_{\mathrm{b}}(s) \geq C\sqrt{T}$ for all $s\in\llbracket0,T\rrbracket$, i.e., $\ell_{\mathrm{b}}$ lies a distance of at least $C\sqrt{T}$ uniformly below the line segment connecting $x_{k-1}$ and $y_{k-1}$;
		
		\item $x_i - x_{i+1} \geq C\sqrt{T}$ and $y_i - y_{i+1} \geq C\sqrt{T}$ for $i\in\llbracket 1,k-2\rrbracket$;
		
		\item $|z_i-pT| \leq K\sqrt{T}$ for $i\in\llbracket 1,k-1\rrbracket$.
	
	\end{enumerate}
	Let $(L_1,\dots,L_{k-1})$ be a $\llbracket 1,k-1\rrbracket$-indexed $H^{\mathrm{RW}}$ line ensemble with law $\mathbb{P}^{0,T,\vec{x},\vec{y}}_{H^{\mathrm{RW}}}$, and define the avoidance event
	\[
	A = \{L_1(s) \geq \cdots \geq L_{k-1}(s) \geq \ell_{\mathrm{b}}(s) \mbox{ for } s\in\llbracket0,T\rrbracket\}.
	\]
	Then
	\begin{equation}\label{CurveSepBd}
	\mathbb{P}^{0,T,\vec{x},\vec{y}}_{H^{\mathrm{RW}}}(A) \geq \left(\frac{1}{2} - \sum_{n=1}^\infty (-1)^{n-1}e^{-n^2C^2/8\sigma^2}\right)^{k-1},
	\end{equation}
	and the expression in parentheses on the right hand side is strictly positive. Moreover, if $C^2 \geq 8\sigma^2\log 3$, then 
	\begin{equation}\label{CurveSepWeak}
	\mathbb{P}^{0,T,\vec{x},\vec{y}}_{H^{\mathrm{RW}}}(A) \geq \left(1 - 3e^{-C^2/8\sigma^2}\right)^{k-1}.
	\end{equation}
\end{lemma}

\begin{proof}
	Observe that conditions (1) and (2) together imply that the event $A$ occurs as long as each curve $L_i$ remains within a distance of $C\sqrt{T}/2$ from the line segment connecting its endpoints $x_i$ and $y_i$ on the interval $[0,T]$. That is,
	\begin{equation}\label{CurveSepTube}
	\begin{split}
	\mathbb{P}^{0,T,\vec{x},\vec{y}}_{H^{\mathrm{RW}}}(A) &\geq \mathbb{P}^{0,T,\vec{x},\vec{y}}_{H^{\mathrm{RW}}}\left(\bigcap_{i=1}^{k-1} \bigg\{\sup_{s\in[0,T]} |L_i(s)-x_i-(z_i/T)s| < C\sqrt{T}/2\bigg\}\right)\\
	&= \prod_{i=1}^{k-1} \mathbb{P}^{0,T,0,z_i}_{H^{\mathrm{RW}}}\bigg(\sup_{s\in[0,T]}|\ell_i(s)-(z_i/T)s| < C\sqrt{T}/2\bigg).
	\end{split}
	\end{equation}
	Now by Assumption \ref{KMT}, there exist probability spaces with measures $\mathbb{P}_i$ for $i\in\llbracket 1,k-1\rrbracket$, supporting random variables $\ell^{(T,z_i)}$ with laws $\mathbb{P}^{0,T,0,z_i}_{H^{\mathrm{RW}}}$, as well as Brownian bridges $B^{\sigma,i}$ each with variance $\sigma^2 = \sigma^2_p$ depending only on $p$. We have
	\begin{equation}\label{CurveSepCheb}
	\begin{split}
	\mathbb{P}^{0,T,0,z_i}_{H^{\mathrm{RW}}}\bigg(\sup_{s\in[0,T]}|\ell_i(s)-(z_i/T)s| < C\sqrt{T}/2\bigg) &\geq \mathbb{P}_i\bigg(\sup_{s\in[0,T]} |B^{\sigma,i}_{s/T}| < C/4\bigg)\\
	&\qquad\quad - \mathbb{P}_i\left(\Delta(T,z_i) > C\sqrt{T}/4\right),
	\end{split}
	\end{equation}
	where $\Delta(T,z_i)$ is as defined in Assumption \ref{KMT}. Note that the probability on the right hand side of the first line is strictly positive by Lemma \ref{Spread}. Condition (3) in the hypothesis thus allows us to apply Corollary \ref{Cheb} to find $T_4$ depending on $p,C,K$ but not on $i$, so that 
	\[
	\mathbb{P}_i\left(\Delta(T,z_i) > C\sqrt{T}/4\right) \leq \frac{1}{2}\cdot\mathbb{P}_i\bigg(\sup_{s\in[0,T]} |B^{\sigma,i}_{s/T}| < C/4\bigg)
	\]
	for all $T\geq T_4$. It then follows from \eqref{CurveSepCheb} and Lemma \ref{BBmax} that
	\[
	\mathbb{P}^{0,T,0,z_i}_{H^{\mathrm{RW}}}\bigg(\sup_{s\in[0,T]}|\ell_i(s)-(z_i/T)s| < C\sqrt{T}/2\bigg) \geq \frac{1}{2} - \sum_{n=1}^\infty e^{-n^2C^2/8\sigma^2}
	\]
	for $T\geq T_4$, and the right hand side is positive. In view of \eqref{CurveSepTube}, we conclude \eqref{CurveSepBd}.
	
	Now suppose $C^2 \geq 8\sigma^2\log 3$. By \eqref{maxAbsBd}, the term on the right of the first line of \eqref{CurveSepCheb} is bounded below by $1-2\exp(-C^2/8\sigma^2)$. We can enlarge $T_4$ if necessary so that the probability on the second line is $\leq \exp(-C^2/8\sigma^2)$. The sum is then bounded below by $1-3\exp(-C^2/8\sigma^2)$, and the assumption on $C$ implies that this lower bound is positive. Now \eqref{CurveSepWeak} follows from \eqref{CurveSepTube}.
\end{proof}

\subsection{Properties of non-crossing $H^{\mathrm{RW}}$ line ensembles}\label{Section3.3}

In this section, we give several results about non-crossing $H^{\mathrm{RW}}$ line ensembles with the laws $\mathbb{P}^{T_0,T_1,\vec{x},\vec{y},f,g}_{\mathrm{avoid},H^{\mathrm{RW}};S}$ of Definition \ref{HRWavoiddef}. The first result gives a lower bound on the probability that a collection of non-crossing random walks on $[0,T]$ rises very high on scale $\sqrt{T}$.

\begin{lemma}\label{CurvesHigh}
	Fix $p\in(\alpha,\beta)$, $k\in\mathbb{N}$, $M,M_1 > 0$, and let $\sigma = \sigma_p$ be as in Assumption \ref{KMT} for this value of $p$. Then we can find $T_5\in\mathbb{N}$ depending on $p,k,M,M_1$ such that the following holds for all $T\geq T_5$. Let $\vec{x},\vec{y}\in\mathfrak{W}_k$ be such that $y_i - x_i \in \llbracket \alpha T,\beta T\rrbracket$ for $i=1,\dots,k$, and suppose $x_k \geq - M_1\sqrt{T}$, and $y_k \geq pT - M_1\sqrt{T}$. Then for any $S\subseteq \llbracket 0,T\rrbracket$ we have
	\begin{equation}\label{19ineq}
		\begin{split}
		&\mathbb{P}^{0,T,\vec{x},\vec{y}}_{\mathrm{avoid}, H^{\mathrm{RW}}; S}\left(Q_k(T/2) - pT/2 \geq M\sqrt{T}\right) \\
		\geq \; & \frac{2^{3k/2}}{\pi^{k/2}\sigma^k}\left(1 - 2\sum_{n=1}^\infty (-1)^{n-1} e^{-2n^2/\sigma^2}\right)^{2k}\exp\left(-
		2k(M+M_1+10k-4)^2/\sigma^2\right).
		\end{split}
	\end{equation}
\end{lemma}

\begin{proof}
	We first exploit monotone coupling to reduce to the case in which the entries of the data $\vec{x},\vec{y}$ are well-separated. Define vectors $\vec{x}\,',\vec{y}\,'\in\mathfrak{W}_k$ by
	\begin{align*}
		x_i' &= \lfloor - M_1\sqrt{T} \rfloor - 10(i-1)\lceil\sqrt{T}\rceil, \qquad
		y_i' = \lfloor pT - M_1\sqrt{T}\rfloor - 10(i-1)\lceil\sqrt{T}\rceil.
	\end{align*}
	Then $x_i'\leq x_k \leq x_i$ and $y_i' \leq y_k \leq y_i$ for $1\leq i\leq k-1$, so Lemma \ref{MCL} implies that
	\begin{equation*}
		\mathbb{P}^{0,T,\vec{x},\vec{y}}_{\mathrm{avoid}, H^{\mathrm{RW}}; S} \Big(Q_k(T/2) - pT/2 \geq M\sqrt{T}\Big) \geq \mathbb{P}^{0,T,\vec{x}\,',\vec{y}\,'}_{\mathrm{avoid}, H^{\mathrm{RW}}; S} \Big(Q_k(T/2) - pT/2 \geq M\sqrt{T}\Big).
	\end{equation*}
	We now introduce the constants
	\[
	K_i := pT/2 + M\sqrt{T}+(10(k-i)+5)\lceil\sqrt{T}\rceil
	\] 
	for $1\leq i\leq k$. This is the midpoint of the two points $pT/2 + M\sqrt{T} + 10(k-i)\lceil\sqrt{T}\rceil$ and $pT/2 + M\sqrt{T} + 10(k-i+1)\lceil\sqrt{T}\rceil$. Let $E$ denote the event that the following conditions hold for $1\leq i\leq k$:
	\begin{enumerate}[label=(\arabic*)]
		
		\item $\left| Q_i(T/2) - K_i \right| \leq 2\sqrt{T}$,
		
		\item $\sup_{s\in[0,T/2]} \Big|Q_i(s)-x_i'-\dfrac{K_i-x_i'}{T/2}\,s\Big| \leq 3\sqrt{T}$,
		
		\item $\sup_{s\in[T/2,T]} \Big|Q_i(s)-K_i-\dfrac{y_i'-K_i}{T/2}(s-T/2)\Big| \leq 3\sqrt{T}$.
		
	\end{enumerate}
	Note that $K_k - pT/2 - M\sqrt{T} \geq 5\sqrt{T}$, so the first condition implies in particular that $Q_k(T/2)-pT/2 \geq M\sqrt{T}$, and also that $Q_i(T/2)-Q_{i+1}(T/2)\geq 6\sqrt{T}$ for each $i$. The second and third conditions require that each curve $Q_i$ remain within a distance of $3\sqrt{T}$ of the graph of the piecewise linear function on $[0,T]$ passing through the points $(0,x_i')$, $(T/2,K_i)$, and $(T,y_i')$. We observe that
	\[
	\mathbb{P}^{0,T,\vec{x}\,',\vec{y}\,'}_{\mathrm{avoid}, H^{\mathrm{RW}}; S} \left(Q_k(T/2) - ptT \geq M\sqrt{T}\right) \geq \mathbb{P}^{0,T,\vec{x}\,',\vec{y}\,'}_{\mathrm{avoid}, H^{\mathrm{RW}}; S}(E) \geq \mathbb{P}^{0,T,\vec{x}\,',\vec{y}\,'}_{H^{\mathrm{RW}}}(\tilde{E}).
	\]
	Here we have written $\tilde{E}$ for the event defined in the same way as $E$ but with $(Q_1,\dots,Q_k)$ replaced by a line ensemble $(L_1,\dots,L_k)$ with law $\mathbb{P}^{0,T,\vec{x}\,',\vec{y}\,'}_{H^{\mathrm{RW}}}$. The second inequality follows since on the event $\tilde{E}$ we have $Q_1(s)\geq\cdots\geq Q_k(s)$ for all $s\in\llbracket 0,T \rrbracket$. Then we have
	\begin{equation}\label{19gibbs}
		\begin{split}
			\mathbb{P}^{0,T,\vec{x}\,',\vec{y}\,'}_{H^{\mathrm{RW}}}(\tilde{E}) &= \prod_{i=1}^k \mathbb{P}^{0,T,0,y_i'-x_i'}_{H^{\mathrm{RW}}}\bigg(\left|L_i(T/2) - K_i + x_i'\right|\leq 2\sqrt{T} \quad\mathrm{and}\\
			&\qquad\qquad \sup_{s\in[0,T/2]}\left|L_i(s) - \frac{K_i-x_i'}{T/2}\,s\right| \leq 3\sqrt{T}\quad\mathrm{and}\\
			&\qquad\qquad \sup_{s\in[T/2,T]}\left|L_i(s)-(K_i-x_i')-\frac{y_i'-K_i}{T/2}(s-T/2)\right| \leq 3\sqrt{T}\bigg).
		\end{split}
	\end{equation}
	Note that $y_i' - x_i'$ is independent of $i$; we will write $z = y_i' - x_i'$. Let $\tilde{\mathbb{P}}$ be a probability measure supporting a random variable $\ell^{(T,z)}$ with law $\mathbb{P}_{H^{\mathrm{RW}}}^{0,T,0,z}$ coupled with a Brownian bridge $B^\sigma$ with variance $\sigma^2$, as in Assumption \ref{KMT}. Note that $K_i - x_i'$ and $y_i' - K_i$ are also both independent of $i$. Writing $M_2 := M+M_1+10k-5$, we see that each of the $k$ factors in the product in \eqref{19gibbs} is bounded below for large $T$ by
	\begin{equation}\label{19BB}
		\begin{split}
			& \mathbb{P}^{0,T,0,z}_{H^{\mathrm{RW}}}\bigg(\left|\ell(T/2)-pT/2-M_2\sqrt{T}\right|\leq 2\sqrt{T} - 10k + 5\quad\mathrm{and}\\
			&\qquad \sup_{s\in[0,T/2]}\left|\ell(s)-ps-\frac{M_2}{\sqrt{T}/2}\,s\right| \leq 3\sqrt{T} - 1 \quad\mathrm{and}\\
			&\qquad \sup_{s\in[T/2,T]}\left|\ell(s)-ps-M_2\sqrt{T}+\frac{M_2}{\sqrt{T}/2}(s-T/2)\right| \leq 3\sqrt{T} - 1 \bigg)\\
			\geq \; & \tilde{\mathbb{P}}\bigg(\left|\sqrt{T}\,B^\sigma_{1/2} - M_2\sqrt{T}\right|\leq \sqrt{T} \quad\mathrm{and}\\
			&\qquad\sup_{s\in[0,T/2]}\left|\sqrt{T}\,B^\sigma_{s/T}-M_2\sqrt{T}\cdot\frac{s}{T/2}\right| \leq 2\sqrt{T}\quad\mathrm{and}\\
			&\qquad \sup_{s\in[T/2,T]}\left|\sqrt{T}\,B^\sigma_{s/T}-M_2\sqrt{T}\cdot\frac{T-s}{T/2}\right| \leq 2\sqrt{T} \bigg) -  \tilde{\mathbb{P}}\left(\Delta(T,z) > \sqrt{T}/2\right).
		\end{split}
	\end{equation}
	Here $\Delta(T,z)$ is as in Assumption \ref{KMT}. Note that $B^\sigma_{1/2}$ is a centered Gaussian random variable with variance $\sigma^2/4 = \sigma^2(1/2)(1-1/2)$. Writing $\xi = B^\sigma_{1/2}$, it follows from Lemma \ref{2bridges} that there exist independent Brownian bridges $B^1,B^2$ with variance $\sigma^2/2$ so that $B^\sigma_s$ has the same law as $\frac{s}{T/2}\xi + B^1_{2s/T}$ for $s\in[0,T/2]$ and $\frac{T-s}{T/2}\xi + B^2_{(2s-T)/T}$ for $s\in[T/2,T]$. The first term in the last expression in \eqref{19BB} is thus equal to
	\begin{equation*}
		\begin{split}
			&\tilde{\mathbb{P}}\bigg(|\xi - M_2|\leq 1 \quad\mathrm{and} \sup_{s\in[0,T/2]}\left|B^1_{s/T}-(M_2-\xi)\cdot\frac{s}{T/2}\right| \leq 2 \quad \mathrm{and}\\
			&\qquad  \sup_{s\in[T/2,T]}\left|B^2_{(2s-T)/T}-(M_2-\xi)\cdot\frac{T-s}{T/2}\right| \leq 2 \bigg) \\
			\geq \; & \tilde{\mathbb{P}}\bigg(|\xi - M_2|\leq 1 \quad\mathrm{and} \sup_{s\in[0,T/2]}\big|B^1_{2s/T}\big| \leq 1 \qquad \mathrm{and}\quad \sup_{s\in[T/2,T]}\big|B^2_{(2s-T)/T}\big| \leq 1 \bigg) \\
			= \; & \tilde{\mathbb{P}}\Big(|\xi-M_2|\leq 1\Big)\cdot \tilde{\mathbb{P}}\bigg(\sup_{s\in[0,T/2]} \big|B^1_{2s/T}\big|\leq 1\bigg)\cdot \tilde{\mathbb{P}}\bigg(\sup_{s\in[0,T/2]} \big|B^2_{(2s-T)/T}\big|\leq 1\bigg) \\
			\geq \; & \left(1-2\sum_{n=1}^\infty (-1)^{n-1}e^{-2n^2/\sigma^2}\right)^2 \int_{M_2-1}^{M_2+1} \frac{e^{-2\xi^2/\sigma^2}}{\sqrt{\pi \sigma^2/2}}\,d\xi \\
			\geq \; & \left(1-2\sum_{n=1}^\infty (-1)^{n-1}e^{-2n^2/\sigma^2}\right)^2\cdot\frac{2\sqrt{2}\,e^{-2(M_2+1)^2/\sigma^2}}{\sqrt{\pi}\,\sigma}.
		\end{split}
	\end{equation*}
	In the fourth line, we used the fact that $\xi$, $B^1_\cdot$, and $B^2_\cdot$ are independent, and in the last line we used Lemma \ref{BBmax}. Since $|z-pT|\leq (M_1+1)\sqrt{T}$, Lemma \ref{Cheb} allows us to choose $T$ large enough so that $\tilde{\mathbb{P}}(\Delta(T,z) > \sqrt{T}/2)$ is less than 1/2 the expression on the right of the last line. Then in view of \eqref{19gibbs} and \eqref{19BB}, and recalling the definition of $M_2$, we conclude \eqref{19ineq}.
	
\end{proof}

We now state an important weak convergence result for scaled non-crossing $H^{\mathrm{RW}}$ line ensembles. This result will be used in the proofs of both Theorems \ref{GoodTight} and \ref{GoodBGP}. Roughly, the following lemma states that if the data of a sequence of non-crossing $H^{\mathrm{RW}}$ line ensembles converge, then after appropriate scaling, the sequence of line ensembles will converge weakly to avoiding Brownian bridges with the limiting data. We state and prove the result for a general random walk Hamiltonian, not necessarily the one fixed earlier in this section. We note that this lemma requires only Assumption \ref{KMT} on $H^{\mathrm{RW}}$.

\begin{lemma}\label{ScaledWeakConv}
	Fix $\alpha,\beta\in\mathbb{Z}$ with $\alpha < \beta$ and an $\llbracket\alpha,\beta\rrbracket$-supported random walk Hamiltonian $H^{\mathrm{RW}}$. Fix $p\in(\alpha,\beta)$, $t\in(0,1)$, $k\in\mathbb{N}$, $\vec{x},\vec{y}\in W_k^\circ$, and $a,b\in\mathbb{R}$ with $a<b$, and let $\sigma$ be the constant appearing in Assumption \ref{KMT}. For $T\in\mathbb{N}$, write $a_T = \lfloor aT\rfloor$, $b_T = \lceil bT\rceil$. Let $\vec{x}\,^T,\vec{y}\,^T\in\mathfrak{W}_k$ be sequences such that for each $i\in\llbracket 1,k\rrbracket$,
	\[
	\frac{x_i^T - pa_T}{\sigma\sqrt{T}} \longrightarrow x_i, \quad \frac{y_i^T - pb_T}{\sigma\sqrt{T}} \longrightarrow y_i
	\]
	as $T\to\infty$. Let $f : [a-1,b+1]\to(-\infty,\infty]$ and $g : [a-1,b+1]\to[-\infty,\infty)$ be continuous functions such that $f(t) > g(t)$ for all $t\in[a-1,b+1]$, $f(a) > x_1$, $f(b) > y_1$, $g(a) < x_k$, and $g(b) < y_k$. Let $f_T : [(a-1)T,(b+1)T]\to(-\infty,\infty]$ and $g_T : [(a-1)T,(b+1)T]\to[-\infty,\infty)$ be sequences of continuous functions such that
	\[
	\frac{f_T(tT) - ptT}{\sigma\sqrt{T}} \longrightarrow f(t), \quad \frac{g_T(tT) - ptT}{\sigma\sqrt{T}} \longrightarrow g(t)
	\]
	uniformly in $t\in[a-1,b+1]$ as $T\to\infty$.
	
	Then we can find $T_6 \in\mathbb{N}$ so that the measure $\mathbb{P}^{a_T,b_T,\vec{x}\,^T,\vec{y}\,^T,f_T,g_T}_{\mathrm{avoid},H^{\mathrm{RW}}}$ is well defined for $T\geq T_6$. Moreover, if $Q^T$ is a sequence of random variables with these laws and $Z^T$ are the $C([a,b])^k$-valued random variables defined by
	\[
	Z^T(t) = \left(\frac{Q_1^T(tT) - ptT}{\sigma\sqrt{T}}, \dots, \frac{Q_k^T(tT) - ptT}{\sigma\sqrt{T}}\right), \quad t\in[a,b],
	\]
	then the law of $Z^T$ converges weakly to $\mathbb{P}^{a,b,\vec{x},\vec{y},f,g}_{\mathrm{avoid}}$ as $T\to\infty$.
\end{lemma}

\begin{remark}\label{Bernoulli}
	We will use the special case of Lemma \ref{ScaledWeakConv} of symmetric Bernoulli random walks with $p = 1/2$, i.e., $\alpha = 0$, $\beta = 1$, and $\exp(-H^{\mathrm{RW}}(0)) = \exp(-H^{\mathrm{RW}}(1)) = 1/2$. In this case, the constant $\sigma$ is given by $\sigma^2 = \frac{1}{2}(1-\frac{1}{2}) = 1/4$ (see Example \ref{ex2}).
\end{remark}

\begin{proof}
	The proof of this lemma will rely on the KMT coupling provided by Assumption \ref{KMT}. We split the argument into three steps for clarity.\\
	
	\noindent\textbf{Step 1.} We first prove that for each $i\in\llbracket 1,k\rrbracket$, if $\ell_i^T$ have laws $\mathbb{P}^{a_T,b_T,x_i^T,y_i^T}_{H^{\mathrm{RW}}}$ and if the independent $C([a,b])$-valued random variables $Y_i^T$ are defined by
	\[
	Y^T_i(t) = \frac{\ell_i^T(tT) - ptT}{\sigma\sqrt{T}}\, \quad t\in[a,b],
	\] 
	then $Y_i^T$ converges weakly to $\mathbb{P}^{a,b,x_i,y_i}_{\mathrm{free}}$. Independence then implies that $Y^T := (Y_1^T,\dots,Y_k^T)$ converges weakly to $\mathbb{P}^{a,b,\vec{x},\vec{y}}_{\mathrm{free}}$, i.e., to $k$ independent Brownian bridges. Note that the $Y_i^T$ are well-defined for large enough $T$. Indeed, we have $(y_i^T-x_i^T-p(b_T-a_T))/\sigma\sqrt{T} \to y_i - x_i$ and $\sqrt{T} = o(b_T-a_T)$. Since $\alpha < p < \beta$, we can find $T_{60}$ so that $\alpha(b_T-a_T) \leq y_i^T-x_i^T \leq \beta(b_T-a_T)$ for all $T\geq T_{60}$.
	
	Let $B$ denote a standard Brownian bridge on $[0,1]$, and define $C([a,b])$-valued random variables $\tilde{B}^T, \tilde{B}$ via
	\begin{align*}
		\tilde{B}^T(t) &= \sqrt{\frac{b_T-a_T}{T}}\cdot B\left(\frac{tT-a_T}{b_T-a_T}\right) + \frac{tT-a_T}{b_T-a_T}\cdot\frac{x_i^T-pa_T}{\sqrt{T}} + \frac{b_T-tT}{b_T-a_T}\cdot\frac{y_i^T-pb_T}{\sqrt{T}},\\
		\tilde{B}(t) &= \sqrt{b-a}\cdot B\left(\frac{t-a}{b-a}\right) + \frac{t-a}{b-a}\cdot x_i + \frac{b-t}{b-a}\cdot y_i.
	\end{align*}
	We note that $\tilde{B}$ has law $\mathbb{P}^{a,b,x_i,y_i}_{\mathrm{free}}$. Since $b_T/T \to b$ and $a_T/T\to a$, our assumptions on $x_i^T,y_i^T$ imply that $\tilde{B}^T \implies \tilde{B}$ as $T\to\infty$. Now by \cite[Theorem 3.1]{Billing}, to show that $Y_i^T \implies \tilde{B}$ as $T\to\infty$, it suffices to find a sequence of probability spaces supporting $Y_i^T$ and $\tilde{B}^T$ such that
	\begin{equation}\label{rhoWeak}
		\rho(Y_i^T,\tilde{B}^T) := \sup_{t\in[a,b]} \left|Y_i^T(t) - \tilde{B}^T(t)\right| \implies 0 \quad \mbox{as } T\to\infty.
	\end{equation}
	Let us write
	\[
	\Delta_T = \sup_{t\in[aT,bT]}\left|\sigma\sqrt{T}\,\tilde{B}^T(t/T)-\sigma\sqrt{T}\,Y_i^T(t/T)\right| = \sigma\sqrt{T}\,\rho(Y_i^T,\tilde{B}^T).
	\]
	Since $\sigma B$ is a Brownian bridge with variance $\sigma$, Assumption \ref{KMT} gives us for each $T$ a probability space supporting $\tilde{B}^T$ and $Y_i^T$ as well as constants $0 < C,A,a < \infty$ such that
	\[
	\ex\left[e^{a\Delta_T}\right] \leq Ce^{A\log((b-a)T)} e^{|y_i^T-x_i^T-p(b-a)T|^2/T}
	\]
	We have $(x_i^T-paT)/\sqrt{T}\to \sigma x_i$ and $(y_i^T-pbT)/\sqrt{T}\to \sigma y_i$, so there exists $T_0\in\mathbb{N}$ so that $|y_i^T-x_i^T-p(b-a)T| \leq \sigma(y_i-x_i+1)\sqrt{T}$ for $T\geq T_0$. Chebyshev's inequality now implies that for any $\epsilon>0$ and $T\geq T_0$ we have
	\begin{align*}
	\mathbb{P}\left(\rho(Y_i^T,\tilde{B}^T) > \epsilon\right) &= \mathbb{P}\left(a\Delta_T > a\sigma\epsilon\sqrt{T}\right) \leq e^{-a\sigma\epsilon\sqrt{T}}\,\ex\left[e^{a\Delta_T}\right] \\
	&\leq C\exp\left[-a\sigma\epsilon\sqrt{T} + A\log T + A\log(b-a)+\sigma^2(y_i-x_i+1)^2\right].
	\end{align*}
	The last expression goes to 0 as $T\to\infty$, proving \eqref{rhoWeak} and completing this step.\\
	
	\noindent\textbf{Step 2.} Next, we fix $T_6$ as in the statement of the lemma, so that $\mathbb{P}^{a_T,b_T,\vec{x}\,^T,\vec{y}\,^T,f_T,g_T}_{\mathrm{avoid},H^{\mathrm{RW}}}$ is well defined for $T\geq T_6$. Observe that there exist $\epsilon > 0$ and continuous functions $h_1,\dots,h_k : [a,b]\to\mathbb{R}$ for $i\in\llbracket 1,k\rrbracket$ depending on $a,b,\vec{x},\vec{y},f,g$, such that $h_i(a) = x_i$, $h_i(b)=y_i$, and if $u_i:[a,b]\to\mathbb{R}$ are continuous functions with $\rho(u_i,h_i) := \sup_{x\in[a,b]} |u_i(x)-h_i(x)| < \epsilon$, then
	\begin{equation}\label{fuepsilon}
		f(x) - \epsilon > u_1(x) + \epsilon > u_1(x) - \epsilon > \cdots > u_k(x) + \epsilon > u_k(x) - \epsilon > g(x) + \epsilon
	\end{equation}
	for all $x\in[a,b]$. By Lemma \ref{Spread}, we have
	\begin{equation*}
		\mathbb{P}^{a,b,\vec{x},\vec{y}}_{\mathrm{free}}\left(\rho(\mathcal{Q}_i,h_i) < \epsilon \mbox{ for } i\in\llbracket 1,k\rrbracket\right) > 0.
	\end{equation*}
	The random variables $Y^T$ defined in Step 1 for $T\geq T_{60}$ converge weakly to $\mathbb{P}^{a,b,\vec{x},\vec{y}}_{\mathrm{free}}$, so we can find $T_{61}\geq T_{60}$ so that if $T\geq T_{61}$ then
	\begin{equation}\label{HRWwindow}
	\mathbb{P}^{a_T,b_T,\vec{x}\,^T,\vec{y}\,^T}_{H^{\mathrm{RW}}}\left(\rho(Y^T_i,h_i) < \epsilon \mbox{ for } i\in\llbracket 1,k\rrbracket\right) > 0.
	\end{equation}
	We now choose $T_{62}$ so that 
	\[
	\sup_{t\in[a_T/T,b_T/T]}\left|\frac{f_T(tT)-ptT}{\sigma\sqrt{T}} - f(t)\right| < \epsilon/4, \quad \sup_{t\in[a_T/T,b_T/T]}\left|\frac{g_T(tT)-ptT}{\sigma\sqrt{T}} - g(t)\right| < \epsilon/4
	\] 
	for $T\geq T_{62}$. If $f=\infty$ (resp. $g=-\infty$), we interpret this to mean that $f_N=\infty$ (resp. $g_N=-\infty$). We take $T_{63}$ large enough so that if $T\geq T_{63}$ and $|x-y|\leq 1/\sqrt{T}$ then $|f(x)-f(y)|<\epsilon/4$ and $|g(x)-g(y)|<\epsilon/4$. Lastly, we choose $T_{64}$ so that $1/\sqrt{T_{64}} < \epsilon/4$. Then \eqref{fuepsilon} implies that for $T\geq T_6 := \max(T_{61},T_{62},T_{63},T_{64})$, we have
	\begin{align*}
	&\{\rho(\mathcal{Q}_i,h_i) > \epsilon \mbox{ for } i\in\llbracket 1,k\rrbracket\}\\ 
	\subset\,&\left\{\frac{f_T(tT) - ptT}{\sigma\sqrt{T}} \geq Y^T_1(t) \geq \cdots \geq Y^T_k(t) \geq \frac{g_T(tT) - ptT}{\sigma\sqrt{T}} \mbox{ for } t\in[a_T/T,b_T/T]\right\}\\
	=\,&\left\{f_T(t) \geq \ell_1^T(t) \geq \cdots \geq \ell_k^T(t) \geq g_T(t) \mbox{ for } t\in[a_T,b_T]\right\},
	\end{align*}
	where $\ell_i^T$ have laws $\mathbb{P}^{a_T,b_T,x_i^T,y_i^T}_{H^{\mathrm{RW}}}$ as in Step 1. In combination with \eqref{HRWwindow}, this implies that $\mathbb{P}^{a_T,b_T,\vec{x}\,^T,\vec{y}\,^T,f_T,g_T}_{\mathrm{avoid},H^{\mathrm{RW}}}$ is well-defined.\\
	
	\noindent\textbf{Step 3.} We now complete the proof by showing that $Z^T \implies \mathcal{Z}$, where $\mathcal{Z}$ has law $\mathbb{P}^{a,b,\vec{x},\vec{y},f,g}_{\mathrm{avoid}}$. We write $\Sigma = \llbracket 1,k\rrbracket$, $\Lambda = [a,b]$, and $\Lambda_T = [a_T,b_T]$. By \cite[Theorem 2.1]{Billing}, it suffices to show that for any bounded continuous function $F : C(\Sigma\times\Lambda)\to\mathbb{R}$ we have
	\begin{equation}\label{scaledavoidweak}
		\lim_{T\to\infty} \ex[F(Z^T)] = \ex[F(\mathcal{Z})].
	\end{equation}
	We define the functions $\chi,\chi^T:C(\Sigma\times\Lambda)\to\mathbb{R}$ by
	\begin{align*}
		\chi(\mathcal{L}) &= \mathbf{1}\{f(t) > \mathcal{L}_1(t) > \cdots > \mathcal{L}_k(t) > g(t) \mbox{ for } t\in\Lambda\},\\
		\chi^T(L) &= \mathbf{1}\left\{\frac{f_T(tT)-ptT}{\sigma\sqrt{T}} \geq L_1(t) \geq \cdots \geq L_k(t) \geq \frac{g_T(tT)-ptT}{\sigma\sqrt{T}} \mbox{ for } t\in\Lambda\right\}.
	\end{align*}
	Then we observe that for $T\geq T_6$,
	\begin{equation}\label{scaledavoidcond}
		\ex[F(Z^T)] = \frac{\ex[F(Y^T)\chi^T(Y^T)]}{\ex[\chi^T(Y^T)]},
	\end{equation}
	where $Y^T$ is as in Step 1. By our choice of $T_6$ in Step 2, the denominator in \eqref{scaledavoidcond} is positive for all $T\geq T_6$. Similarly, we have
	\begin{equation}\label{BBavoidcond}
		\ex[F(\mathcal{Z})] = \frac{\ex[F(\mathcal{B})\chi(\mathcal{B})]}{\ex[\chi(\mathcal{L})]}, 
	\end{equation}
	where $\mathcal{B}$ has law $\mathbb{P}^{a,b,\vec{x},\vec{y}}_{\mathrm{free}}$. From \eqref{scaledavoidcond} and \eqref{BBavoidcond}, we see that to prove \eqref{scaledavoidweak} it suffices to show that for any bounded continuous function $F:C(\Sigma\times\Lambda)\to\mathbb{R}$, we have
	\begin{equation}\label{scaledBBex}
		\lim_{T\to\infty}\ex[F(Y^T)\chi^T(Y^T)] = \ex[F(\mathcal{B})\chi(\mathcal{B})].
	\end{equation}
	By Step 1, $Y^T \implies \mathcal{B}$ as $T\to\infty$. The Skorohod representation theorem \cite[Theorem 6.7]{Billing} gives a probability space $(\Omega,\mathcal{F},\mathbb{P})$ supporting $C(\Sigma\times\Lambda)$-valued random variables $\mathcal{Y}^T$ with the laws of $Y^T$ and a $C(\Sigma\times\Lambda)$-valued random variable $\mathcal{L}$ with law $\mathbb{P}^{a,b,\vec{x},\vec{y}}_{\mathrm{free}}$, such that $\mathcal{Y}^T(\omega) \to \mathcal{L}(\omega)$ uniformly on compact sets as $T\to\infty$ for all $\omega\in\Omega$. Define the events
	\begin{align*}
		A &= \{\omega\in\Omega : f > \mathcal{L}_1(\omega) > \cdots > \mathcal{L}_k(\omega) > g \mbox{ on } [a,b]\},\\
		C &= \{\omega\in\Omega : \mathcal{L}_i(\omega)(r) < \mathcal{L}_{i+1}(\omega)(r) \mbox{ for some } i\in\llbracket 0,k\rrbracket \mbox{ and } r\in[a,b]\},
	\end{align*}
	where in the definition of $E_2$ we use the convention $\mathcal{L}_0 = f$, $\mathcal{L}_{k+1} = g$. The continuity of $F$ and the assumptions on $f_T,g_T$ imply that $F(\mathcal{Y}^T)\chi^T(\mathcal{Y}^T) \to F(\mathcal{L})$ on the event $A$, and $F(\mathcal{Y}^T)\chi^T(\mathcal{Y}^T)\to 0$ on the event $C$. By Lemma \ref{NoTouch} we have $\mathbb{P}(A\cup C) = 1$, so $\mathbb{P}$-a.s. we have $F(\mathcal{Y}^T)\chi^T(\mathcal{Y}^T) \to F(\mathcal{L})\chi(\mathcal{L})$. The bounded convergence theorem then implies \eqref{scaledBBex}, completing the proof of \eqref{scaledavoidweak}.
	
\end{proof}

The next lemma gives a weak convergence result for the sequence $Z^T(t_T)$ along a sequence $t_T$ converging to a fixed interior time $t$, where $Z^T$ is as in Lemma \ref{ScaledWeakConv}. In particular, the limiting random vector almost surely has distinct components. We emphasize that this lemma applies even if $\vec{x}$ or $\vec{y}$ lie in the \textit{closed} Weyl chamber, $\overline{W}_k$, i.e., have repeated components. Thus even if the curves collide at the endpoints in the limit, they will almost surely remain separated in the interior. In \cite{Ber}, the proof of the analogous result (which only dealt with $Z^T(t)$ at a fixed $t$) required a lengthy computation of the exact limiting density $\rho$ appearing in the statement below for the case of non-crossing Bernoulli random walk bridges (see Section 8). This computation does not easily generalize to the broader setting in which we work here. To circumvent this difficulty, we exploit the weak convergence result of Lemma \ref{ScaledWeakConv}, both for non-crossing $H^{\mathrm{RW}}$ line ensembles and for non-crossing Bernoulli line ensembles, to reduce the problem to the case dealt with in \cite{Ber} (see Remark \ref{Bernoulli}).

\begin{lemma}\label{SliceWeakConv}
	Fix $p\in(\alpha,\beta)$, $t\in(0,1)$, $k\in\mathbb{N}$, and $\vec{x},\vec{y}\in\overline{W}_k$. Let $\vec{x}\,^T,\vec{y}\,^T\in\mathfrak{W}_k$ be sequences such that for each $i\in\llbracket 1,k\rrbracket$,
	\[
	\frac{x_i^T}{\sigma\sqrt{T}} \longrightarrow x_i, \quad \frac{y_i^T - pT}{\sigma\sqrt{T}} \longrightarrow y_i
	\]
	as $T\to\infty$. By Lemma \ref{ScaledWeakConv}, the measure $\mathbb{P}^{0,T,\vec{x}\,^T,\vec{y}\,^T}_{\mathrm{avoid},H^{\mathrm{RW}}}$ is well defined for sufficiently large $T$. Let $Q^T$ be a sequence of random variables with these laws and let $Z^T$ be the $C([0,1])^k$-valued random variable appearing in Lemma \ref{ScaledWeakConv} $\mathrm{(}$with $a=0, b=1\mathrm{)}$. Let $(t_T)$ be a sequence in $(0,1)$ converging to $t$. Then as $T\to\infty$, the sequence of random vectors $Z^T(t_T)$ converges weakly to a random vector $Z$ in $\mathbb{R}^k$ with a probability density $\rho$ supported on $W_k^\circ$.
\end{lemma}

\begin{proof}
	We first introduce small perturbations of the vectors $\vec{x},\vec{y}$ that lie in $W_k^\circ$, so that Lemma \ref{ScaledWeakConv} may be applied. For $\epsilon > 0$, we define $\vec{x}_{\epsilon+},\vec{y}_{\epsilon+}\in \mathbb{R}^k$ by
	\[
	(x_{\epsilon+})_i = x_i + (k-i)\epsilon, \quad (y_{\epsilon+})_i = y_i + (k-i)\epsilon.
	\]
	We have $(x_{\epsilon+})_i - (x_{\epsilon+})_{i+1} = x_i - x_{i+1} + \epsilon \geq \epsilon > 0$, so $\vec{x}_{\epsilon+}\in W_k^\circ$, and similarly $\vec{y}_{\epsilon+}\in W_k^\circ$. Likewise, we define $\vec{x}_{\epsilon-},\vec{y}_{\epsilon-}\in W_k^\circ$ via
	\[
	(x_{\epsilon-})_i = x_i - (i-1)\epsilon, \quad (y_{\epsilon-})_i = y_i - (i-1)\epsilon.
	\]
	The key observations are that $(x_{\epsilon-})_i \leq x_i \leq (x_{\epsilon+})_i$ for $i\in\llbracket 1,k\rrbracket$, $\vec{x}_{\epsilon-}\to \vec{x}$ and $\vec{x}_{\epsilon+}\to \vec{x}$ as $\epsilon\to 0$, and similarly for $\vec{y}_{\epsilon-},\vec{y}_{\epsilon+}$. We also let $\vec{x}_{\epsilon\pm}^T,\vec{y}_{\epsilon\pm}^T\in\mathfrak{W}_k$ be sequences such that
	\[
	\frac{(x_{\epsilon\pm}^T)_i}{\sigma\sqrt{T}} \longrightarrow (x_{\epsilon\pm})_i, \quad \frac{(y_{\epsilon\pm}^T)_i - pT}{\sigma\sqrt{T}} \longrightarrow (y_{\epsilon\pm})_i.
	\]
	We write $Q^T_{\epsilon+},Q^T_{\epsilon-}$ for the random variables with laws $\mathbb{P}^{0,T,\vec{x}_{\epsilon\pm}^T,\vec{y}_{\epsilon\pm}^T}_{\mathrm{avoid},H^{\mathrm{RW}}}$, and we let 
	\[
	Z^T_{\epsilon\pm}(x) = \left(\frac{(Q^T_{\epsilon\pm})_1(xT) - pxT}{\sigma\sqrt{T}},\dots,\frac{(Q^T_{\epsilon\pm})_k(xT) - pxT}{\sigma\sqrt{T}}\right), \quad x\in[0,1].
	\]
	Then Lemma \ref{ScaledWeakConv} implies that
	\[
	Z^T_{\epsilon\pm}\implies \mathbb{P}^{0,1,\vec{x}_{\epsilon\pm},\vec{y}_{\epsilon\pm}}_{\mathrm{avoid}} \quad \mathrm{as} \quad T\to\infty.
	\]
	
	We now exploit the result for Bernoulli line ensembles from \cite{Ber}. Let $B$ denote the $\llbracket 0,1\rrbracket$-supported random walk Hamiltonian with $\exp(-B(0)) = \exp(-B(1)) = 1/2$, and take $p=1/2$. As explained in Remark \ref{Bernoulli}, the constant $\sigma$ in this case is $1/2$. Let $\vec{z}_{\epsilon\pm}^T,\vec{w}_{\epsilon\pm}^T\in\mathfrak{W}_k$ be sequences such that
	\[
	\frac{(z_{\epsilon\pm}^T)_i}{\sqrt{T}/2} \longrightarrow (x_{\epsilon\pm})_i, \quad \frac{(w_{\epsilon\pm}^T)_i - T/2}{\sqrt{T}/2} \longrightarrow (y_{\epsilon\pm})_i,
	\]
	and let $L^T,L_{\epsilon\pm}^T$ denote three sequences of random variables with laws $\mathbb{P}^{0,T,\vec{x}\,^T, \vec{y}\,^T}_{\mathrm{avoid},B}$ and $\mathbb{P}^{0,T,\vec{z}_{\epsilon\pm}^T,\vec{w}_{\epsilon\pm}^T}_{\mathrm{avoid},B}$ respectively. Lastly, define 
	\begin{align*}
	Y^T(x) &= \left(\frac{L_1(xT)-xT/2}{\sqrt{T}/2}, \dots, \frac{L_k(xT)-xT/2}{\sqrt{T}/2}\right),\\
	Y^T_{\epsilon\pm}(x) &=  \left(\frac{(L_{\epsilon\pm}^T)_1(xT) - xT/2}{\sqrt{T}/2}, \dots, \frac{(L_{\epsilon\pm}^T)_k(xT) - xT/2}{\sqrt{T}/2}\right), \quad x\in[0,1].
	\end{align*}
	By Lemma \ref{ScaledWeakConv}, $Y_{\epsilon\pm}^T$ also converge weakly to $\mathbb{P}^{0,1,\vec{x}_{\epsilon\pm},\vec{y}_{\epsilon\pm}}_{\mathrm{avoid}}$ as $T\to\infty$, the same limit as $Z_{\epsilon\pm}^T$. Let us denote by $Z_{\epsilon\pm}$ random variables with the laws of these two limits. Now let $t_T,t\in (0,1)$ and $t_T\to t$ as in the statement of the lemma. We claim that
	\begin{equation}\label{YZtT}
	Y^T_{\epsilon\pm}(t_T) \implies Z_{\epsilon\pm}(t) \quad \mathrm{and} \quad Z^T_{\epsilon\pm}(t_T) \implies Z_{\epsilon\pm}(t) 
	\end{equation}
	as $T\to\infty$. To see this, note that since $Y_{\epsilon\pm}^T \implies Z_{\epsilon\pm}$ in the topology of uniform convergence on compact sets, the Skorohod representation theorem allows us to find a probability space with measure $\tilde{\mathbb{P}}$ supporting random variables with the laws of $Y_{\epsilon\pm}^T, Z_{\epsilon\pm}^T$ (we use the same notation for these random variables for simplicity), such that $Y_{\epsilon\pm}^T, Z_{\epsilon\pm}^T \to Z_{\epsilon\pm}$ uniformly on $[0,1]$, $\tilde{\mathbb{P}}$-a.s. Then since the $Y_{\epsilon\pm}^T, Z_{\epsilon\pm}^T$ are continuous, $t_T\to t$, and the convergence is uniform, we have $Y_{\epsilon\pm}^T(t_T), Z_{\epsilon\pm}^T(t_T) \longrightarrow Z_{\epsilon\pm}(t)$, $\tilde{\mathbb{P}}$-a.s. This implies \eqref{YZtT}.
	
	Now the important technical result \cite[Proposition 8.2]{Ber} shows that the random vectors $2Y^T(t)$ and $2Y^T_{\epsilon\pm}(t)$ converge weakly to random vectors with some probability densities $\rho^B,\rho^B_{\epsilon\pm}$ supported on $W_k^\circ$; precise formulas can be found in \cite[Section 8.1]{Ber}, but for our purposes it is sufficient to know the supports. It follows that $Y^T(t_T)$ and $Y^T_{\epsilon\pm}(t_T)$ have limiting densities $\rho,\rho_{\epsilon\pm}$ given by 
	\begin{equation}\label{densities}
		\rho(x) = 2\rho^B(2x), \quad \rho_{\epsilon\pm}(x) = 2\rho^B_{\epsilon\pm}(2x),
	\end{equation} 
	which are also supported on $W_k^\circ$. 
	
	We now show that $Z^T(t_T)$ also converges weakly to a random vector with this density $\rho$. It suffices to prove that for any rectangle of the form $R = [-\infty,a_1)\times\cdots\times[-\infty,a_k)$, we have
	\begin{equation}\label{rectrho}
	\lim_{T\to\infty} \mathbb{P}\left(Z^T(t_T) \in R\right) = \int_R \rho(x)\,dx.
	\end{equation}
	Lemma \ref{MCL} implies that
	\[
		\mathbb{P}\left(Z^T_{\epsilon+}(t_T)\in R\right) \leq \mathbb{P}\left(Z^T(t_T)\in R\right) \leq \mathbb{P}\left(Z^T_{\epsilon-}(t_T)\in R\right).
	\]
	By \eqref{YZtT}, $Z^T_{\epsilon\pm}(t_T)$ have the same weak limits as $Y^T_{\epsilon\pm}(t_T)$, so letting $T\to\infty$ gives
	\begin{equation}\label{rhosqueeze}
	\int_R \rho_{\epsilon+}(x)\,dx \leq \liminf_{T\to\infty} \mathbb{P}\left(Z^T(t_T)\in R\right) \leq \limsup_{T\to\infty} \mathbb{P}\left(Z^T(t_T)\in R\right) \leq \int_R \rho_{\epsilon-}(x)\,dx.
	\end{equation}
	It is shown in \cite[Corollary 8.12]{Ber} that 
	\[
	\lim_{\epsilon\to 0+} \int_R \rho^B_{\epsilon+}(x)\,dx = \lim_{\epsilon\to 0+} \int_R \rho_{\epsilon-}^B(x)\,dx = \int_R \rho^B(x)\,dx.
	\]
	Thus using \eqref{densities} and letting $\epsilon\to 0$ in \eqref{rhosqueeze}, we obtain \eqref{rectrho}. This completes the proof.
	
\end{proof}

The final lemma of this section states roughly that if the entry and exit data of a sequence of non-crossing $H^{\mathrm{RW}}$ line ensembles on $[0,T]$ remain bounded on scale $\sqrt{T}$, then with high probability along any sequence of times converging to an interior time, the curves will in fact separate from one another by a small positive distance $\delta$ on scale $\sqrt{T}$.

\begin{lemma}\label{DeltaSep}
	Fix $p\in(\alpha,\beta)$, $t\in(0,1)$, $k\in\mathbb{N}$, and $M_1,M_2,\epsilon > 0$, and let $(t_T)$ be a sequence in $(0,1)$ with $t_T\to t$ as $T\to\infty$. Then we can find $T_7\in\mathbb{N}$ and $\delta > 0$ depending on $M_1,M_2,p,k,\epsilon,(t_T)$ so that the following holds for all $T\geq T_7$. Let $\vec{x},\vec{y}\in\mathfrak{W}_k$ be such that $y_i-x_i\in\llbracket\alpha T,\beta T\rrbracket$ for $i\in\llbracket 1,k\rrbracket$. Then if $|x_i| \leq M_1\sqrt{T}$ and $|y_i - pT| \leq M_2\sqrt{T}$ for $i\in\llbracket 1,k\rrbracket$, we have
	\[
	\mathbb{P}^{0,T,\vec{x},\vec{y}}_{\mathrm{avoid},H^{\mathrm{RW}}}\left(\min_{1\leq i\leq k-1} \left(Q_i(t_TT) - Q_{i+1}(t_TT)\right) < \delta\sqrt{T}\right) < \epsilon.
	\]
	Here, $(Q_1,\dots,Q_n)$ denotes a random variable with law $\mathbb{P}^{0,T,\vec{x},\vec{y}}_{\mathrm{avoid},H^{\mathrm{RW}}}$.
\end{lemma}

\begin{proof}
	Suppose the conclusion does not hold. Then we can find constants $M_1,M_2,\epsilon > 0$, sequences $T_n\to\infty$, $\delta_n\to 0$, and $\vec{x}\,^n,\vec{y}\,^n\in\mathfrak{W}_k$, and random variables $(Q_1^n,\dots,Q_k^n)$ with laws $\mathbb{P}^{0,T_n,\vec{x}\,^n,\vec{y}\,^n}_{\mathrm{avoid},H^{\mathrm{RW}}}$ under $\mathbb{P}$, such that
	\[
	\mathbb{P}\left(\min_{1\leq i\leq k-1} \left(Z_i^n - Z_{i+1}^n\right) < \delta_n\right) \geq\epsilon, \qquad Z_i^n := \frac{Q_i^n(t_{T_n}T_n) - pt_{T_n}T_n}{\sqrt{T_n}}.
	\]
	We have $|x_i^n/\sqrt{T_n}| \leq M_1$ and $|(y_i^n-pT_n)/\sqrt{T_n}| \leq M_2$ by assumption. Since $\vec{x}\,^n,\vec{y}\,^n\in\overline{W}_k$ for all $n$ and $\overline{W}_k\subset\mathbb{R}^k$ is closed, by passing to subsequences if necessary we can therefore assume that there exist $\vec{x},\vec{y}\in\overline{W}_k$ such that
	\[
	\frac{x_i^n}{\sqrt{T_n}} \longrightarrow x_i^n, \quad \frac{y_i^n - pT_n}{\sqrt{T_n}} \longrightarrow y_i^n
	\]
	as $n\to\infty$. Given $\delta > 0$, we can choose $N$ large enough so that $\delta_n \leq \delta$ for $n\geq N$. Then we have
	\[
	\epsilon \leq \limsup_{n\to\infty}\mathbb{P}\left(\min_{1\leq i\leq k-1} \left(Z_i^n - Z_{i+1}^n\right) \leq \delta\right).
	\]
	By Lemma \ref{SliceWeakConv}, $(Z_1^n,\dots,Z_k^n)$ converges weakly to a random vector $Z = (Z_1,\dots,Z_k)$ with a density $\rho$ supported on $W_k^\circ$ (the same density up to constant factors of $\sigma$); in particular, $Z_i - Z_{i+1} \geq 0$ a.s. The portmanteau theorem \cite[Theorem 3.2.11]{Durrett} in combination with the previous inequality implies that
	\begin{equation}\label{portmanteau}
	\epsilon \leq \mathbb{P}\left(\min_{1\leq i\leq k-1} \left(Z_i - Z_{i+1}\right) \leq \delta\right) \leq \sum_{i=1}^{k-1} \mathbb{P}\left(Z_i - Z_{i+1} \leq \delta\right).
	\end{equation}
	We will find a $\delta > 0$ for which this inequality cannot hold. For $\eta \geq 0$ and $i\in\llbracket 1,k-1\rrbracket$, define the sets
	\[
	E_i^\eta = \{(z_1,\dots,z_k)\in\mathbb{R}^k : 0 \leq z_i - z_{i+1} \leq \eta\}.
	\]
	Then for $\delta > 0$,
	\[
	\mathbb{P}(Z_i - Z_{i+1} \leq \delta) = \int_{\mathbb{R}^k} \rho\cdot\mathbf{1}_{E_i^\delta}\,dx_1\cdots dx_k.
	\]
	Note that $E_i^0 = \{\vec{z}\in\mathbb{R}^k : z_i = z_{i+1}\}$ is a $(k-1)$-dimensional subspace of $\mathbb{R}^k$, so it has Lebesgue measure 0. We see that $\rho\cdot\mathbf{1}_{E_i^\delta} \to \rho\cdot\mathbf{1}_{E_i^0}$ pointwise as $\delta\to 0$, so $\rho\cdot\mathbf{1}_{E_i^\delta} \to 0$ a.s. Since $|\rho\cdot\mathbf{1}_{E_i^\delta}| \leq \rho$ and $\rho$ is integrable, the dominated convergence theorem and the last equality imply that
	\[
	\mathbb{P}(Z_i - Z_{i+1} \leq \delta) \longrightarrow 0 \quad \mbox{as} \quad \delta\to 0.
	\]
	Since $\epsilon>0$ is fixed, this clearly contradicts \eqref{portmanteau} for sufficiently small $\delta > 0$.

\end{proof}

%% file: section4.tex
%
\section{Proof of Theorems \ref{GoodTight} and \ref{GoodBGP} }\label{Section4}

In this section we prove our main technical results, Theorems \ref{GoodTight} and \ref{GoodBGP}. Throughout, we assume that $k \in \mathbb{N}$ with $k \geq 2$, $\alpha,\beta\in\mathbb{Z}$ with $\alpha <\beta$, $p \in (\alpha,\beta)$, $\gamma,\lambda > 0$, $\theta\in[0,2)$, and an $\llbracket\alpha,\beta\rrbracket$-supported random walk Hamiltonian $H^{\mathrm{RW}}$ are all fixed. We assume \begin{equation*}
\left(\mathfrak{L}^N\right)_{N=1}^{\infty}, \qquad \mathfrak{L}^N = (L^N_1,L^N_2, \dots, L^N_k),
\end{equation*}
is a $(\gamma,p,\lambda,\theta)$-good sequence of $\llbracket 1, k\rrbracket$-indexed $H^{\mathrm{RW}}$ line ensembles as in Definition \ref{gooddef}, defined on a probability space with measure $\mathbb{P}$. 

After stating several important results in Section \ref{Section4.1}, we prove Theorem \ref{GoodTight} in Section \ref{Section4.2} and Theorem \ref{GoodBGP} in Section \ref{Section4.3}.
%
\subsection{Three important bounds}\label{Section4.1}
We first state three key results which we will use in the proof of the theorem. We give the proofs of these results in Sections \ref{Section5} and \ref{Section6}. For given $r > 0$ and $m \in \mathbb{N}$, we will use the notation
\begin{equation}\label{eqsts}
	t_m =\lfloor (r+m) N^{\gamma} \rfloor.
\end{equation}
	
	The first result effectively gives lower bounds on the acceptance probability of Definition \ref{HRWavoiddef}, which will be needed for the proofs of both theorems. The proof of this proposition occupies Section \ref{Section6}.
	
\begin{proposition}\label{APProp} For any $r,\epsilon > 0$ there exist $\delta > 0$ and $N_1 \in\mathbb{N}$ depending on $r,\epsilon$ such that for all $N \geq N_1$,
	\[
	\mathbb{P}\Big(Z\big( -t_1, t_1, \vec{x}, \vec{y} , \infty,  L^N_{k}\llbracket -t_1, t_1\rrbracket\big) < \delta\Big) < \epsilon,
	\]
where $\vec{x} = (L_1^N(-t_1), \dots, L_{k-1}^N(-t_1))$, $\vec{y} = (L_1^N(t_1), \dots, L^N_{k-1}(t_1))$.
\end{proposition}

	The next two results are ``no big max'' and ``no low min'' bounds, which are needed for the proof of tightness in Theorem \ref{GoodTight}. We give the proofs in Section \ref{Section5}.

\begin{lemma}\label{NoBigMax} For any $r,\epsilon > 0$, there exist $R_1 > 0$ and $N_2 \in\mathbb{N}$ depending on $r,\epsilon$ such that for all $N \geq N_2$,
	\[
	\mathbb{P} \bigg( \sup_{s \in [ -t_3, t_3] }\left( L^N_1(s) - p s \right) \geq  R_1N^{\gamma/2} \bigg) < \epsilon.
	\]
\end{lemma}

\begin{lemma}\label{NoLowMin}  For any $r,\epsilon > 0$ there exist $R_2 > 0$ and $N_3 \in\mathbb{N}$ depending on $r,\epsilon$ such that for all $N \geq N_3$,
	\[
	\mathbb{P}\left( \inf_{s \in [ -t_3, t_3 ]}\left(L^N_{k-1}(s) - p s \right) \leq - R_2N^{\gamma/2} \right) < \epsilon.
	\]
\end{lemma}

%
\subsection{Proof of Theorem \ref{GoodTight} }\label{Section4.2}

With the results of the previous section, we are equipped to prove tightness. It suffices to show that for each $i\in\llbracket 1,k-1\rrbracket$ and for every $r,\epsilon > 0$, we have
\begin{equation}\label{unifbd}
	\lim_{R\to\infty}\limsup_{N\to\infty} \mathbb{P}\left(|f_i^N(0)| \geq R\right) = 0,
\end{equation}
\begin{equation}\label{equicont}
	\lim_{\delta\searrow 0}\limsup_{N\to\infty} \mathbb{P}\bigg(\sup_{x,y\in[-r,r],\,|x-y|<\delta} \left|f_i^N(x) - f_i^N(y)\right| \geq \epsilon \bigg) = 0.
\end{equation}
The proof of \eqref{unifbd} is essentially immediate from Lemmas \ref{NoBigMax} and \ref{NoLowMin}. Indeed, given $\eta > 0$, these lemmas allow us to find $R>0$ and $N_0\in\mathbb{N}$ so that
\begin{align*}
	\mathbb{P}\left(f_i^N(0) \geq R\right) = \mathbb{P}\left(L_i^N(0) \geq RN^{\gamma/2}\right) &\leq \mathbb{P}\bigg(\sup_{s\in[-t_3,t_3]} \left(L_1^N(s) - ps\right) \geq RN^{\gamma/2}\bigg) < \eta/2,\\
	\mathbb{P}\left(f_i^N(0) \leq -R\right) = \mathbb{P}\left(L_i^N(0) \leq -RN^{\gamma/2}\right) &\leq \mathbb{P}\bigg(\inf_{s\in[-t_3,t_3]} \left(L_{k-1}^N(s) - ps\right) \leq -RN^{\gamma/2}\bigg) < \eta/2.
\end{align*}
Note we used here the fact that $L_1(0) \geq L_i(0) \geq L_{k-1}(0)$. Combining these gives
\[
\mathbb{P}\left(|f_i^N(0)| \geq R\right) \leq \mathbb{P}\left(f_i^N(0) \geq R\right) + \mathbb{P}\left(f_i^N(0) \leq -R\right) < \eta
\]
for $N\geq N_0$, which implies \eqref{unifbd}.

To prove \eqref{equicont}, given $\delta > 0$ and $N\in\mathbb{N}$ we define the event
\[
E_\delta = \left\{\sup_{s,t\in[-t_1,t_1],\,|s-t|<\delta N^\gamma} \left|L_i^N(s) - L_i^N(t) - p(s-t)\right| \geq \epsilon N^{\gamma/2}\right\},
\]
where we recall that $t_1 = \lfloor (r+1)N^\gamma\rfloor$ as in \eqref{eqsts}. Then we observe that
\begin{align*}
	\mathbb{P}\Bigg(\sup_{\substack{x,y\in[-r,r],\\|x-y|<\delta}} \left|f_i^N(x) - f_i^N(y)\right| \geq \epsilon \Bigg) &= \mathbb{P}\Bigg(\sup_{\substack{s,t\in[-rN^\gamma,rN^\gamma], \\ |s-t|<\delta N^\gamma}} \left|L_i^N(s) - L_i^N(t) - p(s-t)\right| \geq \epsilon N^{\gamma/2}\Bigg) \leq \mathbb{P}(E_\delta).
\end{align*}
We will complete the proof by showing that given $\eta > 0$, we can find $\delta_0 > 0$ and $N_1\in\mathbb{N}$ so that if $0<\delta\leq\delta_0$ and $N\geq N_1$, then
\begin{equation}\label{Edelta}
	\mathbb{P}(E_\delta) < \eta.
\end{equation}
We split the proof of this fact into two steps for clarity.\\

\noindent\textbf{Step 1.} Here we define two further events we will use to prove \eqref{Edelta}. We fix $\eta > 0$ as above. Given $\tilde{\delta}>0$, $R>0$, and $N\in\mathbb{N}$, we define
\[
F = \left\{\max_{1\leq j\leq k-1}\left|L_j^N(\pm t_1) \mp pt_1\right| \leq RN^{\gamma/2}\right\}, \quad G = \left\{Z(-t_1,t_1,\vec{x},\vec{y},\infty,L_k^N\llbracket-t_1,t_1\rrbracket) \geq \tilde{\delta}\right\}.
\] 
Here we have written $\vec{x} = (L_1^N(-t_1),\dots,L_{k-1}^N(-t_1))$, $\vec{y} = (L_1^N(t_1),\dots,L_{k-1}^N(t_1))$ as in Proposition \ref{APProp}. Lemmas \ref{NoBigMax} and \ref{NoLowMin} and Proposition \ref{APProp} allow us to choose $\tilde{\delta} > 0$, $R>0$, and $N_{10}\in\mathbb{N}$ so that for $N\geq N_{10}$ we have
\[
\mathbb{P}(F^c \cup G^c) < \eta/2.
\]
This implies
\[
\mathbb{P}(E_\delta) \leq \mathbb{P}(E_\delta\cap F\cap G) + \mathbb{P}(F^c\cup G^c) < \mathbb{P}(E_\delta\cap F\cap G) + \eta/2.
\]
In the following step, we will find $\delta_0 > 0$ and $N_{11}\in\mathbb{N}$ so that if $0<\delta\leq\delta_0$ and $N\geq N_{11}$, then
\begin{equation}\label{EdeltaFG}
	\mathbb{P}(E_\delta\cap F\cap G) < \eta/2.
\end{equation}
This implies \eqref{Edelta} for $N\geq N_1 := \max(N_{10},N_{11})$.\\

\noindent\textbf{Step 2.} We now prove \eqref{EdeltaFG}. Let $\mathcal{F}$ denote the $\sigma$-algebra generated by the random variables $L_k^N\llbracket -t_1,t_1\rrbracket$ and $L_i^N(\pm t_1)$ for $i\in\llbracket 1,k-1\rrbracket$. Then we have $F,G\in\mathcal{F}$. It follows that
\[
\mathbb{P}(E_\delta \cap F \cap G) = \ex[\mathbf{1}_{E_\delta}\mathbf{1}_F\mathbf{1}_G] = \ex[\mathbf{1}_F\mathbf{1}_G\cdot\ex[\mathbf{1}_{E_\delta}\,|\,\mathcal{F}]].
\]
The $H^{\mathrm{RW}}$-Gibbs property (Definition \ref{DefHRWGP}) implies that
\[
\ex[\mathbf{1}_{E_\delta}\,|\,\mathcal{F}] = \ex^{-t_1,t_1,\vec{x},\vec{y},\infty,L_k\llbracket-t_1,t_1\rrbracket}_{\mathrm{avoid},H^{\mathrm{RW}}}[\mathbf{1}_{E_\delta}].
\]
We also observe that the Radon-Nikodym derivative of the measure $\pr^{-t_1,t_1,\vec{x},\vec{y},\infty,L_k\llbracket-t_1,t_1\rrbracket}_{\mathrm{avoid},H^{\mathrm{RW}}}$ with respect to $\pr^{-t_1,t_1,\vec{x},\vec{y}}_{H^{\mathrm{RW}}}$ is given by
\[
\frac{d\pr^{-t_1,t_1,\vec{x},\vec{y},\infty,L_k\llbracket-t_1,t_1\rrbracket}_{\mathrm{avoid},H^{\mathrm{RW}}}}{d\pr^{-t_1,t_1,\vec{x},\vec{y}}_{H^{\mathrm{RW}}}}(Q_1,\dots,Q_{k-1}) = \frac{\mathbf{1}_A}{Z(-t_1,t_1,\vec{x},\vec{y},\infty,L_k\llbracket -t_1,t_1\rrbracket)},
\]
where $(Q_1,\dots,Q_{k-1})$ has law $\mathbb{P}^{-t_1,t_1,\vec{x},\vec{y}}_{H^{\mathrm{RW}}}$ and $A$ denotes the avoidance event that $L_1 \geq \cdots \geq L_{k-1} \geq L_k\llbracket -t_1,t_1\rrbracket$ on $\llbracket -t_1,t_1\rrbracket$. This is immediate from the fact that the acceptance probability $Z$ is defined as $\mathbb{P}_{H^{\mathrm{RW}}}^{-t_1,t_1,\vec{x},\vec{y}}(A)$ (see Definition \ref{HRWavoiddef}).

The last three equations together imply that
\begin{align*}
	\mathbb{P}(E_\delta\cap F\cap G) &= \ex\left[\mathbf{1}_F\mathbf{1}_G\cdot \ex^{-t_1,t_1,\vec{x},\vec{y}}_{H^{\mathrm{RW}}}\left[\frac{\mathbf{1}_{\tilde{E}_\delta}\mathbf{1}_A}{Z(-t_1,t_1,\vec{x},\vec{y},\infty,L_k\llbracket -t_1,t_1\rrbracket)}\right]\right],
\end{align*}
where $\tilde{E}_\delta$ is defined in the same way as $E_\delta$ with $(L_1,\dots,L_{k-1})$ replaced by $(Q_1,\dots,Q_{k-1})$. Since the acceptance probability is bounded below by $\tilde{\delta}$ on the event $G$, we have
\begin{equation}\label{tildedelta-1}
	\mathbb{P}(E_\delta\cap F\cap G) \leq \tilde{\delta}^{-1}\cdot\ex\left[\mathbf{1}_F\cdot \ex^{-t_1,t_1,\vec{x},\vec{y}}_{H^{\mathrm{RW}}}\left[\mathbf{1}_{\tilde{E}_\delta}\right]\right] = \tilde{\delta}^{-1}\cdot\ex\left[\mathbf{1}_F\cdot\mathbb{P}^{-t_1,t_1,\vec{x},\vec{y}}_{H^{\mathrm{RW}}}(\tilde{E}_\delta)\right].
\end{equation}
Note that $Q_1,\dots,Q_{k-1}$ are each independent with laws $\mathbb{P}^{-t_1,t_1,x_i,y_i}_{H^{\mathrm{RW}}}$, so we have
\[
\mathbb{P}^{-t_1,t_1,\vec{x},\vec{y}}_{H^{\mathrm{RW}}}(\tilde{E}_\delta) = \prod_{i=1}^{k-1} \mathbb{P}^{0,2t_1,0,y_i-x_i}_{H^{\mathrm{RW}}}\left(\sup_{x,y\in[0,2t_1],\,|x-y| < \delta N^{\gamma/2}} |\ell_i(x)-\ell_i(y) - p(x-y)| \geq \epsilon N^{\gamma/2}\right),
\]
where $\ell_i$ have laws $\mathbb{P}^{0,2t_1,0,y_i-x_i}_{H^{\mathrm{RW}}}$. Since on the event $F$ we have $|y_i - x_i - 2pt_1| \leq 2RN^{\gamma/2}$, Lemma \ref{MOC} allows us to choose $\delta_0 > 0$ and $N_{11}\in\mathbb{N}$ so that if $0<\delta\leq\delta_0$ and $N\geq N_{11}$, then each of the factors in the product on the right is bounded by $(\tilde{\delta}\eta/2)^{1/(k-1)}$ on the event $F$. Then we have
\[
\mathbf{1}_F\cdot\mathbb{P}^{-t_1,t_1,\vec{x},\vec{y}}_{H^{\mathrm{RW}}}(\tilde{E}_\delta) \leq \tilde{\delta}\eta/2,
\]
which in view of \eqref{tildedelta-1} implies \eqref{EdeltaFG} for $0<\delta\leq\delta_0$ and $N\geq N_{11}$. This completes the proof of Theorem \ref{GoodTight}.

%

\subsection{Proof of Theorem \ref{GoodBGP} }\label{Section4.3}

We now turn to the proof of the Brownian Gibbs property for subsequential limits. We fix notation as in the theorem statement; namely, we let $f^\infty = (f_1^\infty,\dots,f_{k-1}^\infty)$ denote a weak subsequential limit of the sequence $\{f^N\}$ appearing in Theorem \ref{GoodTight}. We write $\mathcal{L}^N := \sigma^{-1} f^N$, $\mathcal{L}^\infty := \sigma^{-1} f^\infty$, with $\sigma = \sigma_p$ as in Assumption \ref{KMT}. Roughly, the idea is to argue that the $H^{\mathrm{RW}}$-Gibbs property transfers to the partial Brownian Gibbs property in the limit. We will utilize the weak convergence result Lemma \ref{ScaledWeakConv} for non-crossing random walk bridges.

We will need the following lemma in order to apply Lemma \ref{ScaledWeakConv}, which states that at any fixed time, the values of the curves of $\mathcal{L}^\infty$ are almost surely distinct.

\begin{lemma}\label{DistinctPts}
	For any $x\in\mathbb{R}$, we have $\mathcal{L}^\infty(x) = (\mathcal{L}^\infty_1(x),\dots,\mathcal{L}^\infty_{k-1}(x)) \in W_{k-1}^\circ$, $\mathbb{P}$-a.s.
\end{lemma}

\begin{proof}
	By passing to a subsequence, we may assume that $\mathcal{L}^N \implies \mathcal{L}^\infty$ as $N\to\infty$. By the Skorohod representation theorem, we may assume furthermore that $\mathcal{L}^N \to \mathcal{L}^\infty$ uniformly on compact sets as $N\to\infty$, $\mathbb{P}$-a.s. We recall that $\sigma \mathcal{L}^N_i(x) = N^{-\gamma/2}(L^N_i(xN^\gamma) - pxN^\gamma)$.
	
	Suppose that $\mathcal{L}_i^\infty(x) = \mathcal{L}_{i+1}^\infty(x)$ for some $i\in\llbracket 1,k-2\rrbracket$. Let us write $a = \lfloor xN^\gamma\rfloor $, $b = \lceil (x+2)N^\gamma\rceil$. If $Q_i,Q_{i+1}$ are independent $H^{\mathrm{RW}}$ bridges with laws $\mathbb{P}^{a,b,L_i^N(a),L_i^N(b)}_{H^{\mathrm{RW}}}$ and $\mathbb{P}^{a,b,L_{i+1}^N(a),L_{i+1}^N(b)}_{H^{\mathrm{RW}}}$, then $\ell := Q_i - Q_{i+1}$ is a random walk bridge on $[a,b]$ taking values in $\llbracket \alpha - \beta, -\alpha + \beta\rrbracket$. Note in particular that 
	\begin{equation}\label{EndptLim}
		\begin{split}
			N^{-\gamma/2}\ell(a)/\sigma &= \mathcal{L}_i^N(aN^{-\gamma}) - \mathcal{L}_{i+1}^N(aN^{-\gamma}) \longrightarrow \mathcal{L}_i^\infty(x) - \mathcal{L}_{i+1}^\infty(x) = 0,\\
			N^{-\gamma/2}\ell(b)/\sigma &= \mathcal{L}_i^N(bN^{-\gamma}) - \mathcal{L}_{i+1}^N(bN^{-\gamma}) \longrightarrow \mathcal{L}_i^\infty(x+2) - \mathcal{L}_{i+1}^\infty(x+2)
		\end{split}
	\end{equation}
	as $N\to\infty$. Here we used the facts that $aN^{-\gamma}\to x$ and $bN^{-\gamma}\to x+2$ as $N\to\infty$, and the assumption that the convergence of $\mathcal{L}^N$ is uniform.
	
	By Lemma \ref{ScaledWeakConv}, the rescaled bridges $N^{-\gamma/2}(Q_{i+1}(tN^\gamma) - ptN^\gamma)/\sigma$ and $N^{-\gamma/2}(Q_i(tN^\gamma)-ptN^\gamma)/\sigma$, with $t\in[aN^{-\gamma},bN^{-\gamma}]$, converge weakly to two Brownian bridges $B^1$ and $B^2$ on $[x,x+2]$. Therefore, their difference $N^{-\gamma/2}\ell(tN^\gamma)/\sigma$ converges weakly to the difference of two independent Brownian bridges, $B := B^1-B^2$. As a difference of Brownian bridges, $B$ is itself a Brownian bridge up to scaling. In particular, in view of \eqref{EndptLim} we see that $B/\sqrt{2}$ is a Brownian bridge from $(x,0)$ to $(x+2,y/\sqrt{2})$, where $y = \mathcal{L}_i^\infty(x+2) - \mathcal{L}_{i+1}^\infty(x+2)$ (see \cite[Lemma 7.6]{Ber}). Therefore $N^{-\gamma/2}\ell(tN^\gamma)/\sigma\sqrt{2}$ converges weakly to $\mathbb{P}^{x,x+2,0,y/\sqrt{2}}_{\mathrm{free}}$. Since $B/\sqrt{2}$ is a Brownian bridge started at 0, with probability one we have $\min_{t\in[x,x+2]} B_t/\sqrt{2} < 0$. (This follows from the Blumenthal 0-1 law; see \cite[Lemma 7.5]{Ber}.) Thus given $\delta > 0$, we can choose $N$ large enough so that the probability of $N^{-\gamma/2}\ell(tN^\gamma)/\sigma\sqrt{2}$, or equivalently $\ell$, remaining above 0 on $[a,b]$ is less than $\delta$. Let us write $\mathbb{P}^{a,b,\ell(a),\ell(b)}_{\mathrm{diff}}$ for the law of $\ell$; this is a random measure depending on $\ell(a)$ and $\ell(b)$. Then for large enough $N$ we have
	\begin{equation}\label{collideAP}
		\begin{split}
			\mathbb{P}\left(\mathcal{L}_i^\infty(x) = \mathcal{L}_{i+1}^\infty(x)\right) &\leq \mathbb{P}\bigg(\mathbb{P}^{a,b,\ell(a),\ell(b)}_{\mathrm{diff}}\bigg(\inf_{s\in[a,b]}\ell(s) \geq 0\bigg) < \delta\bigg)\\ 
			&\leq \mathbb{P}\left(Z(a,b,L^N(a),L^N(b),\infty,L^N_k) < \delta\right).
		\end{split}
	\end{equation}
	Here $Z$ denotes the acceptance probability of Definition \ref{HRWavoiddef}. This is the probability that $k-1$ independent Bernoulli bridges $Q_1,\dots,Q_{k-1}$ on $[a,b]$ with entrance and exit data $L^N(aN^\gamma)$ and $L^N(bN^\gamma)$ do not cross one another or $L_k^N$. The last inequality follows because $\ell$ has the law of the difference of $Q_i$ and $Q_{i+1}$, and the acceptance probability is bounded above by the probability that $Q_i$ and $Q_{i+1}$ do not cross, i.e., that $Q_i - Q_{i+1} \geq 0$. By Proposition \ref{APProp}, given $\epsilon > 0$ we can choose $\delta$ so that the probability on second line in \eqref{collideAP} is $<\epsilon$. We conclude that
	\[
	\mathbb{P}\left(\mathcal{L}_i^\infty(x) = \mathcal{L}_{i+1}^\infty(x)\right) = 0.
	\]
\end{proof}

We conclude this section with the proof of the theorem.

\begin{proof}[Proof of Theorem \ref{GoodBGP}]
	As above, we assume without loss of generality that $\mathcal{L}^N \implies \mathcal{L}^\infty$ as $N\to\infty$. Fix a set $K = \llbracket k_1,k_2\rrbracket \subseteq \llbracket 1, k-2\rrbracket$ and $a,b\in\mathbb{R}$ with $a<b$, as well as a bounded Borel-measurable function $F:C(K\times[a,b])\to\mathbb{R}$. It suffices to prove that $\mathbb{P}$-a.s., $\mathcal{L}^\infty$ is non-intersecting and
	\begin{equation}\label{BGPcondex}
		\ex[F(\mathcal{L}^\infty|_{K\times[a,b]})\,|\,\mathcal{F}_{\mathrm{ext}}(K\times(a,b))] = \ex^{a,b,\vec{x},\vec{y},f,g}_{\mathrm{avoid}}[F(\mathcal{Q})],
	\end{equation}
	where $\vec{x} = (\mathcal{L}^\infty_{k_1}(a),\dots,\mathcal{L}^\infty_{k_2}(a))$, $\vec{y} = (\mathcal{L}^\infty_{k_1}(b),\dots,\mathcal{L}^\infty_{k_2}(b))$, $f=\mathcal{L}^\infty_{k_1-1}$ (with $\mathcal{L}^\infty_0 = +\infty$), $g=\mathcal{L}^\infty_{k_2+1}$, the $\sigma$-algebra $\mathcal{F}_{\mathrm{ext}}(K\times(a,b))$ is as in Definition \ref{DefBGP}, and $\mathcal{Q}$ has law $\mathbb{P}^{a,b,\vec{x},\vec{y},f,g}_{\mathrm{avoid}}$. Note that \eqref{BGPcondex} in fact implies that $\mathcal{L}^\infty$ is a.s. non-intersecting. Indeed, take $K = \llbracket 1,k-2\rrbracket$, $[a,b]=[-m,m]$ where $m\in\mathbb{N}$, and $F(f_1,\dots,f_{k-2}) = \mathbf{1}\{f_1(s) > \cdots > f_{k-2}(s) \mbox{ for all } s\in[-m,m]\}$. Then the right-hand side is 1, so $\mathcal{L}^\infty|_{\llbracket 1,k-2\rrbracket\times[-m,m]}$ is a.s. non-intersecting. Taking the countable intersection over $m,k\in\mathbb{N}$  shows that $\mathcal{L}^\infty$ is a.s. non-intersecting. It remains to prove \eqref{BGPcondex}, which we do in two steps.\\
	
	\noindent\textbf{Step 1. } Fix $m\in\mathbb{N}$, $n_1,\dots,n_m\in\llbracket 1,k\rrbracket$, $t_1,\dots,t_m\in\mathbb{R}$, and $h_1,\dots,h_m : \mathbb{R}\to\mathbb{R}$ bounded continuous functions. Define $S = \{i\in\llbracket 1,m\rrbracket : n_i \in K, t_i \in [a,b]\}$. In this step we prove that
	\begin{equation}\label{BBcondexsplit}
		\ex\left[\prod_{i=1}^m h_i(\mathcal{L}^\infty_{n_i}(t_i))\right] = \ex\left[\prod_{s\notin S} h_s(\mathcal{L}^\infty_{n_s}(t_s))\cdot\ex^{a,b,\vec{x},\vec{y},f,g}_{\mathrm{avoid}}\left[\prod_{s\in S} h_s(Q_{n_s}(t_s))\right]\right],
	\end{equation}
	where $Q$ denotes a random variable with law $\pr^{a,b,\vec{x},\vec{y},f,g}_{\mathrm{avoid}}$. By assumption, we have
	\begin{equation}\label{BGPweak}
		\lim_{N\to\infty}\ex\left[\prod_{i=1}^m h_i(\mathcal{L}^N_{n_i}(t_i))\right] = \ex\left[\prod_{i=1}^m h_i(\mathcal{L}^\infty_{n_i}(t_i))\right].
	\end{equation}
	We define the sequences $a_N = \lfloor aN^\gamma\rfloor N^{-\gamma}$, $b_N = \lceil bN^\gamma\rceil N^{-\gamma}$, $f_N = \mathcal{L}_{k_1-1}^N$ (where $\mathcal{L}_0 = +\infty$), and $g_N = \mathcal{L}_{k_2+1}^N$. We also define $\vec{x}\,^N,\vec{y}\,^N \in \mathfrak{W}^{k_2-k_1+1}$ so that
	\[
	\frac{x_i^N - pa_NN^\gamma}{\sigma N^{\gamma/2}} = \mathcal{L}_i^N(a_N), \quad \frac{y_i^N - pb_NN^\gamma}{\sigma N^{\gamma/2}} = \mathcal{L}_i^N(b_N), \quad i \in \llbracket k_1, k_2\rrbracket.
	\]
	Since $a_N \to a$, $b_N\to b$, we may choose $N_0$ sufficiently large so that if $N\geq N_0$, then $t_s < a_N$ or $t_s > b_N$ for all $s\notin S$ with $n_s \in K$. Note that the line ensemble $(\mathcal{L}_1^N,\dots,\mathcal{L}_{k-1}^N)$ satisfies the $H^{\mathrm{RW}}$-Gibbs property of Definition \ref{DefHRWGP}. It follows that the law of $\mathcal{L}^N|_{K\times[a,b]}$ conditioned on the $\sigma$-algebra $\mathcal{F} = \sigma\left(\mathcal{L}^N_{k_1-1}, \mathcal{L}^N_{k_2+1}, \mathcal{L}^N_{k_1}(a_N), \mathcal{L}^N_{k_1}(b_N),\dots,\mathcal{L}^N_{k_2}(a_N),\mathcal{L}^N_{k_2}(b_N)\right)$ is the law of the $C(K\times[a,b])$-valued random variable $Z^N$ given by
	\[
	Z^N(t) = \left(\frac{Q_{k_1}^N(tN^\gamma)-ptN^\gamma}{\sigma N^{\gamma/2}}, \dots, \frac{Q_{k_2}^N(tN^\gamma)-ptN^\gamma}{\sigma N^{\gamma/2}}\right), \quad t\in[a,b],
	\] 
	where $(Q_{k_1}^N,\dots,Q_{k_2}^N)$ has law $\mathbb{P}^{a_N,b_N,\vec{x}\,^N,\vec{y}\,^N,f_N,g_N}_{\mathrm{avoid},H^{\mathrm{RW}}}$. Let us write $\ex_{\mathrm{avoid},N}$ for the expectation with respect to the law of $Z^N$. Then we have
	\begin{equation}\label{BBschur}
		\ex\left[\prod_{i=1}^m h_i(Y^N_{n_i}(t_i))\right] = \ex\left[\prod_{s\notin S} h_s(Y^N_{n_s}(t_s))\cdot\ex_{\mathrm{avoid},N}\left[\prod_{s\in S} h_s(Z^N_{n_s-k_1+1}(t_s))\right]\right].
	\end{equation}
	Now by Lemma \ref{DistinctPts}, we have $\mathbb{P}$-a.s. that $\vec{x},\vec{y} \in W_{k_2-k_1+1}^\circ$. By the Skorohod representation theorem, we may assume that $\mathcal{L}^N \to \mathcal{L}^\infty$ uniformly on compact sets almost surely. In particular, $f_N\to f = f^\infty_{k_2+1}$ and $g_N\to g = f^\infty_{k_1-1}$ uniformly on $[a-1,b+1]\supseteq [a_N,b_N]$, and $N^{-\gamma/2}(x_i^N - pa_N N^\gamma)/\sigma \to x_i$, $N^{-\gamma/2}(y_i^N-pb_N N^{\gamma})/\sigma \to y_i$ for $i\in\llbracket k_1, k_2\rrbracket$. It follows from Lemma \ref{ScaledWeakConv} and the definition of $Z^N$ above that 
	\begin{equation}\label{BGPNweak}
		\lim_{N\to\infty} \ex_{\mathrm{avoid},N}\left[\prod_{s\in S} h_s(Z^N_{n_s-k_1+1}(t_s))\right] = \ex^{a,b,\vec{x},\vec{y},f,g}_{\mathrm{avoid}}\left[\prod_{s\in S} h_s(Q_{n_s}(t_s))\right].
	\end{equation}
	Lastly, the continuity of the $h_i$ implies that
	\begin{equation}\label{BGPuniform}
		\lim_{N\to\infty}\prod_{s\notin S} h_s(\mathcal{L}_{n_s}^N(t_s)) = \prod_{s\notin S} h_s(\mathcal{L}^\infty_{n_s}(t_s)).
	\end{equation}
	Combining \eqref{BGPweak}, \eqref{BBschur}, \eqref{BGPNweak}, and \eqref{BGPuniform} and applying the dominated convergence theorem proves \eqref{BBcondexsplit}.\\
	
	\noindent\textbf{Step 2. } In this step we prove \eqref{BGPcondex} as a consequence of \eqref{BBcondexsplit}. For $n\in\mathbb{N}$ we define piecewise linear functions
	\[
	\chi_n(x,r) = \begin{cases}
		0, & x > r + 1/n,\\
		1-n(x-r), & x\in[r,r+1/n],\\
		1, & x < r.
	\end{cases}
	\]
	We fix $m_1,m_2\in\mathbb{N}$, $n^1_1,\dots,n^1_{m_1},n^2_1,\dots,n^2_{m_2}\in\llbracket 1,k\rrbracket$, $t^1_1,\dots,t^1_{m_1},t^2_1,\dots,t^2_{m_2}\in\mathbb{R}$, such that $(n^1_i,t^1_i)\notin K\times[a,b]$ and $(n^2_i,t^2_i)\in K\times[a,b]$ for all $i$. Then \eqref{BBcondexsplit} implies that
	\[
	\ex\left[\prod_{i=1}^{m_1} \chi_n(\mathcal{L}_{n_i^1}^\infty(t_i^1),a_i)\prod_{i=1}^{m_2}\chi_n(\mathcal{L}_{n_i^2}^\infty(t_i^2),b_i)\right] = \ex\left[\prod_{i=1}^{m_1} \chi_n(\mathcal{L}_{n_i^1}^\infty(t_i^1),a_i) \cdot \ex^{a,b,\vec{x},\vec{y},f,g}_{\mathrm{avoid}}\left[\prod_{i=1}^{m_2} \chi_n(Q_{n_i^2}(t_i^2),b_i)\right]\right].
	\]
	Letting $n\to\infty$, we have $\chi_n(x,r)\to \chi(x,r) := \mathbf{1}_{x\leq r}$, and the dominated convergence theorem implies that
	\[
	\ex\left[\prod_{i=1}^{m_1} \chi(\mathcal{L}_{n_i^1}^\infty(t_i^1),a_i)\prod_{i=1}^{m_2}\chi(\mathcal{L}_{n_i^2}^\infty(t_i^2),b_i)\right] = \ex\left[\prod_{i=1}^{m_1} \chi(\mathcal{L}_{n_i^1}^\infty(t_i^1),a_i)\cdot \ex^{a,b,\vec{x},\vec{y},f,g}_{\mathrm{avoid}}\left[\prod_{i=1}^{m_2} \chi(Q_{n_i^2}(t_i^2),b_i)\right]\right].
	\]
	Let $\mathcal{H}$ denote the space of bounded Borel measurable functions $H:C(K\times[a,b])\to\mathbb{R}$ satisfying
	\begin{equation}\label{BGPH}
		\ex\left[\prod_{i=1}^{m_1} \chi(\mathcal{L}_{n_i^1}^\infty(t_i^1),a_i)\cdot H(\mathcal{L}^\infty|_{K\times[a,b]})\right] = \ex\left[\prod_{i=1}^{m_1} \chi(\mathcal{L}_{n_i^1}^\infty(t_i^1),a_i)\cdot\ex^{a,b,\vec{x},\vec{y},f,g}_{\mathrm{avoid}}\left[H(\mathcal{Q})\right]\right].
	\end{equation}
	The above implies that $\mathcal{H}$ contains all functions $\mathbf{1}_A$ for sets $A$ contained in the $\pi$-system $\mathcal{A}$ consisting of sets of the form
	\[
	\{h\in C(K\times[a,b]) : h(n_i^2,t_i^2) \leq b_i \mbox{ for } i\in\llbracket 1,m_2\rrbracket\}.
	\]
	Note that $\mathcal{H}$ is closed under linear combinations by linearity of expectation, and if $H_n\in\mathcal{H}$ are nonnegative bounded measurable functions converging monotonically to a bounded function $H$, then $H\in\mathcal{H}$ by the monotone convergence theorem. Thus by the monotone class theorem \cite[Theorem 5.2.2]{Durrett}, $\mathcal{H}$ contains all bounded $\sigma(\mathcal{A})$-measurable functions. Since the finite dimensional sets in $\mathcal{A}$ generate the full Borel $\sigma$-algebra $\mathcal{C}_K$ (see for instance \cite[Lemma 3.1]{DimMat}), we have in particular that $F\in\mathcal{H}$.
	
	To conclude, let $\mathcal{B}$ denote the collection of sets $B\in\mathcal{F}_{\mathrm{ext}}(K\times(a,b))$ such that
	\begin{equation}\label{BGPB}
		\ex[\mathbf{1}_B \cdot F(\mathcal{L}^\infty|_{K\times[a,b]})] = \ex[\mathbf{1}_B \cdot \ex^{a,b,\vec{x},\vec{y},f,g}_{\mathrm{avoid}}[F(\mathcal{Q})]].
	\end{equation}
	We observe that $\mathcal{B}$ is a $\lambda$-system. Indeed, since \eqref{BGPH} holds for $H=F$, taking $a_i,b_i\to\infty$ and applying the dominated convergence theorem shows that \eqref{BGPB} holds with $\mathbf{1}_B = 1$. Thus if $B\in\mathcal{B}$ then $B^c\in\mathcal{B}$ since $\mathbf{1}_{B^c} = 1-\mathbf{1}_B$. If $B_i\in\mathcal{B}$, $i\in\mathbb{N}$, are pairwise disjoint and $B=\bigcup_i B_i$, then $\mathbf{1}_B = \sum_i \mathbf{1}_{B_i}$, and it follows from the monotone convergence theorem that $B\in\mathcal{B}$. Moreover, \eqref{BGPH} with $H=F$ implies that $\mathcal{B}$ contains the $\pi$-system $P$ of sets of the form
	\[
	\{h\in C(\Sigma\times\mathbb{R}) : h(n_i,t_i) \leq a_i \mbox{ for } i \in\llbracket 1,m_1\rrbracket, \mbox{ where } (n_i,t_i)\notin K\times(a,b)\}.
	\]
	By the $\pi$-$\lambda$ theorem \cite[Theorem 2.1.6]{Durrett} it follows that $\mathcal{B}$ contains $\sigma(P) = \mathcal{F}_{\mathrm{ext}}(K\times(a,b))$. Thus \eqref{BGPB} holds for all $B\in\mathcal{F}_{\mathrm{ext}}(K\times(a,b))$. It is proven in \cite[Lemma 3.4]{DimMat} that $\ex^{a,b,\vec{x},\vec{y},f,g}_{\mathrm{avoid}}[F(\mathcal{Q})]$ is an $\mathcal{F}_{\mathrm{ext}}(K\times(a,b))$-measurable function, and \eqref{BGPcondex} therefore follows from \eqref{BGPB}.
	
\end{proof}

%% file: section5.tex
%
\section{No big max and no low min}\label{Section5}

In this section we prove the ``no big max'' Lemma \ref{NoBigMax} and ``no low min'' Lemma \ref{NoLowMin}, in Sections \ref{NBMpf} and \ref{NLMpf} respectively.

\subsection{Proof of Lemma \ref{NoBigMax}}\label{NBMpf}

Let us briefly explain the idea of the argument. We bound separately the suprema over $[0,t_3]$ and $[-t_3,0]$; the arguments are symmetrical. For the former, we consider two negative times $-s_4 < -s_3$ comparable in magnitude to $t_3$, and we use the one-point tightness assumption in Definition \ref{gooddef} to ensure that the top curve $L_1^N$ lies within a certain window on scale $N^{\gamma/2}$ at $-s_4$.  If the supremum over $[0,t_3]$ were sufficiently large, then we could find a time $s\in[0,t_3]$ where $L_1^N$ is very high. Then applying the $H^{\mathrm{RW}}$-Gibbs property along with Lemma \ref{BridgeInterp} for a bridge between $-s_4$ and $s$, we argue that in this situation $L_1^N$ would rise high above the window it is expected to lie in at time $-s_3$, which is a very unlikely event. We now make this argument precise.
\begin{proof}
	We will prove that we can find $R_1,N_2$ so that
	\begin{equation}\label{NoBigMax+}
		\mathbb{P} \bigg( \sup_{s \in [ 0, t_3] }\big( L^N_1(s) - p s \big) \geq  R_1N^{\gamma/2} \bigg) < \epsilon/2
	\end{equation}
	for $N\geq N_2$. A very similar argument will prove the same inequality with $[-t_3,0]$ in place of $[0,t_3]$, and together with \eqref{NoBigMax+} this implies the conclusion of the Lemma.
	
	We first introduce notation to be used in the proof. We write $s_3 = \lceil \lfloor r+3\rfloor N^\gamma\rceil$, $s_4 = \lceil \lceil r+4\rceil N^\gamma\rceil$, so that $s_3 \leq t_3 \leq s_4$. Throughout, we will assume $N$ is large enough so that $\psi(N)N^\gamma \geq s_4$, which is possible by condition (2) in Definition \ref{gooddef}, and thus $\mathfrak{L}^N$ is defined at $\pm s_4$. We define events
	\[
	E = \left\{\left|L_1^N(-s_4) + ps_4\right| > MN^{\gamma/2}\right\}, \quad F = \left\{L_1^N(-s_3) + ps_3 > MN^{\gamma/2}\right\},
	\]
	\[
	G = \left\{\sup_{s\in[0,s_4]}\left(L_1^N(s)-ps\right) \geq 3M(2r+10)^{3/2}N^{\gamma/2}\right\}.
	\]
	Here, $M>0$ is chosen along with $N_{20}\in\mathbb{N}$ according to condition (3) in Definition \ref{gooddef} so that
	\begin{equation}\label{NoBigMaxEF}
		\mathbb{P}(E) < \epsilon/4 \quad \mathrm{and} \quad \mathbb{P}(F) < \epsilon/12
	\end{equation}
	for $N \geq N_{20}$. (Note that $-s_3 = \lfloor -\lfloor r+3\rfloor N^\gamma\rfloor$ and $-s_4 = \lfloor -\lceil r+4\rceil N^\gamma\rfloor$.) The proof of \eqref{NoBigMax+} will be accomplished by finding $N_{21}$ so that
	\begin{equation}\label{NoBigMaxGE}
		\mathbb{P}(G\setminus E) < \epsilon/4
	\end{equation}
	for $N\geq N_{21}$. In view of \eqref{NoBigMaxEF}, this implies $\mathbb{P}(G) \leq \mathbb{P}(G\setminus E) + \mathbb{P}(E) < \epsilon/2$ for $N\geq N_2 := \max(N_{20},N_{21})$, which in turn implies \eqref{NoBigMax+} with $R_1 = 3M(2r+10)^{3/2}$. We split the proof of \eqref{NoBigMaxGE} into three steps for clarity.\\
	
	\noindent\textbf{Step 1. } In this step, we decompose the set $G\setminus E$ into a countable disjoint union
	\begin{equation}\label{NoBigMaxDisj}
		G\setminus E = \bigsqcup_{(a,b,s,\ell_{\mathrm{t}},\ell_{\mathrm{b}}) \in I} H(a,b,s,\ell_{\mathrm{t}},\ell_{\mathrm{b}}),
	\end{equation}
	with the sets $I$ and $H(a,b,s,\ell_{\mathrm{t}},\ell_{\mathrm{b}})$ defined as follows. For $a,b\in\mathbb{Z}$, $s\in\llbracket 0, s_4 \rrbracket$, and $\ell_{\mathrm{t}},\ell_{\mathrm{b}}$ discrete paths on $\llbracket s,s_4\rrbracket$ and $\llbracket-s_4,s\rrbracket$ respectively, we define $H(a,b,s,\ell_{\mathrm{t}},\ell_{\mathrm{b}})$ to be the event that $L_1^N(-s_4) = a$, $L_1^N(s) = b$, $L_1^N$ agrees with $\ell_{\mathrm{t}}$ on $\llbracket s,s_4\rrbracket$, and $L_2^N$ agrees with $\ell_{\mathrm{b}}$ on $\llbracket-s_4,s\rrbracket$. We define $I$ as the set of tuples $(a,b,s,\ell_{\mathrm{t}},\ell_{\mathrm{b}})$ satisfying
	\begin{enumerate}[label=(\arabic*)]
		\item $0\leq s\leq s_4$ and $\alpha(s+s_4)\leq b-a \leq \beta(s + s_4)$,
		\item $|a + ps_4| \leq MN^{\gamma/2}$ and $b-ps \geq 3M(2r+10)^{3/2}N^{\gamma/2}$,
		\item $z_1,z_2\in\mathbb{Z}$, $z_1\leq a$, $z_2\leq b$, $\alpha(s+s_4) \leq z_2-z_1 \leq \beta(s+s_4)$, and $\ell_{\mathrm{b}}\in\Omega(-s_4, s, z_1, z_2)$,
		\item if $s < s' \leq s_4$, then $\ell_{\mathrm{t}}(s') -ps' < 3M(2r+10)^{3/2}N^{\gamma/2}$.
	\end{enumerate}
	It is clear that $I$ is countable since $a,b,s$ are each integers and there are finitely many choices of $\ell_{\mathrm{t}},\ell_{\mathrm{b}}$ for fixed choices of $a,b,s$. It is also clear from the four conditions that $G\setminus E = \bigcup_I H(a,b,s,\ell_{\mathrm{t}},\ell_{\mathrm{b}})$. In particular, to see that condition (4) is not too restrictive, note that on the event $G$, there must be a last time on $[0,s_4]$ at which the lower bound is exceeded. Moreover, condition (4) clearly implies that the sets $H(a,b,s,\ell_{\mathrm{t}},\ell_{\mathrm{b}})$ are pairwise disjoint. This proves \eqref{NoBigMaxDisj}.\\
	
	\noindent\textbf{Step 2. } In this step, we find $N_{21}$ so that if $N\geq N_{21}$, $(a,b,s,\ell_{\mathrm{t}},\ell_{\mathrm{b}})\in I$, and $\mathbb{P}(H(a,b,s,\ell_{\mathrm{t}},\ell_{\mathrm{b}})) > 0$, then
	\begin{equation}\label{NoBigMax1/3a}
		\mathbb{P}^{-s_4,s,a,b,\infty,\ell_{\mathrm{b}}}_{\mathrm{avoid},H^{\mathrm{RW}}}\left(\ell(-s_3) + ps_3 > MN^{\gamma/2}\right) \geq \frac{1}{3}.
	\end{equation}
	Note that the condition $\mathbb{P}(H(a,b,s,\ell_{\mathrm{t}},\ell_{\mathrm{b}}))>0$ implies that $\Omega_{\mathrm{avoid}}(-s_4,s,a,b,\infty,\ell_{\mathrm{b}}) \neq \varnothing$, so the measure on the left is well-defined. By Lemma \ref{MCL}, we have
	\[
	\mathbb{P}^{-s_4,s,a,b}_{H^{\mathrm{RW}}}\left(\ell(-s_3) + ps_3 > MN^{\gamma/2}\right) \leq \mathbb{P}^{-s_4,s,a,b,\infty,\ell_{\mathrm{b}}}_{\mathrm{avoid},H^{\mathrm{RW}}}\left(\ell(-s_3) + ps_3 > MN^{\gamma/2}\right)
	\] 
	so it suffices to show that
	\begin{equation}\label{NoBigMax1/3}
		\mathbb{P}^{-s_4,s,a,b}_{H^{\mathrm{RW}}}\left(\ell(-s_3) + ps_3 > MN^{\gamma/2}\right) \geq \frac{1}{3}.
	\end{equation}
	The assumption in condition (2) that $a+ps_4 \geq -MN^{\gamma/2}$ implies that
	\begin{equation}\label{NoBigMax2M}
		\begin{split}
			\mathbb{P}^{-s_4,s,a,b}_{H^{\mathrm{RW}}}\left(\ell(-s_3) + ps_3 > MN^{\gamma/2}\right) &= \mathbb{P}^{0,s+s_4,0,b-a}_{H^{\mathrm{RW}}}\left(\ell(s_4-s_3) + a + ps_3 > MN^{\gamma/2}\right)\\
			&\geq \mathbb{P}^{0,s+s_4,0,b-a}_{H^{\mathrm{RW}}}\left(\ell(s_4-s_3) -p(s_4-s_3) > 2MN^{\gamma/2}\right).
		\end{split}
	\end{equation}
	Furthermore, the facts that $a+ps_4 \leq MN^{\gamma/2}$, $b-ps \geq 3M(2r+10)^{3/2}N^{\gamma/2}$, and $s+s_4 \leq 2s_4 \leq (2r+10)N^{\gamma}$ imply that
	\begin{align*}
		b-a &\geq p(s+s_4) + 2M(2r+10)^{3/2}N^{\gamma/2}\geq p(s+s_4) + 2M(2r+10)(s+s_4)^{1/2}.
	\end{align*}
	We can thus apply Lemma \ref{BridgeInterp} with $M_1 = 0$, $M_2 = 2M(2r+10)$ to obtain $N_{21}\in\mathbb{N}$ such that 
	\begin{equation}\label{NoBigMaxInterp}
		\mathbb{P}^{0,s+s_4,0,b-a}_{H^{\mathrm{RW}}}\left(\ell(s_4-s_3) \geq \frac{s_4-s_3}{s+s_4}\left[p(s+s_4) + M_2 (s+s_4)^{1/2}\right] - (s+s_4)^{1/3}\right) \geq \frac{1}{3}
	\end{equation}
	for $N\geq N_{21}$. Note that $\frac{s_4-s_3}{s+s_4} \geq \frac{1}{2r+11}$ for sufficiently large $N$, so we have
	\begin{align*}
		\frac{s_4-s_3}{s+s_4}\left[p(s+s_4) + M_2 (s+s_4)^{1/2}\right] - (s+s_4)^{1/3} &\geq p(s_4-s_3) + M(s+s_4)^{1/2}\\
		&\geq p(s_4-s_3) + 2MN^{\gamma/2}.
	\end{align*}
	The last inequality follows since $s+s_4 \geq s_4 \geq 4N^\gamma$. In view of \eqref{NoBigMax2M} and \eqref{NoBigMaxInterp}, this implies \eqref{NoBigMax1/3} for $N\geq N_{21}$.\\
	
	\noindent\textbf{Step 3. } To conclude, we prove \eqref{NoBigMaxGE} for $N \geq N_{21}$, with $N_{21}$ as in Step 2. Let $J$ denote the set of $(a,b,s,\ell_{\mathrm{b}})\in I$ for which $\mathbb{P}(H(a,b,s,\ell_{\mathrm{b}})) > 0$. If $(a,b,s,\ell_{\mathrm{b}})\in J$, the $H^{\mathrm{RW}}$-Gibbs property (see Definition \ref{DefHRWGP}) together with \eqref{NoBigMax1/3a} implies that
	\begin{equation*}
		\mathbb{P}\left(L_1^N(-s_3) + ps_3 > MN^{\gamma/2}\,\Big|\,H(a,b,s,\ell_{\mathrm{t}},\ell_{\mathrm{b}})\right) = \mathbb{P}^{-s_4,s,a,b,\infty,\ell_{\mathrm{b}}}_{\mathrm{avoid},H^{\mathrm{RW}}}\left(\ell(-s_3) + ps_3 > MN^{\gamma/2}\right) \geq \frac{1}{3}
	\end{equation*}
	for all $N\geq N_{21}$. It follows from \eqref{NoBigMaxEF} that
	\begin{align*}
		\epsilon/12 &> \mathbb{P}(F) \geq \sum_{(a,b,s,\ell_{\mathrm{t}},\ell_{\mathrm{b}})\in J} \mathbb{P}(F\cap H(a,b,s,\ell_{\mathrm{t}},\ell_{\mathrm{b}}))\\
		&= \sum_{(a,b,s,\ell_{\mathrm{t}},\ell_{\mathrm{b}})\in J}\mathbb{P}\left(L_1^N(-s_3) + ps_3 > MN^{\gamma/2}\,\Big|\,H(a,b,s,\ell_{\mathrm{t}},\ell_{\mathrm{b}})\right)\mathbb{P}(H(a,b,s,\ell_{\mathrm{t}},\ell_{\mathrm{b}}))\\
		&\geq \frac{1}{3}\sum_{(a,b,s,\ell_{\mathrm{t}},\ell_{\mathrm{b}})\in I}\mathbb{P}(H(a,b,s,\ell_{\mathrm{t}},\ell_{\mathrm{b}})) = \frac{1}{3}\cdot\mathbb{P}(G\setminus E)
	\end{align*}
	for $N\geq N_{21}$. In the last line, we were able to replace $J$ with $I$ since the terms coming from $I\setminus J$ do not contribute to the sum, and the last equality follows from \eqref{NoBigMaxDisj}. This implies \eqref{NoBigMaxGE}, completing the proof.
	
\end{proof}

\subsection{Proof of Lemma \ref{NoLowMin}}\label{NLMpf}

We will first prove the following lemma, which says roughly that with high probability, the $(k-1)$st curve cannot dip uniformly low on scale $N^{\gamma/2}$ on any large interval. We will use this lemma in the proof of \ref{NoLowMin} in order to ``pin'' the curve at two points on opposite sides of the origin where it cannot be low. The key to the proof of this lemma is the parabolic shift assumption in the definition of a $(\gamma,p,\lambda,\theta)$-good sequence. In the case $k=3$, the rough idea is that if the second curve is uniformly low on a large interval $[a,b]$, then the top curve will not feel it and will look like a free bridge by the $H^{\mathrm{RW}}$-Gibbs property. Then at the midpoint of $[a,b]$, we expect the top curve to be close to the midpoint of a straight line segment, but this will be far from the value of the parabola at the midpoint. The same reasoning applies for larger $k$ once we make use of monotone coupling. We make this argument precise in the proof below.

\begin{lemma}\label{ParabolaMin} For any $r,\epsilon>0$ there exist $R, N_5\in\mathbb{N}$ such that for all $N\geq N_5$,
	\[
	\mathbb{P}\left(\max_{s\in[a,b]}\left(L_{k-1}^N(s) - ps\right) \leq -\left(\lambda R^2 + R^\theta + \phi(\epsilon/16)\right)N^{\gamma/2}\right) < \epsilon,
	\]
	where $a = \lfloor \lfloor r\rfloor N^\gamma\rfloor$, $b = \lfloor R N^\gamma\rfloor$. The same statement holds for the maximum on the interval $[-b',-a']$, where $b' = \lceil RN^\gamma\rceil$ and $a' = \lceil \lfloor r\rfloor N^\gamma\rceil$.
\end{lemma}

\begin{proof}
	We write $\eta = \epsilon/16$ for convenience. The proof of the statement with $[-b',-a']$ in place of $[a,b]$ is essentially the same, as will become clear. We will assume that $r\in\mathbb{Z}$, which causes no loss of generality due to the way $a$ is defined. We begin by establishing some notation to be used in the proof and defining the constant $R$ in the statement of the lemma. We fix the constant
	\begin{equation}\label{ParabolaR}
		\begin{split}
			C &= \sqrt{8\sigma^2\log\frac{3}{1-(11/12)^{1/(k-2)}}},
		\end{split}
	\end{equation}
	and we fix $R > r + 2$ as a positive integer with the same parity as $r$, large enough so that
	\begin{equation}\label{ParabolaRcond}
	\frac{\lambda}{4}(R-r)^2 - \left(\frac{R+r}{2}\right)^\theta - Ck\sqrt{R-r} > 3Ck + 2\phi(\eta).
	\end{equation}
	We may do so because $\theta < 2$. In particular, we note that $r,R$ and the midpoint $(r+R)/2$ are all integers. We write $\rho := (r+R)/2$, and $c = \lfloor \rho N^\gamma\rfloor \in [a,b]$. We will assume throughout that $N$ is large enough so that $\psi(N) > R+1$ and thus $\mathfrak{L}^N$ is defined at $b$. We may do so by condition (2) in Definition \ref{gooddef}. We now define the two events
	\[
	E = \left\{\max_{x\in[a,b]}\left(L_{k-1}^N(s) - ps\right) \leq -\left(\lambda R^2 + R^\theta + \phi(\eta)\right)N^{\gamma/2}\right\}.
	\]
	\[
	F = \left\{N^{-\gamma/2}\left(L_1^N(c) - pc\right) + \lambda\rho^2 < - (\phi(\eta) + \rho^\theta) \right\},
	\]
	To prove the lemma, we find $N_5$ so that 
	\begin{equation}\label{ParabolaFbd}
		\mathbb{P}(E) < \epsilon
	\end{equation}
	for all $N\geq N_5$. Let $G$ denote the subset of $E$ for which we have
	\begin{align*}
		&\left|N^{-\gamma/2}\left(L_1^N(a) - pa\right) + \lambda r^2\right| < \phi(\eta) + r^\theta,\\
		&\left|N^{-\gamma/2}\left(L_1^N(b) - pb\right) + \lambda R^2\right| < \phi(\eta) + R^\theta.
	\end{align*}
	By condition (3) in Definition \ref{gooddef}, we can find $N_{50}$ so that for all $N\geq N_{50}$, these two inequalities hold simultaneously with probability at least $1-4\eta$. Thus for $N\geq N_{50}$,
	\begin{equation}\label{ParabolaFsb}
		\mathbb{P}(E) \leq \mathbb{P}(G) + 4\eta < \mathbb{P}(G) + \epsilon/2.
	\end{equation}
	Note that condition (3) in Definition \ref{gooddef} allows us to find $N_{51}$ so that $\mathbb{P}(F) < 2\eta = \epsilon/8$ for $N\geq N_{51}$. To prove \eqref{ParabolaFbd}, we will find $N_{52}$ so that for $N\geq N_{52}$ we have
	\begin{equation}\label{ParabolaF|G}
		\mathbb{P}(F\,|\,G) \geq 1/4.
	\end{equation}
	Then for $N\geq \max(N_{51},N_{52})$ we have
	\[
	\mathbb{P}(G) = \frac{\mathbb{P}(F\cap G)}{\mathbb{P}(F\,|\,G)} \leq 4\mathbb{P}(F) < \epsilon/2,
	\]
	and in view of \eqref{ParabolaFsb} this proves \eqref{ParabolaFbd} for $N \geq N_5 := \max(N_{50},N_{51},N_{52})$. We split the proof of \eqref{ParabolaF|G} into three steps for clarity.\\
	
	\noindent\textbf{Step 1. } To simplify the problem of bounding the conditional probability in \eqref{ParabolaF|G}, we first decompose the event $G$ into a countable disjoint union
	\begin{equation}\label{ParabolaGdisj}
		G = \bigsqcup_{(\vec{x},\vec{y},\ell_{\mathrm{b}})\in I} H(\vec{x},\vec{y},\ell_{\mathrm{b}}),
	\end{equation}
	with the sets $I$ and $H(\vec{x},\vec{y},\ell_{\mathrm{b}})$ defined in the following. For $\vec{x},\vec{y} \in \mathfrak{W}_{k-2}$, and a discrete path $\ell_{\mathrm{b}}$ on $\llbracket a,b\rrbracket$, we let $H(\vec{x},\vec{y},\ell_{\mathrm{b}})$ denote the event that $L_i^N(a) = x_i$ and $L_i^N(b) = y_i$ for $i\in\llbracket 1,k-2\rrbracket$, and $L_{k-1}^N$ agrees with $\ell_{\mathrm{b}}$ on $\llbracket a,b\rrbracket$. Let $I$ denote the set of triples $(\vec{x},\vec{y},\ell_{\mathrm{b}})$ such that
	\begin{enumerate}[label=(\arabic*)]
		\item $\alpha(b-a) \leq y_i-x_i \leq \beta(b-a)$ for $i\in\llbracket 1,k-2\rrbracket$,
		
		\item $|N^{-\gamma/2}(x_1 - pa)+\lambda r^2| < \phi(\eta) + r^\theta$ and $|N^{-\gamma/2}(y_1-pb)+\lambda R^2| < \phi(\eta) + R^\theta$,
		
		\item $z_1\leq x_{k-2}$, $z_2\leq y_{k-2}$, and $\ell_{\mathrm{b}} \in \Omega(a,b,z_1,z_2)$,
		
		\item $\max_{x\in[a,b]} (\ell_{\mathrm{b}}(s)-ps) \leq -(\lambda R^2+R^\theta+\phi(\eta))N^{\gamma/2}$.
		
	\end{enumerate}
	It is clear that $I$ is countable since $x_i,y_i,z_1,z_2$ are integers and there are finitely many $\ell_{\mathrm{b}}$ for each choice of $z_1,z_2$. The three conditions together show that $G = \bigcup_I H(\vec{x},\vec{y},\ell_{\mathrm{b}})$, and it is also clear that the $H(\vec{x},\vec{y},\ell_{\mathrm{b}})$ are pairwise disjoint. This proves \eqref{ParabolaGdisj}.
	
	Now let $J$ denote the set of $(\vec{x},\vec{y},\ell_{\mathrm{b}})\in I$ for which $\mathbb{P}(H(\vec{x},\vec{y},\ell_{\mathrm{b}})) > 0$. In the following step, we will find $N_{52}$ so that for all $N\geq N_{52}$ and all $(\vec{x},\vec{y},\ell_{\mathrm{b}}) \in J$, we have
	\begin{equation}\label{ParabolaFcond}
		\mathbb{P}(F\,|\,H(\vec{x},\vec{y},\ell_{\mathrm{b}})) > 1/4.
	\end{equation}
	It follows from \eqref{ParabolaGdisj} that for $N\geq N_{52}$,
	\begin{align*}
		\mathbb{P}(F\,|\,G) &= \sum_{(\vec{x},\vec{y},\ell_{\mathrm{b}})\in J} \frac{\mathbb{P}(F\,|\,H(\vec{x},\vec{y},\ell_{\mathrm{b}}))\mathbb{P}(H(\vec{x},\vec{y},\ell_{\mathrm{b}}))}{\mathbb{P}(G)} \geq \frac{1}{4}\cdot\frac{\sum_{(\vec{x},\vec{y},\ell_{\mathrm{b}})\in J} \mathbb{P}(H(\vec{x},\vec{y},\ell_{\mathrm{b}}))}{\mathbb{P}(G)} = \frac{1}{4},
	\end{align*}
	which proves \eqref{ParabolaF|G}.\\
	
	\noindent\textbf{Step 2. } In this step we prove \eqref{ParabolaFcond}. Fix $(\vec{x},\vec{y},\ell_{\mathrm{b}})\in J$. We first observe by the $H^{\mathrm{RW}}$-Gibbs property (see Definition \ref{DefHRWGP}) and the definition of $F$ that
	\[
	\mathbb{P}(F\,|\,H(\vec{x},\vec{y},\ell_{\mathrm{b}})) = \mathbb{P}^{a,b,\vec{x},\vec{y},\infty,\ell_{\mathrm{b}}}_{\mathrm{avoid},H^{\mathrm{RW}}}\left(Q_1(c) - pc < -(\lambda R^2 + R^\theta + \phi(\eta))N^{\gamma/2}\right).
	\]
	Next, notice that by Lemma \ref{MCL}, the probability on the right will decrease if we raise the entry and exit data $\vec{x}$ and $\vec{y}$. We define new data $\vec{x}\,',\vec{y}\,'\in\mathfrak{W}_{k-2}$ as follows. Write
	\begin{equation}\label{ParabBar}
		\overline{x} = \lceil pa - (\lambda r^2 - r^\theta - \phi(\eta) )N^{\gamma/2}\rceil, \quad \overline{y} = \lceil pb - (\lambda R^2 - R^\theta - \phi(\eta))N^{\gamma/2}\rceil.
	\end{equation}
	Condition (2) in the definition of $I$ in Step 1 implies that $L_1^N(a) \leq \overline{x}$ and $L_1^N(b) \leq \overline{y}$. Now writing $T = b-a$, we put
	\[
	x_i' = \overline{x} + (k-1-i)\lceil C\sqrt{T}\rceil, \quad y_i' = \overline{y} + (k-1-i)\lceil C\sqrt{T}\rceil,
	\]
	for $i\in\llbracket 1,k-2\rrbracket$. Then $\vec{x}\,',\vec{y}\,'$ satisfy
	\begin{enumerate}[label=(\alph*)]
		\item $x_i' \geq \overline{x} \geq x_1 \geq x_i$ and $y_i' \geq \overline{y} \geq y_1 \geq y_i$ for $i\in\llbracket 1,k-2\rrbracket$,
		
		\item $x_i' - x_{i+1}' \geq C\sqrt{T}$ for $i\in\llbracket 1,k-3\rrbracket$,
		
		\item $x_{k-1}' + (y_{k-1}' - x_{k-1}')s/T - \ell_{\mathrm{b}}(s) \geq C\sqrt{T}$ for $s\in[0,T]$.
		
	\end{enumerate}	
	Note that the last condition follows from condition (4) in Step 1, since this implies that $\ell_{\mathrm{b}}$ lies uniformly below the line segment connecting $\overline{x}$ and $\overline{y}$. In particular, condition (a) and Lemma \ref{MCL} imply that
	\begin{equation}\label{ParabolaFsplit}
		\begin{split}
			&\mathbb{P}(F\,|\,H(\vec{x},\vec{y},\ell_{\mathrm{b}})) \geq \mathbb{P}^{a,b,\vec{x}\,',\vec{y}\,',\infty,\ell_{\mathrm{b}}}_{\mathrm{avoid},H^{\mathrm{RW}}}\left(Q_1(c) - pc < -(\lambda \rho^2 + \rho^\theta + \phi(\eta))N^{\gamma/2}\right)\\
			\geq \; & \mathbb{P}^{a,b,x_1',y_1'}_{H^{\mathrm{RW}}}\left(\ell(c) - pc < -(\lambda \rho^2 + \rho^\theta + \phi(\eta))N^{\gamma/2}\right) - \mathbb{P}^{a,b,\vec{x}\,',\vec{y}\,',\infty,\ell_{\mathrm{b}}}_{H^{\mathrm{RW}}}(A^c),
		\end{split}
	\end{equation}
	where $A = \{L_1 \geq \cdots \geq L_{k-1} \geq \ell_{\mathrm{b}} \mbox{ on } [a,b]\}$ is the avoidance event for a line ensemble $(L_1,\dots,L_{k-1})$ with law $\mathbb{P}^{a,b,\vec{x}\,',\vec{y}\,',\infty,\ell_{\mathrm{b}}}_{H^{\mathrm{RW}}}$. We implicitly used the fact that since the curves are independent under this law, the law of $L_1$ is $\mathbb{P}^{a,b,x_1',y_1'}_{H^{\mathrm{RW}}}$. 
	
	Observe that properties (b) and (c) of $\vec{x}\,',\vec{y}\,'$ above imply via Lemma \ref{CurveSep} that there is an $N_{53}$ independent of $\vec{x},\vec{y},\ell_{\mathrm{b}}$ so that for all $N\geq N_{53}$, we have
	\[
	\mathbb{P}^{a,b,\vec{x}\,',\vec{y}\,'}_{H^{\mathrm{RW}}}(A) \geq \left(1-3e^{-C^2/8\sigma^2}\right)^{k-2}.
	\]
	By our choice of $C$ in \eqref{ParabolaR}, the right hand side is equal to 11/12. In the following step, we will find $N_{54}$ independent of $\vec{x},\vec{y},\ell_{\mathrm{b}}$ so that for $N\geq N_{54}$ we have
	\begin{equation}\label{Parabola1/3}
		\mathbb{P}^{a,b,x_1',y_1'}_{H^{\mathrm{RW}}}\left(\ell(c) - pc < -(\lambda \rho^2 + \rho^\theta + \phi(\eta))N^{\gamma/2}\right) \geq 1/3. 
	\end{equation}
	This gives a lower bound of $1/3-1/12 = 1/4$ in \eqref{ParabolaFsplit} for $N\geq N_{52}$, proving \eqref{ParabolaFcond}.\\
	
	\noindent\textbf{Step 3. } Here we prove \eqref{Parabola1/3}. With notation as in Step 2, we have
	\begin{align*}
		&\mathbb{P}^{a,b,x_1',y_1'}_{H^{\mathrm{RW}}}\left(\ell(c) - pc < -(\lambda \rho^2 + \rho^\theta +\phi(\eta))N^{\gamma/2}\right)\\ 
		= \; & \mathbb{P}^{0,T,x_1',y_1'}_{H^{\mathrm{RW}}}\left(\ell(c - a) - pc < -(\lambda \rho^2 + \rho^\theta + \phi(\eta))N^{\gamma/2}\right)\\
		\geq \; & \mathbb{P}^{0,T,\overline{x},\overline{y}}_{H^{\mathrm{RW}}}\left(\ell(c - a) - pc < -\left(\lambda \rho^2 + \rho^\theta + \phi(\eta) + (k-2)\lceil C\sqrt{R-r+2}\rceil\right)N^{\gamma/2}\right)\\
		\geq \; & \mathbb{P}^{0,T,\overline{x},\overline{y}}_{H^{\mathrm{RW}}}\left(\ell(c - a) - \frac{\overline{x}+\overline{y}}{2} < \left(\lambda\left(\frac{R^2+r^2}{2} - \rho^2\right) - \rho^\theta - Ck\big(\sqrt{R-r}+3\big)-2\phi(\eta)\right)N^{\gamma/2}\right).
	\end{align*}
	Let us elaborate on these inequalities. The first equality follows simply by shifting the domain. In the third line, we vertically shifted the entry and exit data to $\overline{x},\overline{y}$ and used the definitions of $x_1',y_1'$ in terms of these. In the last line, we used the fact that $|pc - (\overline{x}+\overline{y})/2| \leq (\lambda(R^2+r^2)/2 + \phi(\eta) + 2)N^{\gamma/2}$. We absorbed the 2 by replacing $k-2$ with $k$, and we used the estimate $\sqrt{R-r+2} \leq \sqrt{R-r} + 2$ following from concavity of the square root function.
	
	Now let us write $M_0$ for the constant multiplying $N^{\gamma/2}$ on the right hand side of the last line in the tower of inequalities above, and $M := M_0\sqrt{R-r-2}$. Shifting vertically by $\overline{x}$, writing $z = \overline{y} - \overline{x}$, and using the fact that $\sqrt{T} \geq \sqrt{R-r-2}\,N^{\gamma/2}$, we obtain
	\begin{equation}\label{ParabolaM}
		\begin{split}
			&\mathbb{P}^{a,b,x_1',y_1'}_{H^{\mathrm{RW}}}\left(\ell(c) - pc < -(\lambda \rho^2+\rho^\theta+\phi(\eta))N^{\gamma/2}\right) \geq \mathbb{P}^{0,T,0,z}_{H^{\mathrm{RW}}}\left(\ell(c - a) - z/2 < M\sqrt{T}\right).
		\end{split}
	\end{equation}
	We claim that $M>0$. Since by assumption $R-r-2 > 0$, it suffices to show that $M_0 > 0$. Note that
	\[
	\frac{R^2+r^2}{2} - \rho^2 = \frac{R^2+r^2}{2} - \frac{R^2+r^2-2rR}{4} = \frac{R^2-2rR+r^2}{4} = \frac{(R-r)^2}{4},
	\]
	so
	\begin{align*}
		M_0 &= \frac{\lambda}{4}(R-r)^2 - \left(\frac{R+r}{2}\right)^\theta - Ck\sqrt{R-r} - \left(3Ck + 2\phi(\eta) \right).
	\end{align*}
	In view of \eqref{ParabolaRcond}, we see that the right hand side is positive, so $M>0$ as claimed.
	
	It now suffices to prove a lower bound of 1/3 in \eqref{ParabolaM} for large $N$. By Assumption \ref{KMT}, we have a probability measure $\mathbb{P}_0$ supporting a random variable $\ell^{(T,z)}$ with law $\mathbb{P}^{0,T,0,z}_{H^{\mathrm{RW}}}$, coupled with a Brownian bridge $B^\sigma$ with variance $\sigma^2$ depending on $p$. Let us write $s = c - a$. We see that
	\[
	|s-T/2| = \left|\left\lfloor\frac{R+r}{2}\,N^\gamma\right\rfloor - \lfloor rN^\gamma\rfloor - \frac{\lfloor RN^\gamma\rfloor - \lfloor rN^\gamma\rfloor}{2}\right| \leq 2.
	\]
	Since $z \leq pT - (\lambda(R^2-r^2) - (R^\theta-r^\theta))N^{\gamma/2} + 1$ in view of \eqref{ParabBar}, we can take $N$ large enough so that $(s/T - 1/2)z \leq 2z/T \leq 3p \leq M\sqrt{T}/2$. Then the right hand side of \eqref{ParabolaM} is equal to
	\begin{align*}
		\mathbb{P}^{0,T,0,z}_{H^{\mathrm{RW}}}\left(\ell(s) - z/2 < M\sqrt{T}\right) &= \mathbb{P}^{0,T,0,z}_{H^{\mathrm{RW}}}\left(\ell(s) - (s/T) z < M\sqrt{T} - (s/T - 1/2)z\right)\\
		&\geq \mathbb{P}^{0,T,0,z}_{H^{\mathrm{RW}}}\left(\ell(s) - (s/T) z < M\sqrt{T}/2\right)\\
		&\geq \mathbb{P}_0\left(\left[\ell^{(T,z)}(s) - (s/T)z - \sqrt{T}B^\sigma_{s/T}\right] + \sqrt{T} B^\sigma_{s/T} < M\sqrt{T}/2\right)\\
		&\geq \mathbb{P}_0\left(B^\sigma_{s/T} \leq 0 \quad \mathrm{and} \quad \Delta(T,z) < M\sqrt{T}/2 \right)\\ 
		&\geq \frac{1}{2} - \mathbb{P}_0\left(\Delta(T,z) \geq M\sqrt{T}/2\right).
	\end{align*}
	Here, $\Delta(T,z)$ is as defined in Assumption \ref{KMT}. For the first term in the last line, we used the fact that $\mathbb{P}_0(B^\sigma_{s/T} \leq 0) = 1/2$, independent of $T$. For the second term, we have
	\[
	|z-pT| = |\overline{y}-\overline{x}-pT| \leq (\lambda(R^2-r^2)-(R^\theta-r^\theta) + 2)N^{\gamma/2}
	\]
	so Corollary \ref{Cheb} allows us to choose $N_{54}$ large enough so that $\mathbb{P}_0(\Delta(T,z) \geq M\sqrt{T}) < 1/6$ for $N\geq N_{54}$. This gives a lower bound in \eqref{ParabolaM} of $1/2 - 1/6 = 1/3$ for $N\geq N_{54}$, proving \eqref{Parabola1/3}.
	
\end{proof}

We now prove Lemma \ref{NoLowMin}. The idea is to use Lemma \ref{ParabolaMin} to pin the curve $L_{k-1}^N$ at two points far to the left and to the right of zero, and then use Lemma \ref{BridgeInf} to control the curve in between these points.

\begin{proof}[Proof of Lemma \ref{NoLowMin}]
	
	To begin, we establish notation to be used in the proofs. We define events
	\[
	E = \left\{\inf_{s\in[-t_3,t_3]}\left(L_{k-1}^N(s) - ps\right) \leq -R_2N^{\gamma/2}\right\},
	\]
	\[
	F = \left\{\max_{x\in[c,d]} \left(L_{k-1}^N(xN^\gamma) - pxN^\gamma\right) > -MN^{\gamma/2}\right\} \cap \left\{\max_{x\in[-d',-c']} \left(L_{k-1}^N(xN^\gamma) - pxN^\gamma\right) > -MN^{\gamma/2}\right\}.
	\]
	Here, $c = \lfloor\lfloor r+3\rfloor N^\gamma\rfloor$, $d = \lfloor RN^\gamma\rfloor$, $d' = \lceil RN^\gamma\rceil$, $c' = \lceil \lfloor r+3\rfloor N^\gamma\rceil$, and the constants $R$ and $M := \lambda R^2 + \phi(\epsilon/64)$ are taken via Lemma \ref{ParabolaMin} so that there exists $N_{30}$ with
	\begin{equation}\label{NoLowMinFbd}
		\mathbb{P}(F) > 1-\epsilon/2
	\end{equation}
	for all $N\geq N_{30}$. We prove the lemma by finding $N_3\in\mathbb{N}$ and $R_2 > 0$ so that for all $N\geq N_3$ we have
	\begin{equation}\label{NoLowMinEbd}
		\mathbb{P}(E) < \epsilon.
	\end{equation}
	To do so, we find $N_{31}$ so that for $N\geq N_{31}$,
	\begin{equation}\label{NoLowMinEFbd}
		\mathbb{P}(E\cap F) < \epsilon/2.
	\end{equation}
	Then for $N\geq N_3 := \max(N_{30},N_{31})$, \eqref{NoLowMinFbd} and \eqref{NoLowMinEFbd} together imply that
	\[
	\mathbb{P}(E) \leq \mathbb{P}(E\cap F) + \mathbb{P}(F) < \epsilon,
	\]
	proving \eqref{NoLowMinEbd} and the lemma. We split the proof of \eqref{NoLowMinEFbd} into three steps for clarity.\\
	
	\noindent\textbf{Step 1. } To simplify the proof of \eqref{NoLowMinEFbd}, we first express $F$ as a disjoint union
	\begin{equation}\label{NoLowMinFdisj}
		F = \bigsqcup_{(a,b,\vec{x},\vec{y},\ell_{\mathrm{b}})\in I} H(a,b,\vec{x},\vec{y},\ell_{\mathrm{b}},\ell_{\mathrm{t}}^1, \ell_{\mathrm{t}}^2),
	\end{equation}
	with the sets $I$ and $H(a,b,\vec{x},\vec{y},\ell_{\mathrm{b}},\ell_{\mathrm{t}}^1, \ell_{\mathrm{t}}^2)$ defined as follows.
	
	For $a\in\llbracket -d',-c'\rrbracket$, $b\in\llbracket c,d\rrbracket$, $\vec{x},\vec{y}\in\mathfrak{W}_{k-1}$, and discrete paths $\ell_{\mathrm{b}}$ on $\llbracket a,b\rrbracket$, $\ell_{\mathrm{t}}^1$ on $\llbracket-d',a\rrbracket$, and $\ell_{\mathrm{t}}^2$ on $\llbracket b,d\rrbracket$, we define $H(a,b,\vec{x},\vec{y},\ell_{\mathrm{b}})$ to be the event that $L_i^N(a) = x_i$ and $L_i^N(b) = y_i$ for $i\in\llbracket 1,k-1\rrbracket$, $L_{k-1}^N$ agrees with $\ell_{\mathrm{t}}^1$ on $\llbracket-d',a\rrbracket$ and with $\ell_{\mathrm{t}}^2$ on $\llbracket b,d\rrbracket$, and $L_k^N$ agrees with $\ell_{\mathrm{b}}$ on $\llbracket a,b\rrbracket$. Let $I$ be the collection of tuples $(a,b,\vec{x},\vec{y},\ell_{\mathrm{b}})$ satisfying 
	\begin{enumerate}[label=(\arabic*)]
		
		\item $\alpha(b-a) \leq y_i - x_i \leq \beta(b-a)$ for $i\in\llbracket 1,k-1\rrbracket$, 
		
		\item $x_{k-1} - pa > - MN^{\gamma/2}$ and $y_{k-1} - pb > - MN^{\gamma/2}$,
		
		\item if $a'\in\llbracket -d',a\rrbracket$ and $a' < a$, then $\ell_{\mathrm{t}}^1(a') - pa' \geq -MN^{\gamma/2}$,
		
		\item if $b'\in\llbracket b,d\rrbracket$ and $b' > b$, then $\ell_{\mathrm{t}}^2(b') - pb' \geq -MN^{\gamma/2}$,
		
		\item $z_1,z_2\in\mathbb{Z}$, $z_1\leq x_{k-1}$, $z_2\leq y_{k-1}$, and $\ell_{\mathrm{b}}\in\Omega(-a,b,z_1,z_2)$.
		
	\end{enumerate} 
	It is clear that $I$ is countable, and the five conditions together imply that $F = \bigcup_I H(a,b,\vec{x},\vec{y},\ell_{\mathrm{b}})$. This proves \eqref{NoLowMinFdisj}.
	
	Now let $J$ denote the subset of $I$ for which $\mathbb{P}(H(a,b,\vec{x},\vec{y},\ell_{\mathrm{b}},\ell_{\mathrm{t}}^1,\ell_{\mathrm{t}}^2)) > 0$. In the next step, we will find $N_{31}$ so that for all $N\geq N_{31}$ and $(a,b,\vec{x},\vec{y},\ell_{\mathrm{b}},\ell_{\mathrm{t}}^1,\ell_{\mathrm{t}}^2)\in J$, we have
	\begin{equation}\label{NoLowMinEcond}
		\mathbb{P}(E\,|\,H(a,b,\vec{x},\vec{y},\ell_{\mathrm{b}},\ell_{\mathrm{t}}^1,\ell_{\mathrm{t}}^2)) < \epsilon/2.
	\end{equation}
	It then follows from \eqref{NoLowMinFdisj} that for $N\geq N_{31}$ we have
	\begin{align*}
		\mathbb{P}(E\cap F) &= \sum_{(a,b,\vec{x},\vec{y},\ell_{\mathrm{b}},\ell_{\mathrm{t}}^1,\ell_{\mathrm{t}}^2) \in J} \mathbb{P}(E\,|\,H(a,b,\vec{x},\vec{y},\ell_{\mathrm{b}},\ell_{\mathrm{t}}^1,\ell_{\mathrm{t}}^2)) \mathbb{P}(H(a,b,\vec{x},\vec{y},\ell_{\mathrm{b}},\ell_{\mathrm{t}}^1,\ell_{\mathrm{t}}^2))\\
		&< \frac{\epsilon}{2}\sum_{(a,b,\vec{x},\vec{y},\ell_{\mathrm{b}},\ell_{\mathrm{t}}^1,\ell_{\mathrm{t}}^2)\in J} \mathbb{P}(H(a,b,\vec{x},\vec{y},\ell_{\mathrm{b}},\ell_{\mathrm{t}}^1,\ell_{\mathrm{t}}^2)) \leq \epsilon/2.
	\end{align*}
	This proves \eqref{NoLowMinEFbd}.\\
	
	\noindent\textbf{Step 2. } In this step we find $N_{31}$ so that \eqref{NoLowMinEcond} holds for all $N\geq N_{31}$ and $(a,b,\vec{x},\vec{y},\ell_{\mathrm{b}},\ell_{\mathrm{t}}^1,\ell_{\mathrm{t}}^2) \in J$. We first observe by the $H^{\mathrm{RW}}$-Gibbs property (see Definition \ref{DefHRWGP}) that
	\[
	\mathbb{P}(E\,|\,H(a,b,\vec{x},\vec{y},\ell_{\mathrm{b}},\ell_{\mathrm{t}}^1,\ell_{\mathrm{t}}^2)) = \mathbb{P}^{a,b,\vec{x},\vec{y},\infty,\ell_{\mathrm{b}}}_{\mathrm{avoid},H^{\mathrm{RW}}}\left(\inf_{s\in[-t_3,t_3]} \left(Q_{k-1}(s) - ps\right) \leq -R_2N^{\gamma/2}\right).
	\] 
	We next observe by Lemma \ref{MCL} that if we lower the entry and exit data $\vec{x},\vec{y}$, then this probability will increase. We fix the constant
	\[
	C = \sqrt{16\sigma^2\log\frac{3}{1-2^{-1/(k-1)}}},
	\]
	and we define new data $\vec{x}\,',\vec{y}\,'\in\mathfrak{W}_{k-1}$ by
	\[
	x_i' = \lfloor pa - MN^{\gamma/2}\rfloor - (i-1)\lceil CN^{\gamma/2}\rceil, \quad y_i' = \lfloor pb - MN^{\gamma/2}\rfloor - (i-1)\lceil CN^{\gamma/2}\rceil.
	\]
	We also write $T = b-a$, $z = y_{k-1}' - x_{k-1}'$. Then $\vec{x}\,',\vec{y}\,'$ satisfy
	\begin{enumerate}[label=(\alph*)]
		\item $x_i' \geq x_1 \geq x_i$ and $y_i' \geq y_1 \geq y_i$ for $i\in\llbracket 1,k-2\rrbracket$,
		
		\item $x_i' - x_{i+1}' \geq C\sqrt{T}$ for $i\in\llbracket 1,k-2\rrbracket$,
		
		\item $x_{k-1}' + (z/T)s - \ell_{\mathrm{b}}(s) \geq C\sqrt{T}$ for all $s\in[0,T]$.
		
	\end{enumerate}	
	Condition (a) and Lemma \ref{MCL} imply that
	\begin{equation*}
		\begin{split}
			&\mathbb{P}(E\,|\,H(a,b,\vec{x},\vec{y},\ell_{\mathrm{b}})) \leq \mathbb{P}^{a,b,\vec{x}\,',\vec{y}\,',\infty,\ell_{\mathrm{b}}}_{\mathrm{avoid},H^{\mathrm{RW}}}\left(\inf_{s\in[a,b]}\left(Q_{k-1}(s) - ps\right) \leq -R_2N^{\gamma/2}\right)\\
			\leq \; & \mathbb{P}^{a,b,\vec{x}\,',\vec{y}\,'}_{\mathrm{avoid},H^{\mathrm{RW}}}\left(\inf_{s\in[a,b]}\left(Q_{k-1}(s) - ps\right) \leq -R_2N^{\gamma/2}\right) \leq \frac{\mathbb{P}^{a,b,\vec{x}\,',\vec{y}\,'}_{H^{\mathrm{RW}}}\left(\inf_{s\in[a,b]}\left(L_{k-1}(s) - ps\right) \leq -R_2N^{\gamma/2}\right)}{\mathbb{P}^{a,b,\vec{x}\,',\vec{y}\,'}_{H^{\mathrm{RW}}}(A)},
		\end{split}
	\end{equation*}
	where $A = \{L_1 \geq \cdots \geq L_{k-1} \geq \ell_{\mathrm{b}} \mbox{ on } [a,b]\}$ is the avoidance event for a line ensemble $(L_1,\dots,L_{k-1})$ with law $\mathbb{P}^{a,b,\vec{x}\,',\vec{y}\,'}_{H^{\mathrm{RW}}}$. In the next step, we will find $N_{31}$ independent of $a,b,\vec{x},\vec{y}$ so that for $N\geq N_{31}$,
	\begin{equation}\label{NoLowMininf}
		\mathbb{P}^{a,b,\vec{x}\,',\vec{y}\,'}_{H^{\mathrm{RW}}}\left(\inf_{s\in[a,b]}\left(L_{k-1}(s) - ps\right) \leq -R_2N^{\gamma/2}\right) \leq \epsilon/4,
	\end{equation}
	\begin{equation}\label{NoLowMin1/2}
		\mathbb{P}^{a,b,\vec{x}\,',\vec{y}\,'}_{H^{\mathrm{RW}}}(A) \geq 1/2.
	\end{equation}
	This gives $\mathbb{P}(E\,|\,H(a,b,\vec{x},\vec{y},\ell_{\mathrm{b}},\ell_{\mathrm{t}}^1,\ell_{\mathrm{t}}^2)) \leq \epsilon/2$ for $N\geq N_{31}$, proving \eqref{NoLowMinEcond}.\\
	
	\noindent\textbf{Step 3. } In this last step, we find $N\geq N_{31}$ so that \eqref{NoLowMininf} and \eqref{NoLowMin1/2} both hold for all $N\geq N_{31}$ and $(a,b,\vec{x},\vec{y},\ell_{\mathrm{b}},\ell_{\mathrm{t}}^1,\ell_{\mathrm{t}}^2)\in J$. Let us first prove \eqref{NoLowMininf}. We write $z = y_{k-1}' - x_{k-1}'$ and $T = b-a$ as in Step 2. Since $L_1,\dots,L_{k-1}$ are independent under $\mathbb{P}^{a,b,\vec{x}\,',\vec{y}\,'}_{H^{\mathrm{RW}}}$, the law of $L_{k-1}$ is simply $\mathbb{P}^{a,b,x_{k-1}',y_{k-1}'}_{H^{\mathrm{RW}}}$. Thus 
	\begin{align*}
		&\mathbb{P}^{a,b,\vec{x}\,',\vec{y}\,'}_{H^{\mathrm{RW}}}\left(\inf_{s\in[a,b]}\left(L_{k-1}(s) - ps\right) \leq -R_2N^{\gamma/2}\right) = \mathbb{P}^{a,b,x_{k-1}',y_{k-1}'}_{H^{\mathrm{RW}}}\left(\inf_{s\in[a,b]} \left(\ell(s) - ps\right) \leq -R_2N^{\gamma/2}\right)\\ 
		= \; & \mathbb{P}^{0,b-a,0,z}_{H^{\mathrm{RW}}}\left(\inf_{s\in[0,b-a]}\left(\ell(s) - p(s-a) - \lceil pa + MN^{\gamma/2}\rceil - (k-2)\lceil CN^{\gamma/2}\rceil\right) \leq -R_2N^{\gamma/2}\right)\\
		\leq \; &\mathbb{P}^{0,b-a,0,z}_{H^{\mathrm{RW}}}\left(\inf_{s\in[0,b-a]} \left(\ell(s) - ps\right) \leq -(R_2-M-Ck)N^{\gamma/2}\right).
	\end{align*}
	Here, we have shifted horizontally by $-a$ and vertically by $-x_{k-1}'$, and used the estimate $(k-2)(CN^{\gamma/2} + 1) + 1 \leq CkN^{\gamma/2}$, valid for sufficiently large $N$. Since $z \geq p(b-a)$, by Lemma \ref{BridgeInf}, we can find $R_2 > 0$ and $N_{32}$ depending on $M,C,k,\epsilon$ so that the probability in the last line is bounded above by $\epsilon/4$ as long as $b-a \geq N_{32}$. Since $b-a \geq 2rN^\gamma$, it suffices to take $N_{31} \geq (N_{32}/2r)^{1/\gamma}$. This gives \eqref{NoLowMininf} for $N\geq N_{31}$, independent of $a,b,\vec{x},\vec{y}$.
	
	Finally, we show that we can enlarge $N_{31}$ so that \eqref{NoLowMin1/2} holds for $N\geq N_{31}$. We see from the definitions of $\vec{x}\,',\vec{y}\,'$ that $|z-pT| \leq 1$. In combination with properties (b) and (c) of $\vec{x}\,',\vec{y}\,'$ in Step 2, this implies via Lemma \ref{CurveSep} with $\ell_{\mathrm{b}} = -\infty$ that we can find $N_{33}$ so that if $b-a\geq N_{33}$, then
	\[
	\mathbb{P}^{a,b,\vec{x}\,',\vec{y}\,'}_{H^{\mathrm{RW}}}(A) = \mathbb{P}^{0,T,\vec{x}\,',\vec{y}\,'}_{H^{\mathrm{RW}}}(A) \geq \left(1 - 3e^{-C^2/8\sigma^2}\right)^{k-1}.
	\]
	By our choice of $C$ in Step 2, this is equal to 1/2, which proves \eqref{NoLowMin1/2} for $T\geq N_{33}$. Since $T \geq 2rN^\gamma$, we have $T \geq N_{33}$ as long as $N \geq (N_{33}/2r)^{1/\gamma}$. We therefore take $N_{31} = \max((N_{32}/2r)^{1/\gamma},(N_{33}/2r)^{1/\gamma})$. This completes the proof.
	
\end{proof}

%% file: section6.tex
%
\section{Bounds on the acceptance probability}\label{Section6}

In this section we prove Proposition \ref{APProp}. In Section \ref{APpf} we state a technical lemma with which we give the proof of the proposition, and we prove the lemma in Section \ref{APtechpf}.

\subsection{Proof of Proposition \ref{APProp}}\label{APpf}

Throughout this section we will write $S = \llbracket -t_3,-t_1\rrbracket \cup \llbracket t_1,t_3\rrbracket$. We will use the following technical lemma to prove Proposition \ref{APProp}. Roughly, it states that if a collection of random walk bridges on $\llbracket -t_3,t_3\rrbracket$ are conditioned to avoid one another and a lower-bounding curve on $S$ but not necessarily on $\llbracket -t_1,t_1\rrbracket$, and if the data of the curves at $\pm t_3$ lie in some compact window on scale $\sqrt{2t_3}$ (note that this event is likely by Lemmas \ref{NoBigMax} and \ref{NoLowMin}), then the acceptance probability of the curves on $\llbracket -t_1, t_1\rrbracket$ will be lower bounded with nonzero probability. We prove the lemma in the following section.

\begin{lemma}\label{APtech}
	Fix $k\in\mathbb{N}$, $p\in(\alpha,\beta)$, and $M>0$. Then there exist $g,h > 0$ and $N_6\in\mathbb{N}$ so that the following holds for all $N\geq N_6$. Let $\ell_{\mathrm{b}}$ be a discrete path on $\llbracket -t_3,t_3\rrbracket$ with $-\infty\leq \ell_{\mathrm{b}} < +\infty$, and let $\vec{x},\vec{y}\in\mathfrak{W}_{k-1}$ be such that $\Omega_{\mathrm{avoid}}(-t_3,t_3,\vec{x},\vec{y},\infty,\ell_{\mathrm{b}}) \neq \varnothing$. Assume that
	\begin{enumerate}[label=(\arabic*)]
		
		\item $\sup_{s\in[-t_3,t_3]}\big(\ell_{\mathrm{b}}(s) - ps\big) < M\sqrt{2t_3}$,
		
		\item $x_1 \leq -pt_3 + M\sqrt{2t_3}$ and $x_{k-1} \geq \max\big(\ell_{\mathrm{b}}(-t_3), -pt_3 - M\sqrt{2t_3}\big)$,
		
		\item $y_1 \leq pt_3 + M\sqrt{2t_3}$ and $y_{k-1} \geq \max\big(\ell_{\mathrm{b}}(t_3), pt_3 - M\sqrt{2t_3}\big)$.
		
	\end{enumerate}
	Let $\tilde{\mathfrak{Q}}$ have law $\mathbb{P}^{-t_3,t_3,\vec{x},\vec{y},\infty,\ell_{\mathrm{b}}}_{\mathrm{avoid},H^{\mathrm{RW}};S}$. Then
	\begin{equation}\label{APtecheq}
		\mathbb{P}^{-t_3,t_3,\vec{x},\vec{y},\infty,\ell_{\mathrm{b}}}_{\mathrm{avoid},H^{\mathrm{RW}};S}\left(Z(-t_1,t_1,\tilde{\mathfrak{Q}}(-t_1),\tilde{\mathfrak{Q}}(t_1),\infty,\ell_{\mathrm{b}}\llbracket -t_1,t_1\rrbracket) \geq g\right) \geq h.
	\end{equation}
	
\end{lemma}

Using this lemma, we prove Proposition \ref{APProp}.

\begin{proof}
	
	Let $g,h$ be as in Lemma \ref{APtech}, and let us write $\mathfrak{Q}$ for a line ensemble with law $\mathbb{P}^{-t_3,t_3,\vec{x},\vec{y},\infty,\ell_{\mathrm{b}}}_{\mathrm{avoid},H^{\mathrm{RW}}}$. We split the proof into two steps.\\
	
	\noindent\textbf{Step 1.} We first show that \eqref{APtecheq} in the statement of Lemma \ref{APtech} implies that for all $\eta > 0$,
	\begin{equation}\label{APghe}
		\mathbb{P}^{-t_3,t_3,\vec{x},\vec{y},\infty,\ell_{\mathrm{b}}}_{\mathrm{avoid},H^{\mathrm{RW}}}\left(Z(-t_1,t_1,\mathfrak{Q}(-t_1),\mathfrak{Q}(t_1),\infty,\ell_{\mathrm{b}}\llbracket -t_1,t_1\rrbracket) \leq gh\eta\right) \leq \eta.
	\end{equation}
	Let us denote by $\mathbb{P}_S$ and $\widetilde{\mathbb{P}}_S$ the two measures on $\llbracket 1,k-1\rrbracket$ indexed line ensembles on $S$ obtained by restricting $\mathfrak{Q}$ and $\tilde{\mathfrak{Q}}$ to $S$, respectively. We observe that the Radon-Nikodym derivative between these two measures is given on line ensembles $\mathfrak{L} = (L_1,\dots,L_{k-1})$ on $S$ by
	\begin{equation}\label{APRN}
		\frac{d\mathbb{P}_S}{d\widetilde{\mathbb{P}}_S}(\mathfrak{L}) = \frac{\mathbb{P}_S(\mathfrak{L})}{\widetilde{\mathbb{P}}_S(\mathfrak{L})} = Z_S^{-1}\cdot Z(-t_1,t_1,\mathfrak{L}(-t_1),\mathfrak{L}(t_1),\infty,\ell_{\mathrm{b}}\llbracket -t_1,t_1\rrbracket),
	\end{equation}
	where $Z_S$ is the expectation of $Z(-t_1,t_1,\mathfrak{L}(-t_1),\mathfrak{L}(t_1),\infty,\ell_{\mathrm{b}}\llbracket -t_1,t_1\rrbracket)$ with respect to $\widetilde{\mathbb{P}}_S$. Now $(\mathfrak{L}(-t_1),\mathfrak{L}(t_1))$ has the same law as $(\tilde{\mathfrak{Q}}(-t_1),\tilde{\mathfrak{Q}}(t_1))$, so \eqref{APtecheq} implies that
	\[
	Z_S = \ex^{-t_3,t_3,\vec{x},\vec{y},\infty,\ell_{\mathrm{b}}}_{\mathrm{avoid},H^{\mathrm{RW}};S}\left(Z(-t_1,t_1,\tilde{\mathfrak{Q}}(-t_1),\tilde{\mathfrak{Q}}(t_1),\infty,\ell_{\mathrm{b}}\llbracket -t_1,t_1\rrbracket)\right) \geq gh.
	\]
	Let $F$ denote the event that $Z(-t_1,t_1,\mathfrak{L}(-t_1),\mathfrak{L}(t_1),\infty,\ell_{\mathrm{b}}\llbracket -t_1,t_1\rrbracket) \leq gh\eta$. Then using \eqref{APRN}, we find
	\begin{align*}
		\mathbb{P}_S(F) &= \ex_S\left[\mathbf{1}_F\right] = Z_S^{-1}\cdot  \widetilde{\ex}_S\left[\mathbf{1}_F\cdot Z(-t_1,t_1,\mathfrak{L}(-t_1),\mathfrak{L}(t_1),\infty,\ell_{\mathrm{b}}\llbracket -t_1,t_1\rrbracket)\right]\\
		&\leq Z_S^{-1} \cdot \widetilde{\ex}_S[\mathbf{1}_F \cdot gh\eta] \leq (gh)^{-1} gh\eta = \eta.
	\end{align*}
	Since the law of $(\mathfrak{L}(-t_1),\mathfrak{L}(t_1))$ under $\mathbb{P}_S$ is the law of $(\mathfrak{Q}(-t_1),\mathfrak{Q}(t_1))$, the left hand side of \eqref{APghe} is equal to $\mathbb{P}_S(F)$, so we conclude \eqref{APghe}.\\
	
	\noindent\textbf{Step 2.} We now prove Proposition \ref{APProp} using the result from Step 1. Fix $\epsilon > 0$, and for $N\in\mathbb{N}$ define the events
	\begin{align*}
		E_N &= \left\{L_1^N(\pm t_1) \pm pt_1 \leq M\sqrt{2t_3}\right\}\cap\left\{L_{k-1}^N(\pm t_1) \pm pt_1 \geq -M\sqrt{2t_3}\right\}\\ 
		&\qquad\cap \left\{\sup_{s\in[-t_3,t_3]} \big(L_k^N(s) - ps\big) \leq M\sqrt{2t_3}\right\}.
	\end{align*}
	Lemmas \ref{NoBigMax} and \ref{NoLowMin}, coupled with the fact that $L_k^N(s) \leq L_1^N(s)$ for all $s$ almost surely, allow us to find $N_{10}$ so that for $N\geq N_{10}$,
	\begin{equation}\label{APENc}
		\mathbb{P}(E_N^c) < \epsilon/2. 
	\end{equation}
	Let $g,h$ be as above, and put $\delta = gh\epsilon/2$. Define the event
	\[
	G = \left\{Z(-t_1,t_1,\vec{x},\vec{y},\infty,L_k^N\llbracket -t_1,t_1\rrbracket) < \delta\right\}
	\]
	and the $\sigma$-algebra $\mathcal{F} = \sigma(\mathfrak{L}^N(-t_3),\mathfrak{L}^N(t_3),L_k^N\llbracket -t_3,t_3\rrbracket)$. Then we claim that we can find $N_{11}$ so that for all $N\geq N_{11}$ we have
	\begin{align*}
		&\mathbb{P}(G\cap E_N) = \ex\left[\ex\left[\mathbf{1}_G \cdot \mathbf{1}_{E_N}\,|\,\mathcal{F}\right]\right] = \ex\left[\mathbf{1}_{E_N} \cdot \ex[\mathbf{1}_G\,|\,\mathcal{F}]\right]\\
		= \; & \ex\left[\mathbf{1}_{E_N}\cdot\mathbb{P}^{-t_3,t_3,\mathfrak{L}^N(-t_3),\mathfrak{L}^N(t_3),L_k^N\llbracket -t_3,t_3\rrbracket}_{\mathrm{avoid},H^{\mathrm{RW}}}\left(Z(-t_1,t_1,\mathfrak{L}(-t_1),\mathfrak{L}(t_1),\infty,L_k^N\llbracket -t_1,t_1\rrbracket) \leq \epsilon gh/2\right)\right]\\
		\leq \; & \ex[\mathbf{1}_{E_N}\cdot \epsilon/2] \leq \epsilon/2.
	\end{align*}
	The second equality uses the fact that $E_N\in\mathcal{F}$. The equality on the second line follows from the $H^{\mathrm{RW}}$-Gibbs property (see Definition \ref{DefHRWGP}). The first inequality on the third line is valid for $N$ larger than some $N_{11}$ by Step 1, since on the event $G$, the three conditions in the statement of Lemma \ref{APtech} are satisfied with $\mathfrak{L}^N(-t_3),\mathfrak{L}^N(t_3),L_k^N\llbracket -t_3,t_3\rrbracket$ in place of $\vec{x},\vec{y},\ell_{\mathrm{b}}$. Now if $N\geq N_1 := \max(N_{10},N_{11})$, then the above in combination with \eqref{APENc} implies that
	\[
	\mathbb{P}(G) \leq \mathbb{P}(G\cap E_N) + \mathbb{P}(E_N^c) < \epsilon/2 + \epsilon/2 = \epsilon.
	\]
	Recalling the definition of $G$, this proves Proposition \eqref{APProp}.
	
\end{proof}

\subsection{Proof of Lemma \ref{APtech}}\label{APtechpf}

Throughout this section, we will assume the same notation as in the statement of Lemma \ref{APtech}. Namely, we fix $k\in\mathbb{N}$, $p\in(\alpha,\beta)$, $M > 0$, a discrete path $\ell_{\mathrm{b}} : \llbracket - t_3,t_3\rrbracket \to \mathbb{R}\cup\{-\infty\}$, and $\vec{x},\vec{y}\in\mathfrak{W}_{k-1}$ such that $\Omega_{\mathrm{avoid}}(-t_3,t_3,\vec{x},\vec{y},\infty,\ell_{\mathrm{b}})\neq\varnothing$. We assume that
\begin{enumerate}[label=(\arabic*)]
	
	\item $\sup_{s\in[-t_3,t_3]}\big(\ell_{\mathrm{b}}(s) - ps\big) < M\sqrt{2t_3}$,
	
	\item $x_1 \leq -pt_3 + M\sqrt{2t_3}$ and $x_{k-1} \geq \max\big(\ell_{\mathrm{b}}(-t_3), -pt_3 - M\sqrt{2t_3}\big)$,
	
	\item $y_1 \leq pt_3 + M\sqrt{2t_3}$ and $y_{k-1} \geq \max\big(\ell_{\mathrm{b}}(t_3), pt_3 - M\sqrt{2t_3}\big)$.
	
\end{enumerate}
We write $S = \llbracket -t_3,-t_1\rrbracket \cup \llbracket t_1,t_3\rrbracket$, and we let $\mathfrak{Q}$ and $\tilde{\mathfrak{Q}}$ denote line ensembles with laws $\mathbb{P}^{-t_3,t_3,\vec{x},\vec{y},\infty,\ell_{\mathrm{b}}}_{\mathrm{avoid},H^{\mathrm{RW}}}$ and $\mathbb{P}^{-t_3,t_3,\vec{x},\vec{y},\infty,\ell_{\mathrm{b}}}_{\mathrm{avoid},H^{\mathrm{RW}};S}$ respectively. 

The proof relies on two auxiliary lemmas. The first provides conditions under which the lower bound of $g$ on the acceptance probability on $\llbracket -t_1,t_1\rrbracket$ is attained, as in \eqref{APtecheq}. Roughly, the conditions are that the entry and exit data at times $\pm t_1$ are separated from one another and from the path $\ell_{\mathrm{b}}$ by some small positive amount on scale $\sqrt{2t_3}$, and moreover that the data lie in a compact window on this scale.

\begin{lemma}\label{APg}
	Fix $\epsilon > 0$ and $U > 0$ such that $U \geq M +  (k-1) \epsilon$, and suppose $\vec{a}, \vec{b} \in \mathfrak{W}_{k-1}$ are such that 
	\begin{enumerate}
		\item $W(2t_3)^{1/2} \geq a_1 + p t_1 \geq a_{k-1} + pt_1 \geq (M +\epsilon) \sqrt{2t_3}$;
		\item $W(2t_3)^{1/2} \geq b_1 - p t_1 \geq b_{k-1} - pt_1 \geq (M + \epsilon) \sqrt{2t_3}$; 
		\item $a_i -a_{i+1} \geq \epsilon (2t_3)^{1/2}$ and $b_{i} - b_{i+1} \geq \epsilon \sqrt{2t_3}$ for $i = 1, \dots, k-2$.
	\end{enumerate}
	Then we can find $g = g(\epsilon, p,k) > 0$ and $N_6 \in \mathbb{N}$ such that for all $N \geq N_6$ we have 
	\begin{equation}\label{eqnRT}
		Z\big(  -t_1, t_1, \vec{a} ,\vec{b}, \ell_{\mathrm{b}}\llbracket -t_1, t_1\rrbracket\big) \geq g.
	\end{equation}
\end{lemma}

\begin{proof}
	This is essentially an immediate consequence of Lemma \ref{CurveSep}. Indeed, in view of condition (1) in the assumptions at the beginning of this section, the rightmost inequalities in conditions (1) and (2) in the hypothesis imply that $\ell_{\mathrm{b}}$ lies a distance of at least $\epsilon\sqrt{2t_3}$ uniformly below the line segment connecting $a_{k-1}$ and $b_{k-1}$; that is, the first condition in Lemma \ref{CurveSep} holds. Condition (3) in the hypothesis implies condition (2) in Lemma \ref{CurveSep} with $C = \epsilon$. Lastly, we see from conditions (1) and (2) that $|a_i - b_i - 2pt_3| \leq (U-M-\epsilon)\sqrt{2t_3}$ for each $i\in\llbracket 1,k-1\rrbracket$, so condition (3) in Lemma \ref{CurveSep} is also satisfied. We conclude that \eqref{eqnRT} holds with
	\[
	g = \left(\frac{1}{2} - \sum_{n=1}^\infty (-1)^{n-1}e^{-n^2\epsilon^2/8\sigma^2}\right)^{k-1}.
	\]
\end{proof}

The remainder of the proof consists of showing that the conditions listed in Lemma \ref{APg} occur with at least probability $h > 0$ under the law $\mathbb{P}^{-t_3,t_3,\vec{x},\vec{y},\infty,\ell_{\mathrm{b}}}_{\mathrm{avoid},H^{\mathrm{RW}}}$, with the data $a_i,b_i$ given by $Q_i(\pm t_1)$. We first show using Lemma \ref{CurvesHigh} that with positive probability, at times $\pm t_2$ the curves of the line ensemble $\tilde{\mathfrak{Q}}$ (defined at the beginning of this section) will be high above the path $\ell_{\mathrm{b}}$ but will also lie in a compact window on scale $\sqrt{2t_3}$. This is encoded in the following lemma.

\begin{lemma}\label{APt2}
	Given $R > 0$, there exist $U,V \geq M + R$, $h_1 > 0$ and $N_7 \in \mathbb{N}$ (depending on $R$) such that if $N \geq N_7$ we have
	\begin{equation}\label{eqnRT2}
		\mathbb{P}^{-t_3,t_3,\vec{x},\vec{y},\infty,\ell_{\mathrm{b}}}_{\mathrm{avoid},H^{\mathrm{RW}};S} \left(  U\sqrt{2t_3} \geq \tilde{Q}_1(\pm t_2) \mp p t_2 \geq \tilde{Q}_{k-1}(\pm t_2) \mp p t_2 \geq V\sqrt{2t_3}   \right) \geq h_1.
	\end{equation}
\end{lemma}

\begin{proof}
	We will prove that for appropriate choices of $U,V,h_1$, we can find $N_7$ so that for $N\geq N_7$ we have
	\begin{align}
		\mathbb{P}^{-t_3,t_3,\vec{x},\vec{y},\infty,\ell_{\mathrm{b}}}_{\mathrm{avoid},H^{\mathrm{RW}};S}\left(\tilde{Q}_{k-1}(\pm t_2) \mp pt_2 \geq V\sqrt{2t_3}\right) &\geq 2h_1, \label{5.10bound1}\\
		\mathbb{P}^{-t_3,t_3,\vec{x},\vec{y},\infty,\ell_{\mathrm{b}}}_{\mathrm{avoid},H^{\mathrm{RW}};S}\left(\tilde{Q}_1(\pm t_2) \mp pt_2 > U\sqrt{2t_3}\right) &\leq h_1. \label{5.10bound2}
	\end{align}
	Assuming the validity of the claim, we then observe that the probability in \eqref{eqnRT2} is bounded below by $2h_1 - h_1 = h_1$, proving the lemma. We will prove \eqref{5.10bound1} and \eqref{5.10bound2} in three steps. In particular, we fix $h_1$ in Step 1, $V$ in Step 2, and $U$ in Step 3.\\
	
	\noindent\textbf{Step 1.} In this step we prove that there exists $N_7$ so that \eqref{5.10bound1} holds for $N\geq N_7$, assuming a suitable choice of the constant $V$ which will be established in Step 2. We condition on the value of $\tilde{\mathfrak{Q}}$ at 0 and divide $\tilde{\mathfrak{Q}}$ into two independent line ensembles on $\llbracket -t_3,0\rrbracket$ and $\llbracket 0,t_3\rrbracket$. We first observe via Lemma \ref{MCL} that
	\begin{equation}\label{5.10MC}
		\mathbb{P}^{-t_3,t_3,\vec{x},\vec{y},\infty,\ell_{\mathrm{b}}}_{\mathrm{avoid},H^{\mathrm{RW}};S}\left(\tilde{Q}_{k-1}(\pm t_2) \mp pt_2 \geq V\sqrt{2t_3}\right) \geq \mathbb{P}^{-t_3,t_3,\vec{x},\vec{y}}_{\mathrm{avoid}, H^{\mathrm{RW}}; S}\left(\tilde{Q}_{k-1}(\pm t_2) \mp pt_2 \geq V\sqrt{2t_3}\right).
	\end{equation}
	We now fix a constant $K$ depending on $V$ to be specified in Step 2, and we define events
	\[
	E_{\vec{z}} = \left\{\big(\tilde{Q}_1(0),\dots,\tilde{Q}_{k-1}(0)\big) = \vec{z}\right\}, \quad X = \left\{ \vec{z}\in\mathfrak{W}_{k-1} : z_{k-1} \geq K_1(2t_3)^{1/2} \mbox { and } \mathbb{P}^{-t_3,t_3,\vec{x},\vec{y}}_{\mathrm{avoid},H^{\mathrm{RW}}; S}(E_{\vec{z}}) > 0\right\},
	\]
	and $E = \bigsqcup_{\vec{z} \in X} E_{\vec{z}}$. We can choose $N_{70}$ large enough so that $X$ is non-empty for $N\geq N_{70}$. By Lemma \ref{CurvesHigh} we can find $N_{71}$ so that
	\begin{equation}\label{5.10Ebound}
		\mathbb{P}^{-t_3,t_3,\vec{x},\vec{y}}_{\mathrm{avoid}, H^{\mathrm{RW}}; S}(E) \geq \mathbb{P}^{-t_3,t_3,\vec{x},\vec{y}}_{\mathrm{avoid},H^{\mathrm{RW}}; S}\left(\tilde{Q}_{k-1}(0) \geq K\sqrt{2t_3}\right) \geq A,
	\end{equation}
	for $N\geq\tilde{N}_1$, where $A = A(k,p,M,K) > 0$ is a constant given explicitly in \eqref{19ineq}.
	
	Now let $\tilde{Q}_i^1$ and $\tilde{Q}_i^2$ denote the restrictions of $\tilde{Q}_i$ to $\llbracket-t_3,0\rrbracket$ and $\llbracket0,t_3\rrbracket$ respectively for $1\leq i\leq k-1$, and write $S_1 := S\cap\llbracket -t_3,0\rrbracket$, $S_2 := S\cap\llbracket 0, t_3\rrbracket$. If $\vec{z}\in X$, then
	\begin{equation}\label{5.10split}
		\mathbb{P}^{-t_3,t_3,\vec{x},\vec{y}}_{\mathrm{avoid},H^{\mathrm{RW}};S}\left(\tilde{Q}^1_{k-1} = \ell_1, \tilde{Q}^2_{k-1} = \ell_2 \, \big|\, E_{\vec{z}}\right) = \mathbb{P}^{-t_3,0,\vec{x},\vec{z}}_{\mathrm{avoid},H^{\mathrm{RW}};S_1}(\ell_1)\cdot\mathbb{P}^{0,t_3,\vec{z},\vec{y}}_{\mathrm{avoid},H^{\mathrm{RW}};S_2}(\ell_2).
	\end{equation}
	In Step 2, we will find $N_{72}$ so that for $N\geq N_{72}$ we have
	\begin{equation}\label{5.10fourth}
		\begin{split}
			\mathbb{P}^{-t_3,0,\vec{x},\vec{z}}_{\mathrm{avoid},H^{\mathrm{RW}};S_1}\left(\tilde{Q}^1_{k-1}(-t_2) + pt_2 \geq V\sqrt{2t_3}\right) &\geq \frac{1}{4},\\
			\mathbb{P}^{0,t_3,\vec{x},\vec{z}}_{\mathrm{avoid},H^{\mathrm{RW}};S_2}\left(\tilde{Q}^2_{k-1}(t_2) - pt_2 \geq V\sqrt{2t_3}\right) &\geq \frac{1}{4}.
		\end{split}
	\end{equation}
	Using \eqref{5.10Ebound}, \eqref{5.10split}, and \eqref{5.10fourth}, we conclude that
	\[
	\mathbb{P}^{-t_3,t_3,\vec{x},\vec{y}}_{\mathrm{avoid}, H^{\mathrm{RW}};S}\left(\tilde{Q}_{k-1}(\pm t_2) \mp pt_2 \geq (2t_3)^{1/2}V_1^b\right) \geq \frac{A}{16}
	\]
	for $N\geq N_7 = \max(N_{70},N_{71},N_{72})$. In combination with \eqref{5.10MC}, this proves \eqref{5.10bound1} with
	\begin{equation}\label{APh1}
		h_1 := \frac{A}{16} = \frac{2^{3k/2-4}}{\pi^{k/2}\sigma^k}\left(1 - 2\sum_{n=1}^\infty (-1)^{n-1}e^{-2n^2/\sigma^2}\right)^{2k} e^{-2k(K+M+10k-4)^2/\sigma^2}.
	\end{equation}\\
	
	\noindent\textbf{Step 2.} In this step, we fix the constants $V$ and $K$ and prove the inequalities in \eqref{5.10fourth} from Step 1, using Lemma \ref{BridgeInterp}. We first separate the curves of $\tilde{\mathfrak{Q}}$ using monotone coupling; define the constant
	\[
	C = \sqrt{8\sigma^2\log\frac{3}{1-(11/12)^{1/(k-2)}}}.
	\] 
	We define new data $\vec{x}\,', \vec{z}\,', \vec{y}\,'$ at times $-t_2, 0, t_2$ respectively by
	\begin{align*}
		x_i' &= \lfloor -pt_3 - M\sqrt{2t_3}\rfloor - (i-1)\lceil C\sqrt{2t_3}\,\rceil,\\
		z_i' &= \lfloor K\sqrt{2t_3}\rfloor - (i-1)\lceil C\sqrt{2t_3}\,\rceil,\\
		y_i' &= \lfloor pt_3 - M\sqrt{2t_3}\rfloor - (i-1)\lceil C\sqrt{2t_3}\,\rceil.
	\end{align*}
	Note that $x_i' \leq x_{k-1} \leq x_i$ and $x_i' - x_{i+1}' \geq C\sqrt{2t_3}$ for $1\leq i\leq k - 1$, and likewise for $z_i',y_i'$. By Lemma \ref{MCL} we have
	\begin{equation}\label{5.10separate}
		\begin{split}
			&\mathbb{P}^{-t_3,0,\vec{x},\vec{z}}_{\mathrm{avoid},H^{\mathrm{RW}};S_1}\left(\tilde{Q}^1_{k-1}(-t_2) + pt_2 \geq V\sqrt{2t_3}\right) \geq \mathbb{P}^{-t_3,0,\vec{x}\,',\vec{z}\,'}_{\mathrm{avoid},H^{\mathrm{RW}};S_1}\left(\tilde{Q}^1_{k-1}(-t_2) + pt_2 \geq V\sqrt{2t_3}\right) \\ 
			\geq \; & \mathbb{P}^{-t_3,0,x_{k-1}',z_{k-1}'}_{H^{\mathrm{RW}}}\left(\ell_1(-t_2) + pt_2 \geq V\sqrt{2t_3}\right) - \left( 1 - \mathbb{P}^{-t_3,t_3,\vec{x}\,',\vec{z}\,'}_{H^{\mathrm{RW}}}\left(\tilde{Q}^1_1 \geq \cdots \geq \tilde{Q}_{k-1}^1\right)\right).
		\end{split}
	\end{equation} 
	To bound the first term on the second line, first note that 
	\begin{align*}
		x_{k-1}' &\geq \tilde{x} := -pt_3 - (M+C(k-1))\sqrt{2t_3},\\ 
		z_{k-1}' &\geq \tilde{z} := K\sqrt{2t_3} - C(k-1)\sqrt{2t_3}
	\end{align*} for sufficiently large $N$. Thus by Lemma \ref{BridgeInterp}, we can find $N_{73}$ so that for $N\geq N_{73}$,
	\begin{equation}\label{5.10third1}
		\mathbb{P}^{-t_3,0,x_{k-1}',z_{k-1}'}_{H^{\mathrm{RW}}}\left(\ell_1(-t_2) \geq \frac{t_2}{t_3}\,\tilde{x} + \frac{t_3-t_2}{t_3}\,\tilde{z} - (2t_3)^{1/4}\right) \geq \frac{1}{3}.
	\end{equation}  
	Moreover, as long as $N_{73}^\gamma > 2$, we have for $N\geq N_{73}$ that
	\begin{equation}\label{2r+5}
		\frac{t_3-t_2}{t_3} \geq 1 - \frac{(r+2)N^\alpha}{(r+3)N^\alpha - 1} > 1-\frac{r+2}{r+5/2} = \frac{1}{2r+5}.
	\end{equation}
	We now put
	\begin{equation}\label{5.10const}
		\begin{split}
			V &= 2M + Ck + R, \qquad K = (4r+10)V.
		\end{split}
	\end{equation}
	Note in particular that $V \geq M + R$. It follows from \eqref{2r+5} that 
	\begin{align*}
		\frac{t_2}{t_3}\,\tilde{x} + \frac{t_3-t_2}{t_3}\,\tilde{z} - (2t_3)^{1/4} &= -pt_2 - C(k-1)\sqrt{2t_3} - \frac{t_2}{t_3}\,M\sqrt{2t_3} + \frac{t_3-t_2}{t_3}\,K\sqrt{2t_3} - (2t_3)^{1/4} \\ &\geq  -pt_2 - Ck\sqrt{2t_3} - M\sqrt{2t_3} + \frac{1}{2r+5}\,K\sqrt{2t_3}\\ 
		&=  -pt_2 + (M + Ck + 2(M+R))\sqrt{2t_3}\\
		&> -pt_2 + V\sqrt{2t_3}.
	\end{align*}
	For the first inequality, we used the fact that $t_2/t_3 < 1$, and we assumed that $N_{73}$ is sufficiently large so that $C(k-1)\sqrt{2t_3} + (2t_3)^{1/4} \leq Ck\sqrt{2t_3}$ for $N\geq\tilde{N}_3$. Using \eqref{5.10third1}, we conclude for $N\geq N_{73}$ that
	\begin{equation}\label{5.10third2}
		\mathbb{P}^{-t_3,0,x_{k-1}',z_{k-1}'}_{H^{\mathrm{RW}}}\left(\ell_1(-t_2) + pt_2 \geq V\sqrt{2t_3}\right) \geq \frac{1}{3}.
	\end{equation}
	Since $|z_i'-x_i'-pt_2| \leq (K+M+1)\sqrt{2t_3}$, we have by Lemma \ref{CurveSep} and our choice of $C$ that the second probability in the second line of \eqref{5.10separate} is bounded below by
	\[
	\left(1-3e^{-C^2/8\sigma^2}\right)^{k-1} \geq 11/12
	\]
	for $N$ larger than some $N_{74}$. It follows from \eqref{5.10separate} and \eqref{5.10third2} that for $N\geq N_{72} := \max(N_{73},N_{74})$, we have
	\begin{equation*}
		\mathbb{P}^{-t_3,0,\vec{x},\vec{z}}_{\mathrm{avoid},H^{\mathrm{RW}};S_1}\left(\tilde{Q}^1_{k-1}(-t_2) + pt_2 \geq V\sqrt{2t_3}\right) \geq \frac{1}{3} - \frac{1}{12} = \frac{1}{4},
	\end{equation*}
	proving the first inequality in \eqref{5.10fourth}. The proof of the second inequality is similar.
	\\
	
	\noindent\textbf{Step 3.} In this last step, we fix $U$ and prove that we can enlarge $N_7$ from Step 1 so that \eqref{5.10bound2} holds for $N\geq N_7$. Let $C$ be as in Step 2, and define new entry and exit data $\vec{x}\,'', \vec{y}\,''\in\mathfrak{W}_{k-1}$ by
	\begin{align*}
		x_i'' &= \lceil -pt_3 + M\sqrt{2t_3}\,\rceil + (k-i)\lceil C\sqrt{2t_3}\,\rceil,\\
		y_i'' &= \lceil pt_3 + M\sqrt{2t_3}\,\rceil + (k-i)\lceil C\sqrt{2t_3}\,\rceil.
	\end{align*}
	Note that $x_i'' \geq x_1 \geq x_i$ and $x_i''-x_{i+1}'' \geq C\sqrt{2t_3}$, and likewise for $y_i''$. Moreover, $\ell_{\mathrm{b}}$ lies a distance of at least $C\sqrt{2t_3}$ uniformly below the line segment connecting $x_{k-1}''$ and $y_{k-1}''$. By Lemma \ref{MCL} we have
	\begin{align*}
		\mathbb{P}^{-t_3,t_3,\vec{x},\vec{y},\infty,\ell_{\mathrm{b}}}_{\mathrm{avoid},H^{\mathrm{RW}};S}\left(\tilde{Q}_1(\pm t_2) \mp pt_2 > U\sqrt{2t_3}\right) &\leq \mathbb{P}^{-t_3,t_3,\vec{x}\,'',\vec{y}\,'',\infty,\ell_{\mathrm{b}}}_{\mathrm{avoid},H^{\mathrm{RW}};S}\left(\sup_{s\in[-t_3,t_3]} \big[\tilde{Q}_1(s)-ps\big]\geq U\sqrt{2t_3}\right)\\
		&\leq \frac{\mathbb{P}^{-t_3,t_3,x_1'',y_1''}_{H^{\mathrm{RW}}}\left(\sup_{s\in[-t_3,t_3]} \big[\tilde{L}_1(s)-ps\big] \geq U\sqrt{2t_3}\right)}{\mathbb{P}^{-t_3,t_3,\vec{x}\,'',\vec{y}\,''}_{H^{\mathrm{RW}}}\left(\tilde{L}_1\geq\cdots\geq\tilde{L}_{k-1}\geq\ell_{\mathrm{b}}\right)}.
	\end{align*}
	By Lemma \ref{BridgeInf}, since $\min(x_1'' + pt_3, \, y_1'' - pt_3) \leq (M_1+C(k-1))\sqrt{2t_3}$, we can choose $U > V$ as well as $N_{75}$ large enough so that the numerator is bounded above by $h_1/2$ for $N\geq N_{75}$. Since $|y_i'' - x_i'' - 2pt_3| \leq 1$, our choice of $C$ and Lemma \ref{CurveSep} give $N_{76}$ so that the denominator is at least $11/12$ for $N\geq N_{76}$. This gives an upper bound of $12/11\cdot h_1/2 < h_1/2$ in the above as long as $N_7\geq\max(N_{75},N_{76})$, proving \eqref{5.10bound2}. This completes the proof of the lemma.
	
\end{proof}

We now complete the proof of Lemma \ref{APtech}. The rough idea is to first argue using Lemma \ref{DeltaSep} that under the conditions of Lemma \ref{APt2}, the curves of $\tilde{\mathfrak{Q}}$ will be separated by a positive distance at the two times halfway bewteen $\pm t_1$ and $\pm t_2$. We then argue that in this situation, the three conditions in Lemma \ref{APg} are satisfied with positive probability $h$, which will imply \eqref{APtecheq}.

Accordingly, we define
\begin{equation}\label{t12}
	t_{12} := \left\lfloor \frac{t_1+t_2}{2}\right\rfloor.
\end{equation}

\begin{proof}[Proof of Lemma \ref{APtech}] We first define several events which will appear in the proof. We let $\tilde{S} = \llbracket -t_2,-t_1\rrbracket \cup \llbracket t_1,t_2\rrbracket$, and for $\vec{c}, \vec{d} \in \mathfrak{W}_{k-1}$ we write $\tilde{\Omega}(\vec{c},\vec{d}) = \Omega_{\mathrm{avoid}}(-t_2, t_2, \vec{c}, \vec{d}, \infty, \ell_{\mathrm{b}}; \tilde{S})$ for brevity. For $s\in\tilde{S}$ we define
	\begin{equation}
		\begin{split}
			&A(\vec{c}, \vec{d},s) = \left\{\tilde{\mathfrak{Q}} \in \tilde\Omega(\vec{c},\vec{d}): \tilde Q_{k-1}(\pm s) \mp ps \geq (M + 1) \sqrt{2t_3}\right\}, \\
			&B(\vec{c},\vec{d},W,s) = \left\{ \tilde{\mathfrak{Q}} \in \tilde\Omega(\vec{c},\vec{d}): \tilde Q_{1}(\pm s) \mp ps \leq W\sqrt{2t_3} \right\},\\
			&C(\vec{c}, \vec{d}, \epsilon, s) = \left\{ \tilde{\mathfrak{Q}} \in \tilde\Omega(\vec{c},\vec{d}): \min_{1\leq i\leq k-2} \big[\tilde Q_{i}(\pm s) - \tilde Q_{i+1}(\pm s)\big] \geq \epsilon \sqrt{2t_3} \right\},\\
			&D(\vec{c},\vec{d},W,\epsilon,s) = A(\vec{c}, \vec{d},s) \cap B(\vec{c},\vec{d},W,s) \cap C(\vec{c}, \vec{d}, \epsilon, s).
		\end{split}
	\end{equation}
	Here, $\epsilon > 0$ and $W \geq M + (k-1)\epsilon$ are constants which we will specify later. By Lemma \ref{APg}, for all $(\vec{c},\vec{d})$ and $N$ sufficiently large we have \[
	D(\vec{c},\vec{d},W,\epsilon,s)  \subset \left\{Z\left(  -t_1, t_1, \mathfrak{Q}(-t_1), \mathfrak{Q}(t_1), \ell_{\mathrm{b}}\llbracket -t_1, t_1\rrbracket\right) > g\right\}
	\]
	for some $g$ depending on $\epsilon,M,W$. Thus we will prove that probability of the event on the left under the uniform measure on $\tilde\Omega(\vec{c},\vec{d})$ is bounded below by $h := h_1/2$, with $h_1$ as in \eqref{APh1}. We divide the proof into several steps.\\
	
	{\bf \raggedleft Step 1.} We first show that there exist $R > 0$ and $\bar N_{60}$ large enough so that if $c_{k-1} + pt_2 \geq (M + R)\sqrt{2t_3}$ and $d_{k-1} - pt_2 \geq (M+R)\sqrt{2t_3}$, then for all $s\in\tilde{S}$ and $N\geq \bar N_{60}$ we have
	\begin{equation}\label{6.2step1}
		\begin{split}
			\mathbb{P}^{-t_2,t_2,\vec{c},\vec{d},\infty,\ell_{\mathrm{b}}}_{\mathrm{avoid},H^{\mathrm{RW}}; \tilde S}\big(A(\vec{c},\vec{d},s)\big) \geq  \frac{19}{20} \quad \mathrm{and} \quad \mathbb{P}^{-t_2,t_2,\vec{c},\vec{d}}_{\mathrm{avoid},H^{\mathrm{RW}};\tilde S}\left(Q_{k-1}(s) \geq \ell_{\mathrm{b}}(s) \mbox{ for all } s\in\tilde{S}\right) \geq \frac{99}{100}.
		\end{split}
	\end{equation} 
	Let us begin with the first inequality. We first apply monotone coupling to separate the curves of $\mathfrak{Q}$. Fix the constant
	\begin{equation}\label{6.2C}
		C = \sqrt{8\sigma^2\log\frac{3}{1-(199/200)^{1/(k-1)}}}
	\end{equation}
	and data $\vec{c}\,', \vec{d}\,' \in \mathfrak{W}_k$ with
	\begin{align*}
		c_i' &= \lfloor -pt_2 + (M+R)\sqrt{2t_3}\rfloor - (i-1)\lceil C\sqrt{2t_2}\,\rceil,\\
		d_i' &= \lfloor pt_2 + (M+R)\sqrt{2t_3}\rfloor - (i-1)\lceil C\sqrt{2t_2}\,\rceil.
	\end{align*}
	Then by Lemma \ref{MCL} we have
	\begin{equation}\label{6.2step1split}
		\begin{split}
			&\mathbb{P}^{-t_2,t_2,\vec{c},\vec{d},\infty,\ell_{\mathrm{b}}}_{\mathrm{avoid},H^{\mathrm{RW}}; \tilde S}\big(A(\vec{c},\vec{d},s)\big) \geq \mathbb{P}^{-t_2, t_2, \vec{c}\,', \vec{d}\,'}_{\mathrm{avoid}, H^{\mathrm{RW}}; \tilde S}\big(A(\vec{c}\,', \vec{d}\,',s)\big) \\ 
			\geq \; & \mathbb{P}^{-t_2, t_2, c_{k-1}', d_{k-1}'}_{H^{\mathrm{RW}}}\left(\inf_{s\in \tilde S}\big[\ell(s) - ps\big] \geq (M+1)\sqrt{2t_3}\right)\\ 
			\quad \; & -  \left( 1 - \mathbb{P}^{-t_2, t_2, \vec{c}\,', \vec{d}\,'}_{H^{\mathrm{RW}}}\left(L_1 \geq \cdots \geq L_{k-1}\right)\right).
		\end{split}
	\end{equation}
	By Lemma \ref{CurveSep} and the choice of $C$, we can find $N_{60}$ so that $\mathbb{P}^{-t_2, t_2, \vec{c}\,', \vec{d}\,'}_{H^{\mathrm{RW}}}(L_1 \geq \cdots \geq L_{k-1})>199/200 > 39/40$ for $N\geq N_{60}$. Writing $z = d_{k-1}' - c_{k-1}'$, the quantity in the second line of \eqref{6.2step1split} is equal to
	\begin{align*}
		&\mathbb{P}^{-t_2, t_2, 0, z}_{H^{\mathrm{RW}}}\Big(\inf_{s\in \tilde S}\big[\ell(s) + c_{k-1}' - ps\big]  \geq (M+ 1)\sqrt{2t_3}\Big) \\
		\geq \; & \mathbb{P}^{0, 2t_2, 0, z}_{H^{\mathrm{RW}}}\Big(\inf_{s\in [0,2t_2]}\big[\ell(s) - ps\big] \geq (-R+Ck+1)\sqrt{2t_3}\Big).
	\end{align*}
	Here we used the fact that $c_{k-1}' \geq -pt_2 + (M+R-Ck)\sqrt{2t_3}$. Now by Lemma \ref{BridgeInf}, we can choose $R$ and $N_{61}$ large enough so that this last probability is greater than $39/40$ for $N \geq N_{61}$. This gives a lower bound in \eqref{6.2step1split} of $39/40-1/40 = 19/20$ for $N\geq\max(N_{60},N_{61})$, proving the first inequality in \eqref{6.2step1}.
	
	We prove the second inequality in \eqref{6.2step1} similarly. Note that since $\ell_{\mathrm{b}}(s) \leq ps + M\sqrt{2t_3}$ on $\llbracket-t_3,t_3\rrbracket$ by assumption, we have
	\begin{equation*}
		\begin{split}
			\mathbb{P}^{-t_2,t_2,\vec{c},\vec{d}}_{\mathrm{avoid},H^{\mathrm{RW}};\tilde S}\left(\tilde Q_{k-1}(s) \geq \ell_{\mathrm{b}}(s) \mbox{ for all } s\in\tilde{S} \right) &\geq \mathbb{P}^{-t_2, t_2, \vec{c},\vec{d}}_{\mathrm{avoid}, H^{\mathrm{RW}}; \tilde S}\left(\inf_{s\in[-t_2, t_2]} \big[Q_{k-1}(s) - ps\big] \geq M\sqrt{2t_3}\right)\\
			&\geq \mathbb{P}^{-t_2, t_2, \vec{c}\,',\vec{d}\,'}_{\mathrm{avoid}, H^{\mathrm{RW}};\tilde S}\left(\inf_{s\in[-t_2, t_2]} \big[\tilde Q_{k-1}(s) - ps\big] \geq M\sqrt{2t_3}\right)\\
			&\geq\mathbb{P}^{0, 2t_2, 0, z}_{H^{\mathrm{RW}}}\left(\inf_{s\in[0, 2t_2]} \big[\ell(s) - ps\big] \geq -(R-Ck)\sqrt{2t_3}\right)\\ 
			&\qquad - \left(1 - \mathbb{P}^{-t_2,t_2,\vec{c}\,',\vec{d}\,'}_{H^{\mathrm{RW}}}(\tilde L_1\geq \cdots \geq \tilde L_{k-1})\right).
		\end{split}
	\end{equation*}
	We enlarge $R$ if necessary using Lemma \ref{BridgeInf} so that the probability in the third line $>199/200$ for $N\geq N_{62}$, and Lemma \ref{CurveSep} implies as above that the probability in the last line $> 199/200$ for $N$ larger than some $N_{63}$. This gives us a lower bound of $199/200 - 1/200 = 99/100$ for $N\geq \max(N_{62},N_{63})$ as desired. This proves the two inequalities in \eqref{6.2step1} for $N\geq \bar N_{60} := \max(N_{60},N_{61},N_{62},N_{63})$.\\
	
	{\bf \raggedleft Step 2.} With $R$ fixed from Step 1, let $U,V,h_1$ be as in Lemma \ref{APt2}  for this choice of $R$. Define the event
	\begin{equation}\label{6.2E}
		\begin{split}
			E = \big\{ \vec{c}, \vec{d} \in \mathfrak{W}_{k-1} : &\; U\sqrt{2t_3} \geq \max(c_1 + p t_2 , d_1 - pt_2) \mbox{ and }\\
			&\; \min(c_{k-1} + p t_2, d_{k-1} - pt_2)  \geq V\sqrt{2t_3} \big\}.
		\end{split}
	\end{equation}
	We show in this step that there exists $W \geq M + k-1$ and $\bar{N}_{61}$ such that for all $(\vec{c}, \vec{d}) \in E$, $s\in\tilde{S}$, and $N\geq\bar{N}_{61}$ we have
	\begin{equation}\label{6.2step2}
		\mathbb{P}^{-t_2,t_2,\vec{c},\vec{d},\infty,\ell_{\mathrm{b}}}_{\mathrm{avoid},H^{\mathrm{RW}};\tilde S}\big(B(\vec{c},\vec{d},U,s)\big) \geq  \frac{19}{20}.
	\end{equation}
	Let $C$ be as in \eqref{6.2C}, and define data $\vec{c}\,'', \vec{d}\,'' \in \mathfrak{W}_{k-1}$ by
	\begin{align*}
		c_i'' &= \lceil -pt_2 + W\sqrt{2t_3}\,\rceil + (k-1-i)\lceil C\sqrt{2t_2}\,\rceil,\\
		d_i'' &= \lceil pt_2 + W\sqrt{2t_3}\,\rceil + (k-1-i)\lceil C\sqrt{2t_2}\,\rceil.
	\end{align*}
	Then $c_i'' \geq c_1 \geq c_i$ and $c_i'' - c_{i+1}'' \geq C(2t_2)^{1/2}$ for each $i$, and likewise for $d_i''$. Furthermore, since $V \geq M+R$, we see that $\ell_{\mathrm{b}}$ lies a distance of at least $R\sqrt{2t_3} > C\sqrt{2t_3}$ uniformly below the line segment connecting $c_{k-1}''$ and $d_{k-1}''$. By Lemma \ref{MCL}, the left hand side of \eqref{6.2step2} is bounded below by
	\begin{equation}\label{6.2step3split}
		\begin{split}
			&\mathbb{P}^{-t_2,t_2,\vec{c}\,'', \vec{d}\,'', \infty,\ell_{\mathrm{b}}}_{\mathrm{avoid}, H^{\mathrm{RW}};\tilde S}\left(\sup_{s\in \tilde S}\big[\tilde Q_1(s) - ps\big] \leq U\sqrt{2t_3}\right)\\
			\geq \; & \mathbb{P}^{0,2t_2,0,z'}_{H^{\mathrm{RW}}}\left(\sup_{s\in[-t_2,t_2]}\big[\ell(s) - ps\big] \leq (W-U-Ck)(2t_3)^{1/2}\right) \\ &\qquad - \left(1 - \mathbb{P}^{-t_2,t_2,\vec{c}\,'', \vec{d}\,'', \infty,\ell_{\mathrm{b}}}_{H^{\mathrm{RW}}}\left(L_1\geq\cdots\geq L_{k-1}\geq \ell_{\mathrm{b}}\right)\right).
		\end{split}
	\end{equation}
	In the last line, we have written $z' = d_1'' - c_1''$, and we used the fact that $c_1'' \leq -pt_2 + (U + Ck)\sqrt{2t_3}$. Lemma \ref{BridgeInf} allows us to find $W$ large enough depending on $U,C,k,p$ and $N_{64}$ so that the probability in the third line of \eqref{6.2step3split} is at least $39/40$ for $N\geq N_{64}$. On the other hand, the above observations regarding $\vec{c}\,''$, $\vec{d}\,''$, and $\ell_{\mathrm{b}}$, as well as the fact that $|d_1'' - c_1'' - 2pt_2| \leq 1$, allow us to conclude from Lemma \ref{CurveSep} that the probability in the last line of \eqref{6.2step3split} is at least $39/40$ for $N$ greater than some $N_{65}$. This gives a lower bound of $39/40 - 1/40 = 19/20$ in \eqref{6.2step3split} for $\bar{N}_{61} := \max(N_{64},N_{65})$ as desired.\\
	
	{\bf \raggedleft Step 3.} In this step, we show that with $E$ and $W$ as in Step 2, there exist $\epsilon > 0$ sufficiently small and $\bar{N}_{62}$ such that for $(\vec{c}, \vec{d}) \in E$ and $N\geq\bar{N}_{62}$, we have
	\begin{equation}\label{LemmaBP2Step3}
		\mathbb{P}^{-t_2,t_2,\vec{c},\vec{d},\infty,\ell_{\mathrm{b}}}_{\mathrm{avoid},H^{\mathrm{RW}};\tilde S}\big(D(\vec{c},\vec{d},W,\epsilon,t_{12}) \big) \geq \frac{1}{2}.
	\end{equation}
	We claim that this follows if we find $N_{66}$ so that for $N\geq N_{66}$,
	\begin{equation}\label{6.2step3cond}
		\mathbb{P}^{-t_2,t_2,\vec{c},\vec{d}}_{\mathrm{avoid},H^{\mathrm{RW}};\tilde S}\big(C(\vec{c},\vec{d},\epsilon,t_{12})\,|\,A(\vec{c},\vec{d},t_1) \cap B(\vec{c},\vec{d},W,t_1)\big) \geq \frac{9}{10}.
	\end{equation}
	To see this, note that \eqref{6.2step1} and \eqref{6.2step2} imply that for $N\geq\max(\bar{N}_{60},\bar{N}_{61})$ we have
	\[
	\mathbb{P}^{-t_2,t_2,\vec{c},\vec{d}}_{\mathrm{avoid},H^{\mathrm{RW}};\tilde S}\big(A(\vec{c},\vec{d},t_1) \cap B(\vec{c},\vec{d},V^{top},t_1)\big) \geq \frac{19}{20} - \frac{1}{20} - \frac{1}{100} > \frac{4}{5},
	\]
	and then \eqref{6.2step3cond} and the second inequality in \eqref{6.2step1} imply for $N\geq \bar{N}_{62} := \max(\bar{N}_{60},\bar{N}_{61},N_{66})$ that
	\[
	\mathbb{P}^{-t_2,t_2,\vec{c},\vec{d},\infty,\ell_{\mathrm{b}}}_{\mathrm{avoid},H^{\mathrm{RW}};\tilde S}\big(A(\vec{c},\vec{d},t_1) \cap B(\vec{c},\vec{d},V^{top},t_1) \cap C(\vec{c},\vec{d},\epsilon,t_{12})\big) > \frac{9}{10}\cdot\frac{4}{5} - \frac{1}{100} > \frac{17}{25}.
	\]
	Then using \eqref{6.2step1} and \eqref{6.2step2} once again and recalling the definition of $D(\vec{c},\vec{d},W,\epsilon,t_{12}) $ gives a lower bound on the probability in \eqref{LemmaBP2Step3} of $17/25 - 1/10 > 14/25 > 1/2$ for $N\geq\bar{N}_2$ as desired. 
	
	In the remainder of this step, we verify \eqref{6.2step3cond}. Observe that $A(\vec{c},\vec{d},t_1) \cap B(\vec{c},\vec{d},W,t_1)$ can be written as a countable disjoint union: 
	\begin{equation}\label{6.2step3disj}
		A(\vec{c},\vec{d},t_1) \cap B(\vec{c},\vec{d},V^{top},t_1) = \bigsqcup_{(\vec{a},\vec{b})\in I} F(\vec{a},\vec{b}).
	\end{equation}
	Here, for $\vec{a},\vec{b}\in\mathfrak{W}_{k-1}$, $F(\vec{a},\vec{b})$ is the event that $\mathfrak{Q}(-t_1) = \vec{a}$ and $\mathfrak{Q}(t_1) = \vec{b}$, and $I$ is the collection of pairs $(\vec{a},\vec{b})$ satisfying
	\begin{enumerate}[label = (\arabic*)]
		
		\item $ \alpha(t_2-t_1) \leq \min(a_i - c_i,\, d_i - b_i) \leq \beta(t_2 - t_1)$ and $2\alpha t_1\leq b_i-a_i \leq 2\beta t_1$ for $1\leq i\leq k-1$,
		
		\item $\min(a_{k-1} + pt_1,\, b_{k-1} - pt_1) \geq (M+1)\sqrt{2t_3}$,
		
		\item $\max(a_1 + pt_1,\, b_1 - pt_1) \leq W\sqrt{2t_3}$.
		
	\end{enumerate}
	We denote by $\mathfrak{Q}^1 = (Q^1_1,\dots,Q^1_{k-1})$ and $\mathfrak{Q}^2 = (Q^2_2,\dots,Q^2_{k-1})$ the restrictions of $\tilde{\mathfrak{Q}}$ to $\llbracket -t_2,-t_1\rrbracket$ and $\llbracket t_1,t_2\rrbracket$ respectively. Then we observe that
	\begin{equation}\label{6.2step3ind}
		\begin{split}
			&\mathbb{P}^{-t_2,t_2,\vec{c},\vec{d}}_{\mathrm{avoid}, H^{\mathrm{RW}}; \tilde S}\left(\mathfrak{Q}^1 = \mathfrak{H}^1, \mathfrak{Q}^2 = \mathfrak{H}^2\,\big|\,F(\vec{a},\vec{b})\right) = \mathbb{P}^{-t_2,-t_1,\vec{c},\vec{a}}_{\mathrm{avoid}, H^{\mathrm{RW}}}\left(\mathfrak{Q}^1 = \mathfrak{H}^1\right) \cdot \mathbb{P}^{t_1,t_2,\vec{b},\vec{d}}_{\mathrm{avoid}, H^{\mathrm{RW}}}\left(\mathfrak{Q}^2 = \mathfrak{H}^2\right).
		\end{split}
	\end{equation}
	We also define $\tilde{I} := \{(\vec{a},\vec{b})\in I : \pr^{-t_2,t_2,\vec c, \vec d}_{\mathrm{avoid},H^{\mathrm{RW}};\tilde{S}}(F(\vec{a},\vec{b})) > 0\}$, which is nonempty for sufficiently large $N$. We now fix $(\vec{a},\vec{b})$ and argue that we can choose $\epsilon > 0$ small enough and $N_{66}$ so that for $N\geq N_{66}$,
	\begin{equation}\label{6.2step3 9/10}
		\pr^{-t_2,t_2,\vec c, \vec d}_{\mathrm{avoid},H^{\mathrm{RW}};\tilde{S}}\left(C(\vec{c},\vec{d},\epsilon,t_{12})\,\big|\, F(\vec{a},\vec{b})\right) \geq \frac{9}{10}.
	\end{equation}
	Enlarging $N_{66}$ if necessary so that $\tilde{I}\neq\varnothing$, using \eqref{6.2step3 9/10} and \eqref{6.2step3disj} and summing over the pairs in $\tilde{I}$ proves \eqref{6.2step3cond} for $N\geq N_{66}$.
	
	To prove \eqref{6.2step3 9/10}, we first show using Lemma \ref{DeltaSep} that we can find $\delta > 0$ and $N_{67}$ so that
	\begin{equation}\label{6.2step3left}
		\mathbb{P}^{-t_2,-t_1,\vec{c},\vec{a}}_{\mathrm{avoid}, H^{\mathrm{RW}}}\left(\max_{1\leq i\leq k-2} \big[Q^1_i(-t_{12}) - Q^1_{i+1}(-t_{12})\big] \geq \delta\sqrt{2t_3}\right) \geq \frac{3}{\sqrt{10}}
	\end{equation}
	for $N\geq N_{67}$. In order to apply the lemma, we rewrite the probability on the left as
	\begin{align*}
		&\mathbb{P}^{0, t_2-t_1,\vec{c},\vec{a}}_{\mathrm{avoid},H^{\mathrm{RW}}}\left(\max_{1\leq i\leq k-2} \big[Q^1_i(t_2-t_{12}) - Q^1_{i+1}(t_2-t_{12})\big] \geq \delta\sqrt{2t_3}\right)\\
		= \; & \mathbb{P}^{0, t_2-t_1,\vec{c},\vec{a}}_{\mathrm{avoid}, H^{\mathrm{RW}}}\left(\max_{1\leq i\leq k-2} \big[Q^1_i(t_N(t_2-t_1)) - Q^1_{i+1}(t_N(t_2-t_1))\big] \geq \delta\sqrt{2t_3}\right),
	\end{align*}
	with the sequence $t_N$ given by
	\[
	t_N := \frac{t_2-t_{12}}{t_2-t_1}.
	\]
	We see from the definition of $t_{12}$ in \eqref{t12} that $t_N \to 1/2$ as $N\to\infty$. Recalling the definition of $E$, we can therefore apply Lemma \ref{DeltaSep} with $M_1 = U$, $M_2 = W$, we obtain $N_{67}$ and $\delta>0$ such that if $N \geq N_{67}$, then
	\[
	\pr^{0,t_2-t_1,\vec c, \vec a}_{\mathrm{avoid},H^{\mathrm{RW}}}\left(\min_{1\leq i\leq k-1} \big[Q^1_i(t_N(t_2-t_1))-Q_{i+1}^1(t_N(t_2-t_1))\big] < \delta\sqrt{t_2-t_1}\right) < 1 - \frac{3}{\sqrt{10}}.
	\]
	Together with the fact that $t_3/4 < t_2-t_1$, this implies that
	\begin{equation*}
		\pr^{-t_2,-t_1,\vec c, \vec a}_{\mathrm{avoid},H^{\mathrm{RW}}}\left(\min_{1\leq i\leq k-1} \big[Q_i^1(-t_{12})-Q_{i+1}^1(-t_{12})\big]<(\delta/4)\sqrt{2t_3}\right) < 1 - \frac{3}{\sqrt{10}}
	\end{equation*}
	for $N\geq N_{67}$. This proves \eqref{6.2step3left}. A similar argument gives us an $\eta>0$ such that
	\[
	\pr^{-t_2,-t_1,\vec c, \vec a}_{\mathrm{avoid},H^{\mathrm{RW}}}\left(\min_{1\leq i\leq k-1} \big[Q_i(-t_{12})-Q_{i+1}(-t_{12})\big]<(\eta/4)\sqrt{2t_3}\right)<1 - \frac{3}{\sqrt{10}}
	\]
	for $N\geq N_{67}$. Then putting $\epsilon = \min(\delta,\eta)/4$ and using \eqref{6.2step3ind}, we obtain \eqref{6.2step3 9/10} for $N\geq N_{67}$.\\
	
	{\bf \raggedleft Step 4.} In this step, we find $\bar{N}_{63}$ so that with $W$ fixed as in Step 2 and $\epsilon$ as in Step 3, we have
	\begin{equation}\label{6.2step4}
		\mathbb{P}^{-t_2,t_2,\vec{c},\vec{d},\infty,\ell_{\mathrm{b}}}_{\mathrm{avoid},H^{\mathrm{RW}};\tilde S}\big(D(\vec{c},\vec{d},W,\epsilon,t_1) \big) \geq \frac{g}{2}
	\end{equation}
	for $N\geq\bar{N}_{63}$, with $g = g(\epsilon,p,k)$ as in Lemma \ref{APg}. We will find $N_{68}$ so that
	\begin{equation}\label{6.2step4sep}
		\mathbb{P}^{-t_2,t_2,\vec{c},\vec{d},\infty,\ell_{\mathrm{b}}}_{\mathrm{avoid},H^{\mathrm{RW}};\tilde S}\left(D(\vec{c},\vec{d},W,\epsilon,t_1) \,\big|\,D(\vec c, \vec d, W, \epsilon,t_{12})\right) \geq g
	\end{equation}
	for $N\geq N_{68}$. Then \eqref{LemmaBP2Step3} implies \eqref{6.2step4} for $N\geq\bar{N}_{63} := \max(\bar{N}_2,N_{68})$.
	
	To prove \eqref{6.2step4sep} we first observe that we can write
	\begin{equation}\label{6.2step4disj}
		D(\vec c, \vec d, W, \epsilon,t_{12}) = \bigsqcup_{(\vec{a},\vec{b})\in J} G(\vec{a},\vec{b}).
	\end{equation}
	Here, for $\vec{a},\vec{b}\in\mathfrak{W}_{k-1}$, $G(\vec{a},\vec{b})$ is the event that $\mathfrak{Q}(-t_{12}) = \vec{a}$ and $\mathfrak{Q}(t_{12}) = \vec{b}$, and $J$ is the collection of $(\vec{a},\vec{b})$ satisfying
	\begin{enumerate}[label = (\arabic*)]
		
		\item $ \alpha(t_2-t_{12}) \leq \min(a_i - c_i,\, d_i - b_i) \leq \beta(t_2 - t_{12})$ and $2\alpha t_{12}\leq b_i-a_i \leq 2\beta t_{12}$ for $1\leq i\leq k-1$,
		
		\item $\min(a_{k-1} + pt_1,\, b_{k-1} - pt_1) \geq (M+1)\sqrt{2t_3}$,
		
		\item $\max(a_1 + pt_1,\, b_1 - pt_1) \leq W\sqrt{2t_3}$,
		
		\item $\min(a_i-a_{i+1}, \, b_i-b_{i+1}) \geq \epsilon\sqrt{2t_3}$ for $1\leq i\leq k-2$.
		
	\end{enumerate}
	We let $\tilde{J} = \{(\vec{a},\vec{b})\in J : \mathbb{P}^{-t_2,t_2,\vec{c},\vec{d},\infty,\ell_{\mathrm{b}}}_{\mathrm{avoid},H^{\mathrm{RW}};\tilde S}(G(\vec{a},\vec{b})) > 0\}$, and we take $N_{68}$ large enough so that $\tilde{J}\neq\varnothing$. We also let $\bar{D}(W,\epsilon,t_1)$ denote the set consisting of elements of $D(\vec{c},\vec{d},W,\epsilon,t_1)$ restricted to $\llbracket -t_{12},t_{12}\rrbracket$. Then for $(\vec{a},\vec{b})\in\tilde{J}$ we have
	\begin{equation*}\label{6.2step4cond}
		\begin{split}
			&\mathbb{P}^{-t_2,t_2,\vec{c},\vec{d},\infty,\ell_{\mathrm{b}}}_{\mathrm{avoid},H^{\mathrm{RW}};\tilde S}\left(D(\vec{c},\vec{d},W,\epsilon,t_1) \,\big|\,G(\vec{a},\vec{b})\right) = \mathbb{P}^{-t_{12},t_{12},\vec{a},\vec{b},\infty,\ell_{\mathrm{b}}}_{\mathrm{avoid},H^{\mathrm{RW}};\tilde S}\left(\bar D(W,\epsilon,t_1)\right) \\
			\geq \; &\mathbb{P}^{-t_{12},t_{12},\vec{a},\vec{b}}_{H^{\mathrm{RW}}}\left(\bar D(W,\epsilon,t_1)\cap \{L_1\geq\cdots\geq L_{k-1}\geq\ell_{\mathrm{b}}\}\right).
		\end{split}
	\end{equation*}
	We observe that the event in the second line occurs as long as each curve $L_i$ remains within a distance of $(\epsilon/2)\sqrt{2t_3}$ from the straight line segment connecting $a_i$ and $b_i$ on $[-t_{12},t_{12}]$, for $1\leq i\leq k-2$. By the argument in the proof of Lemma \ref{CurveSep}, we can enlarge $N_{68}$ so that the probability of this event is bounded below by the same quantity $g$ appearing in Lemma \ref{APg}, for $N\geq N_{68}$. Then using \eqref{6.2step4disj} and summing over $\tilde{J}$ implies \eqref{6.2step4sep}.\\
	
	{\bf \raggedleft Step 5.} In this last step, we complete the proof of the lemma, fixing the constants $h$ and $N_5$. Let $h_1,W,\epsilon,g$ be as in Steps 2, 3, and 4, define
	\[
	h := \frac{gh_1}{2},
	\]
	and let $N_5 = \max(\bar{N}_{60},\bar{N}_{61},\bar{N}_{62},\bar{N}_{63},N_7)$, with $N_7$ as in Lemma \ref{APt2}. In the following we assume that $N\geq N_5$. Let us denote by $H$ the event that
	\begin{enumerate}
		\item $W\sqrt{2t_3} \geq \tilde Q_1(-t_1) + p t_1 \geq \tilde Q_{k-1}(-t_1) + pt_1 \geq (M + 1) \sqrt{2t_3}$,
		\item $W\sqrt{2t_3} \geq \tilde Q_1(t_1) - p t_1 \geq \tilde Q_{k-1}(t_1) - pt_1 \geq (M + 1) \sqrt{2t_3}$,
		\item $\tilde Q_i(-t_1) - \tilde Q_{i+1}(-t_1) \geq \epsilon \sqrt{2t_3}$ and $\tilde Q_i(t_1) - \tilde Q_{i+1}(t_1)  \geq \epsilon \sqrt{2t_3}$ for $i = 1, \dots, k-2$.
	\end{enumerate}
	By \eqref{6.2step4} we have that if $(\vec{c},\vec{d})\in E$ and $N\geq N_5$, then
	\[
	\mathbb{P}_{\mathrm{avoid}, H^{\mathrm{RW}};\tilde{S}}^{-t_2, t_2, \vec{c}, \vec{d}, \infty, \ell_{\mathrm{b}}} ( H) \geq \frac{h}{h_1},
	\]
	Let $Y$ denote the event appearing in \eqref{eqnRT2}. Then we can write $Y = \bigsqcup_{(\vec{c},\vec{d})\in E} Y(\vec{c},\vec{d})$, where $Y(\vec{c},\vec{d})$ is the event that $\tilde{\mathfrak{Q}}(-t_2) = \vec{c}$, $\tilde{\mathfrak{Q}}(t_2) = \vec{d}$, and $E$ is defined in Step 2. If $\tilde{E} = \{(\vec{c},\vec{d})\in E : \mathbb{P}^{-t_3,t_3,\vec{a},\vec{b},\infty,\ell_{\mathrm{b}}}_{\mathrm{avoid},H^{\mathrm{RW}};S}(Y(\vec{c},\vec{d})) > 0\}$, we can assume that $N_5$ is large enough so that $\tilde{E}\neq\varnothing$. It follows from Lemma \ref{APt2} that 
	\[
	\mathbb{P}^{-t_3,t_3,\vec{a},\vec{b},\infty,\ell_{\mathrm{b}}}_{\mathrm{avoid},H^{\mathrm{RW}};S}(Y) \geq h_1.
	\] 
	We conclude that for all $N\geq N_5$,
	\begin{align*}
		\mathbb{P}^{-t_3,t_3,\vec{a},\vec{b},\infty,\ell_{\mathrm{b}}}_{\mathrm{avoid},H^{\mathrm{RW}};S}(H) &\geq \mathbb{P}^{-t_3,t_3,\vec{a},\vec{b},\infty,\ell_{\mathrm{b}}}_{\mathrm{avoid},H^{\mathrm{RW}};S}(H\cap Y)\\ 
		&= \sum_{(\vec{c},\vec{d})\in \tilde E} \mathbb{P}^{-t_3,t_3,\vec{a},\vec{b},\infty,\ell_{\mathrm{b}}}_{\mathrm{avoid},H^{\mathrm{RW}};S}(Y(\vec{c},\vec{d}))\cdot \mathbb{P}^{-t_3,t_3,\vec{a},\vec{b},\infty,\ell_{\mathrm{b}}}_{\mathrm{avoid},H^{\mathrm{RW}};S}(H\,|\,Y(\vec{c},\vec{d}))\\
		&= \sum_{(\vec{c},\vec{d})\in \tilde E} \mathbb{P}^{-t_3,t_3,\vec{a},\vec{b},\infty,\ell_{\mathrm{b}}}_{\mathrm{avoid},H^{\mathrm{RW}};S}(Y(\vec{c},\vec{d}))\cdot \mathbb{P}^{-t_2,t_2,\vec{c},\vec{d},\infty,\ell_{\mathrm{b}}}_{\mathrm{avoid},H^{\mathrm{RW}};\tilde S}(H)\\ 
		&\geq \frac{h}{h_1}\sum_{(\vec{c},\vec{d})\in \tilde E} \mathbb{P}^{-t_3,t_3,\vec{a},\vec{b},\infty,\ell_{\mathrm{b}}}_{\mathrm{avoid},H^{\mathrm{RW}};S}(Y(\vec{c},\vec{d})) = \frac{h}{h_1}\,\mathbb{P}^{-t_3,t_3,\vec{a},\vec{b},\infty,\ell_{\mathrm{b}}}_{\mathrm{avoid},H^{\mathrm{RW}};S}(Y) \geq h.
	\end{align*}
	Now Lemma \ref{APg} implies \eqref{APtecheq}, completing the proof.
\end{proof}

%% file: appendix.tex
%
\section*{Appendix}\label{Appendix}

Here we prove the monotone coupling Lemma \ref{MCL}. Our argument follows the approaches in \cite[Section 6]{CH14} and \cite[Section 6]{Wu19}.

\begin{proof}
	For simplicity, throughout the proof we will write $\Omega_{a,S}$ for $\Omega_{\mathrm{avoid}}(T_0,T_1,\vec{x},\vec{y},\infty,g^{\mathrm{b}};S)$ and $\Omega'_{a,S}$ for $\Omega_{\mathrm{avoid}}(T_0,T_1,\vec{x}\,',\vec{y}\,',\infty,g^{\mathrm{t}};S)$. We split the argument into two steps. In the first step, we will construct two Markov chains ordered with respect to one another which have as invariant measures $\mathbb{P}_{\mathrm{avoid}, H^{\mathrm{RW}}; S}^{T_0, T_1, \vec{x}, \vec{y}, \infty, g^{\mathrm{b}}}$ and $\mathbb{P}_{\mathrm{avoid}, H^{\mathrm{RW}}; S}^{T_0, T_1, \vec{x}\,', \vec{y}\,', \infty, g^{\mathrm{t}}}$, such that their laws converge weakly to these two measures. In the second step we will use the Skorohod representation theorem to construct $(\Omega,\mathcal{F},\mathbb{P})$, $\mathfrak{L}^{\mathrm{b}}$, and $\mathfrak{L}^{\mathrm{t}}$.\\
	
	\noindent\textbf{Step 1.} In this step we construct a Markov chain $(X^n,Y^n)_{n\geq 0}$ on $\Omega_{a,S}\times\Omega'_{a,S}$ with initial distribution given by
	\[
	X^0_i(t) = \min(x_i + \beta t - T_0, \, y_i), \quad Y^0_i(t) = \min(x_i' + \beta t - T_0, \, y_i'),
	\]
	with the following properties:
	\begin{enumerate}[label=(\arabic*)]
		\item $(X^n)_{n\geq 0}$ and $(Y^n)_{n\geq 0}$ are both Markov chains with respect to their own filtrations;
		
		\item $(X^n)$ is irreducible and aperiodic, with invariant measure $\mathbb{P}_{\mathrm{avoid}, H^{\mathrm{RW}}; S}^{T_0, T_1, \vec{x}, \vec{y}, \infty, g^{\mathrm{b}}}$;
		
		\item $(Y^n)$ is irreducible and aperiodic, with invariant measure $\mathbb{P}_{\mathrm{avoid}, H^{\mathrm{RW}}; S}^{T_0, T_1, \vec{x}\,', \vec{y}\,', \infty, g^{\mathrm{t}}}$;
		
		\item $X_i^n \leq Y_i^n$ on $\llbracket T_0, T_1\rrbracket$ for all $i\in\llbracket 1,k\rrbracket$.
		
	\end{enumerate}
	Together with \cite[Theorem 1.8.3]{Norris}, conditions (2) and (3) imply that
	\begin{equation}\label{MCweak} X^n \implies \mathbb{P}^{T_0,T_1,\vec{x},\vec{y},\infty,g^{\mathrm{b}}}_{\mathrm{avoid},H^{\mathrm{RW}};S} \quad \mathrm{and} \quad Y^n \implies \mathbb{P}_{\mathrm{avoid}, H^{\mathrm{RW}}; S}^{T_0, T_1, \vec{x}\,', \vec{y}\,', \infty, g^{\mathrm{t}}} \quad \mbox{as } n\to\infty.
	\end{equation}
	We first observe that $X^0$ is in fact in $\Omega_{a,S}$. To see this, simply observe that $x_i + \beta t - T_0 \leq x_{i+1} + \beta t - T_0$ and $y_i \leq y_{i+1}$ for each $i\in\llbracket 1,k-1\rrbracket$ and $t\in\llbracket T_0,T_1\rrbracket$ by assumption, so that $X^0_i(t) \leq X^0_{i+1}(t)$. Thus $X^0 \in \Omega_{\mathrm{avoid}}(T_0,T_1,\vec{x},\vec{y};S)$. Moreover, $X^0$ is maximal in this collection, in the sense that if $Z \in \Omega_{\mathrm{avoid}}(T_0,T_1,\vec{x},\vec{y};S)$, then $X^0_i \geq Z_i$ on $\llbracket T_0,T_1\rrbracket$ for all $i\in\llbracket 1,k\rrbracket$. By hypothesis there exists some $Z\in\Omega_{a,S}\subset \Omega_{\mathrm{avoid}}(T_0,T_1,\vec{x},\vec{y};S)$, and $X^0_k \geq Z_k \geq g^{\mathrm{b}}$ on $\llbracket T_0,T_1\rrbracket$. Thus $X^0$ is maximal in $\Omega_{a,S}$. Likewise, $Y^0$ is maximal in $\Omega'_{a,S}$.
	
	We now specify the dynamics of $(X^n,Y^n)$. At time $n$, we sample uniformly a triplet $(i,t,\zeta)\in \llbracket 1,k \rrbracket \times \llbracket T_0,T_1\rrbracket \times \{-1,1\}$. We also let $U^{(i,t,\zeta)}$ denote an independent uniform random variable in $[0,1]$. We update $X^n$ to $X^{n+1}$ as follows. Define a candidate $\tilde{X}^n_i$ for $X^{n+1}_i$ via $\tilde{X}^n_i(t) = X^n_i(t) + \zeta$, and $\tilde{X}^n_j(s) = X^n_j(s)$ if $j\neq i$ or $s\neq t$. If $\tilde{X}^n \in \Omega_{a,S}$ and
	\begin{equation}\label{MCupdate}
		R^{(i,t,\zeta)}_{X^n} := \frac{\mathbb{P}^{T_0,T_1,\vec{x},\vec{y}}_{H^{\mathrm{RW}}}(\tilde{X}^n)}{\mathbb{P}^{T_0,T_1,\vec{x},\vec{y}}_{H^{\mathrm{RW}}}(X^n)} \geq U^{(i,t,\zeta)},
	\end{equation}
	we set $X^{n+1} = \tilde{X}^n$. Otherwise we leave $X^{n+1} = X^n$. We update $Y^n$ according to the analogous rule, replacing $R^{(i,t,\zeta)}_{X^n}$ with the corresponding quantity $R^{(i,t,\zeta)}_{Y^n}$. 
	
	By construction, $(X^n,Y^n)\in \Omega_{a,S}\times\Omega'_{a,S}$ for all $n$. It is easy to see that $(X^n,Y^n)$ is a Markov chain, since the value of $(X^{n+1},Y^{n+1})$ depends only on the prior state $(X^n,Y^n)$. In particular $X^{n+1}$ depends only on $X^n$, so that $(X^n)$ is Markov in its own filtration, and likewise for $(Y^n)$. This proves condition (1) above.
	
	We now argue that $(X^n)$ is irreducible; a similar argument works for $(Y^n)$. Fix any $Z\in \Omega_{a,S}$. We must argue that the state $Z$ can be reached from $X^0$ after some number of steps with positive probability. Due to the maximality of $X^0$ we need only move the paths downward, and if we do so starting with the bottom path, then there is no danger of the paths crossing on $S$ at any step. To ensure that $X^n_k = Z$, we initialize $t = T_0 + 1$ and perform the following procedure. If $X^n_k(t) = Z(t)$, we increment $t$ by 1. If $X^n_k(t) > Z(t)$, we sample the triplet $(k,t,-1)$ and perform the update, then increment $t$ by 1. Once $t$ reaches $T_1$, we reset $t = T_0 + 1$. After finitely many iterations, $X^n_k$ will agree with $Z_i$ on all of $\llbracket T_0,T_1\rrbracket$. We then repeat this procedure for $X^n_i$ and $Z_i$, with $i$ descending. Each of these samples has positive probability, and the process terminates in finite time, so the probability of reaching $Z$ is positive.
	
	To see that $(X^n)$ and $(Y^n)$ are aperiodic, simply observe that sampling a triplet $(k,T_0,\zeta)$ leaves both $X^n$ and $Y^n$ unchanged.
	
	We now complete the proof of (2) and (3) by arguing that $(X^n)$ has invariant measure $\mathbb{P}^{T_0,T_1,\vec{x},\vec{y},\infty,g^{\mathrm{b}}}_{\mathrm{avoid},H^{\mathrm{RW}};S}$; a similar argument works for $(Y^n)$. Fix any $\omega\in\Omega_{a,S}$. Observe that if $\tau\in\Omega_{a,S}$, then $\mathbb{P}(X^{n+1} = \omega \,|\, X^n = \tau) \neq 0$ if and only if $\mathbb{P}(X^{n+1} = \tau \,|\, X^n = \omega) \neq 0$, since $\omega$ is reachable from $\tau$ if and only if $\omega$ and $\tau$ differ only at one point $(i,t)$ by 1, and this in turn is equivalent to $\tau$ being reachable from $\omega$. In this case, we note that
	\begin{equation}\label{MCfrac}
		\frac{\mathbb{P}(X^{n+1} = \omega \,|\, X^n = \tau)}{\mathbb{P}(X^{n+1} = \tau \,|\, X^n = \omega)} = \frac{\mathbb{P}^{T_0,T_1,\vec{x},\vec{y}}_{H^{\mathrm{RW}}}(\omega)}{\mathbb{P}^{T_0,T_1,\vec{x},\vec{y}}_{H^{\mathrm{RW}}}(\tau)}.
	\end{equation}
	This follows from the update rule \eqref{MCupdate} by observing that either the ratio on the right or its reciprocal is at least 1, hence necessarily at least as large as a uniform random variable on $[0,1]$. It follows from \eqref{MCfrac} that
	\begin{align*}
		&\sum_{\tau\in\Omega_{a,S}} \mathbb{P}^{T_0,T_1,\vec{x},\vec{y},\infty,g^{\mathrm{b}}}_{\mathrm{avoid},H^{\mathrm{RW}};S}(\tau)\cdot\mathbb{P}(X^{n+1} = \omega \,|\, X^n = \tau)\\ 
		= \; & \sum_{\tau\in\Omega_{a,S}} \mathbb{P}^{T_0,T_1,\vec{x},\vec{y},\infty,g^{\mathrm{b}}}_{\mathrm{avoid},H^{\mathrm{RW}};S}(\tau)\cdot \frac{\mathbb{P}^{T_0,T_1,\vec{x},\vec{y}}_{H^{\mathrm{RW}}}(\omega)}{\mathbb{P}^{T_0,T_1,\vec{x},\vec{y}}_{H^{\mathrm{RW}}}(\tau)}\cdot\mathbb{P}(X^{n+1} = \tau \,|\, X^n = \omega)\\
		= \; & \mathbb{P}^{T_0,T_1,\vec{x},\vec{y},\infty,g^{\mathrm{b}}}_{\mathrm{avoid},H^{\mathrm{RW}};S}(\omega) \sum_{\tau\in\Omega_{a,S}}\mathbb{P}(X^{n+1} = \tau \,|\, X^n = \omega) = \mathbb{P}^{T_0,T_1,\vec{x},\vec{y},\infty,g^{\mathrm{b}}}_{\mathrm{avoid},H^{\mathrm{RW}};S}(\omega).
	\end{align*}
	This proves invariance. (In fact, \eqref{MCfrac} implies reversibility.)
	
	Lastly, we argue that $X_i^n \leq Y_i^n$ on $\llbracket T_0,T_1\rrbracket$ for all $i\in\llbracket 1,k\rrbracket$ and $n\geq 0$. It is clear that $X_i^0 \leq Y_i^0$ for all $i$. Since the update rule cannot move curves of $X^n$ and $Y^n$ in opposite directions, the only way the ordering can be violated at time $n+1$ is if $X_i^n(t) = Y_i^n(t) = z$ for some $i\in\llbracket 1,k\rrbracket$ and $t\in\llbracket T_0+1,T_1-1\rrbracket$. The first case to consider is if $(i,t,+1)$ is sampled and 
	\begin{equation}\label{MCviol}
		R^{(i,t,+1)}_{Y^n} < U^{(i,t,+1)} \leq R^{(i,t,+1)}_{X^n}.
	\end{equation} 
	We argue now that \eqref{MCviol} is impossible. We have
	\begin{align*}
		R^{(i,t,+1)}_{X^n} &= \frac{\exp\left[-H^{\mathrm{RW}}(z+1-X^n_i(t-1)) - H^{\mathrm{RW}}(X^n_i(t+1)-z-1)\right]}{\exp\left[-H^{\mathrm{RW}}(z-X^n_i(t-1)) - H^{\mathrm{RW}}(X^n_i(t+1)-z)\right]}\\
		&= \frac{\exp\left[H^{\mathrm{RW}}(X_i^n(t+1)-z) - H^{\mathrm{RW}}(X_i^n(t+1)-z-1)\right]}{\exp\left[H^{\mathrm{RW}}(z-X_i^n(t-1)+1) - H^{\mathrm{RW}}(z-X_i^n(t-1))\right]},
	\end{align*}
	and similarly
	\begin{equation*}
		R^{(i,t,+1)}_{Y^n} = \frac{\exp\left[H^{\mathrm{RW}}(Y_i^n(t+1)-z) - H^{\mathrm{RW}}(Y_i^n(t+1)-z-1)\right]}{\exp\left[H^{\mathrm{RW}}(z-Y_i^n(t-1)+1) - H^{\mathrm{RW}}(z-Y_i^n(t-1))\right]}.
	\end{equation*}
	Since $X_i^n(t-1) \leq Y_i^n(t-1)$ and $X_i^n(t+1) \leq Y_i^n(t+1)$, the convexity of $H^{\mathrm{RW}}$ implies that
	\begin{align*}
		H^{\mathrm{RW}}(Y_i^n(t+1)-z) - H^{\mathrm{RW}}(Y_i^n(t+1)-z-1) &\leq H^{\mathrm{RW}}(X_i^n(t+1)-z) - H^{\mathrm{RW}}(X_i^n(t+1)-z-1),\\
		H^{\mathrm{RW}}(z-Y_i^n(t-1)+1) - H^{\mathrm{RW}}(z-Y_i^n(t-1)) &\geq H^{\mathrm{RW}}(z-X_i^n(t-1)+1) - H^{\mathrm{RW}}(z-X_i^n(t-1)).
	\end{align*}
	It follows that $R^{(i,t,+1)}_{Y^n} \geq R^{(i,t,+1)}_{X^n}$, so \eqref{MCviol} is impossible. The other case in which the ordering is violated is if $(i,t,-1)$ is sampled and $R^{(i,t,-1)}_{X^n} < U^{(i,t,-1)} \leq R^{(i,t,-1)}_{Y^n}$. A similar argument to the above shows that $R^{(i,t,-1)}_{Y^n} \leq R^{(i,t,-1)}_{X^n}$, so the ordering holds in this case as well. This completes the proof of (4).\\
	
	\noindent\textbf{Step 2.} In this step, we construct the desired probability space $(\Omega,\mathcal{F},\mathbb{P})$ and random variables $\mathfrak{L}^{\mathrm{b}}$, $\mathfrak{L}^{\mathrm{t}}$ using the Markov chains $(X^n)$ and $(Y^n)$ constructed in Step 1. It follows from \eqref{MCweak} that in particular $(X^n)$ and $(Y^n)$ are tight sequences, so $(X^n,Y^n)_{n\geq 0}$ is tight as well. By Prohorov's theorem, $(X^n,Y^n)$ is relatively compact. Let $(n_m)$ be a sequence such that $(X^{n_m},Y^{n_m})$ converges weakly. Since $C(\llbracket 1,k\rrbracket\times\llbracket T_0,T_1\rrbracket)$ is separable by Lemma \ref{Polish}, the Skorohod representation theorem \cite[Theorem 6.7]{Billing} implies that there is a probability space $(\Omega,\mathcal{F},\mathbb{P})$ supporting random variables $\mathfrak{X}^n$, $\mathfrak{Y}^n$ and $\mathfrak{L}^{\mathrm{b}},\mathfrak{L}^{\mathrm{t}}$ taking values in $\Omega_{a,S},\Omega_{a,S}'$ respectively, such that
	\begin{enumerate}[label=(\arabic*)]
		
		\item The law of $(\mathfrak{X}^n,\mathfrak{Y}^n)$ under $\mathbb{P}$ is the same as that of $(X^n,Y^n)$,
		
		\item $\mathfrak{X}^n(\omega) \longrightarrow \mathfrak{L}^{\mathrm{b}}(\omega)$ for all $\omega\in\Omega$,
		
		\item $\mathfrak{Y}^n(\omega) \longrightarrow \mathfrak{L}^{\mathrm{t}}(\omega)$ for all $\omega\in\Omega$.
		
	\end{enumerate}
	
	In particular, (1) implies that $\mathfrak{X}^{n_m}$ has the same law as $X^{n_m}$, which converges weakly to $\mathbb{P}_{\mathrm{avoid},H^{\mathrm{RW}};S}^{T_0,T_1,\vec{x},\vec{y},\infty,g^{\mathrm{b}}}$. It follows from (2) that $\mathfrak{L}^{\mathrm{b}}$ has law $\mathbb{P}_{\mathrm{avoid},H^{\mathrm{RW}};S}^{T_0,T_1,\vec{x},\vec{y},\infty,g^{\mathrm{b}}}$. Similarly, $\mathfrak{L}^{\mathrm{t}}$ has law $\mathbb{P}_{\mathrm{avoid},H^{\mathrm{RW}};S}^{T_0,T_1,\vec{x}\,',\vec{y}\,',\infty,g^{\mathrm{t}}}$. Moreover, condition (4) in Step 1 implies that $\mathfrak{X}^n_i \leq \mathfrak{Y}^n_i$, $\mathbb{P}$-a.s., so $\mathfrak{L}^{\mathrm{b}}_i \leq \mathfrak{L}^{\mathrm{t}}_i$ for $i\in\llbracket 1,k\rrbracket$, $\mathbb{P}$-a.s, as desired.
	
\end{proof}

\section*{Acknowledgments}

The author would like to thank Ivan Corwin and Evgeni Dimitrov for suggesting this problem and providing numerous helpful comments on drafts of this paper. The author is also grateful to Evgeni Dimitrov for many insightful conversations regarding this work, and for guiding the project on Bernoulli line ensembles which led to this paper. We thank the anonymous referees for providing several helpful comments that improved the paper.